\documentclass[a4paper,11pt]{amsart}
\usepackage{amsmath}
\usepackage{cases}
\usepackage{amsfonts}
\usepackage[colorlinks,linkcolor=blue,citecolor=blue]{hyperref}
\usepackage{latexsym, amssymb, amsmath, amsthm, bbm}
\usepackage[all]{xy}
\usepackage{pgfplots}

\DeclareSymbolFont{EulerExtension}{U}{euex}{m}{n}
\DeclareMathSymbol{\euintop}{\mathop} {EulerExtension}{"52}
\DeclareMathSymbol{\euointop}{\mathop} {EulerExtension}{"48}

\allowdisplaybreaks[4]

\def \id{\operatorname{Id}}
\def \ker{\operatorname{Ker}}

\def \Z{\mathbb{Z}}

\def \k{\mathbbm{k}}

\def \Id{\operatorname{Id}}

\def \Aut{\operatorname{Aut}}
\def \id{\operatorname{Id}}
\def \ker{\operatorname{Ker}}

\def \Z{\mathbb{Z}}

\numberwithin{equation}{section}

\newtheorem{theorem}{Theorem}[section]
\newtheorem{lemma}[theorem]{Lemma}
\newtheorem{proposition}[theorem]{Proposition}
\newtheorem{corollary}[theorem]{Corollary}
\newtheorem{definition}[theorem]{Definition}
\newtheorem{example}[theorem]{Example}
\newtheorem{remark}[theorem]{Remark}

\newtheorem*{notation}{Notation}

\begin{document}
\title{On the quasitriangular structures of abelian extensions of $\mathbb{Z}_{2}$}
\thanks{$^\dag$Supported by NSFC 11722016.}


\author{Kun Zhou and Gongxiang Liu}
\address{Department of Mathematics, Nanjing University, Nanjing 210093, China} \email{xzkdh4712@gmail.com gxliu@nju.edu.cn}
\date{}
\maketitle
\begin{abstract}  The aim of this paper is to study quasitriangular structures on a class of semisimple Hopf algebras $\Bbbk^G\#_{\sigma,\tau}\Bbbk \mathbb{Z}_{2}$ constructed through abelian extensions of $\k\Z_2$ by $\Bbbk^G$ for an abelian group $G.$ We prove that there are only two forms of them. Using such description together with some other techniques, we get a complete list of all universal $\mathcal{R}$-matrices on Hopf algebras $H_{2n^2}$, $A_{2n^2,t}$ and $K(8n,\sigma,\tau)$ (see Section 2 for the definition of these Hopf algebras). Then we find a simple criterion for a $K(8n,\sigma,\tau)$ to be a minimal quasitriangular Hopf algebra. As a product, some minimal quisitriangular semisimple Hopf algebras are found.
\end{abstract}

\section{Introduction}
Throughout the paper we work over an algebraically closed field $\Bbbk$ of characteristic 0. Our original motivation comes from the following well-known fact: There is one to one correspondence between the conjugacy classes irreducible depth 2 inclusions of hyperfinite II$_1$ factors with finite index and the isomorphism classes of all Kac algebras, that is, finite dimensional $\mathbb{C}^\ast$-Hopf algebras. Due to the undoubted importance of quasitriangular structure, a natural question is: if a Kac algebra is quasitriangular, then what can say about the corresponding subfactor? Soon, we realized that it seems quite hard to determine when a Hopf algebra is quasitriangular. For example, even for the very simple case, say, the eight-dimensional Kac-Paljutkin algebra $K_8$, all quasitriangular structure on it weren't until 2010 that Wakui figured them out \cite{W}.

In this paper and subsequent works, we try to determine possible quasitriangular structures on a class of semisimple Hopf algebras arising from exact factorizations of
finite groups. The well-known eight-dimensional Kac-Paljutkin algebra $K_8$ is a very special case of them.  The idea of constructing these semisimple Hopf algebras can be tracked back to G. Kac \cite{Kac}: Suppose that $L= G\Gamma$ is an exact factorization of the finite group $L$, into its subgroups $G$ and $\Gamma$, such that $G\cap \Gamma = 1.$ Associated to this exact factorization and appropriate cohomology data $\sigma$ and
$\tau$, there is a semisimple bicrossed product Hopf algebra $H =\Bbbk^G\#_{\sigma,\tau}\Bbbk \Gamma$ (see Section 2 for the definition and \cite{M3,M5,M6} for details and generalizations). The question of existence of quasitriangular structures on $\Bbbk^G\#_{\sigma,\tau}\Bbbk \Gamma$ has been considered before. In 2011, S. Natale \cite{Na} proved that if $L$ is almost simple, then the extension admits no quasitriangular structure.

But for our purpose, we want to find more concrete quasitriangular structures rather than absence of quasitriangular structures. So comparing the Natale's viewpoint, we consider the other extreme case: the almost commutative case. That is, we assume that both $G$ and $\Gamma$ are commutative groups. As the start point, we further assume that $\Gamma$ is just the $\Z_2$ in this paper.


At first, we find that there is a dichotomy on the forms of the quasitriangular structures of $\Bbbk^G\#_{\sigma,\tau}\Bbbk \mathbb{Z}_{2}$. For convenience, we call one form trivial and the other form nontrivial. In principle, the trivial form corresponds to the bicharacters and thus is not very complicated. The difficult point of the this paper is to determine the nontrivial forms. To do that, we get some necessary conditions for the existence of nontrivial forms.  To state the applications of such observations, we give three classes of Hopf algebras which are denoted by $H_{2n^2}$, $A_{2n^2,t}$ and $K(8n,\sigma,\tau)$ respectively. We need point out that the first two classes of Hopf algebras were studied by some authors \cite{D,M2} before. As the main conclusion of this paper,  all universal $\mathcal{R}$-matrices on Hopf algebras $H_{2n^2}$, $A_{2n^2,t}$ and $K(8n,\sigma,\tau)$ are given explicitly.  In addition,  we have identified which are the minimal Hopf algebras among $K(8n,\sigma,\tau)$ (we show that there is almost no minimal quasitriangular structures on $H_{2n^2}$ and $A_{2n^2,t}$). When $K(8n,\sigma,\tau)$ is a minimal quasitriangular Hopf algebra, we explicitly write out all minimal quasitriangular structures on it. As an application of the above conclusions, we find a class of minimal quasitriangular Hopf algebras which are denoted by $K(8n,\zeta)$ ($n\geq 4$ is even) in this paper.
%

This paper is organized as follows. In Section 2, we recall the definition of a  Hopf algebra $\Bbbk^G\#_{\sigma,\tau}\Bbbk \mathbb{Z}_{2}$ and give some examples of them. In Section 3, we show that there are only two possible forms of quasitriangular structures on $\Bbbk^G\#_{\sigma,\tau}\Bbbk \mathbb{Z}_{2}$ and give some necessary conditions for $\Bbbk^G\#_{\sigma,\tau}\Bbbk \mathbb{Z}_{2}$ preserving non-trivial quasitriangular structures. Using our necessary conditions, we can easily get all universal $\mathcal{R}$-matrices on Hopf algebras $H_{2n^2}$, $A_{2n^2,t}$. Section 4 is devoting to figure out all universal $\mathcal{R}$-matrices on Hopf algebras $K(8n,\sigma,\tau)$. This section occupies most of parts of this paper due to the situation becoming complicated than before. As the results, we not only get a complete list of all quasitriangular structures on $K(8n,\sigma,\tau)$ but also find a simple criterion to determine which one is minimal. Moreover, the new class of minimal quasitriangular Hopf algebras are also given in this section by using this criterion.

 All Hopf algebras in this paper are finite dimensional. For the symbol $\delta$ in Section 2, we mean the classical Kronecker's symbol.

\section{Abelian extensions of $\mathbb{Z}_{2}$ and examples}
In this section, we recall the definition of $\Bbbk^G\#_{\sigma,\tau}\Bbbk \mathbb{Z}_{2}$, and then we give some examples of $\Bbbk^G\#_{\sigma,\tau}\Bbbk \mathbb{Z}_{2}$ for guiding our further research.
\subsection{The definition of $\Bbbk^G\#_{\sigma,\tau}\Bbbk \mathbb{Z}_{2}$}
\begin{definition}\label{def2.1.1}
A short exact sequence of Hopf algebras is a sequence of Hopf algebras
and Hopf algebra maps
\begin{equation}\label{ext}
\;\; K\xrightarrow{\iota} H \xrightarrow{\pi} A
\end{equation}
such that
\begin{itemize}
  \item[(i)] $\iota$ is injective,
  \item[(ii)]  $\pi$ is surjective,
  \item[(iii)] $\ker(\pi)= HK^+$, $K^+$ is the kernel of the counit of $K$.
\end{itemize}
\end{definition}
In this situation it is said that $H$ is an extension of $A$ by $K$ \cite[Definiton 1.4]{M3}. An extension \eqref{ext} above such that $K$ is commutative and $A$ is cocommutative is called abelian.
In this paper, we only study the following special abelian extensions
\begin{equation*}
\;\; \Bbbk^G\xrightarrow{\iota} A \xrightarrow{\pi} \Bbbk \mathbb{Z}_{2},
\end{equation*}
where $G$ is a finite abelian group. Abelian extensions were classified by Masuoka
(see \cite[Proposition 1.5]{M3}), and the above $A$ can be expressed as $\Bbbk^G\#_{\sigma,\tau}\Bbbk \mathbb{Z}_{2}$ which is defined as follows.

Let $\mathbb{Z}_2=\{1,x\}$ be the cyclic group of order 2 and let $G$ be a finite group. To give the description of $\Bbbk^G\#_{\sigma,\tau}\Bbbk \mathbb{Z}_{2}$, we need the following data
\begin{itemize}
\item[(i)] $\triangleleft :\mathbb{Z}_2 \rightarrow \Aut(G)$ is an injective group homomorphism.
\item[(ii)] $\sigma:G\rightarrow \Bbbk^\times$ is a map such that $\sigma(g\triangleleft x)=\sigma(g)$ for $g \in G$ and $\sigma(1)=1$.
\item[(iii)] $\tau:G\times G \rightarrow \Bbbk^\times$ is a unital 2-cocycle and satisfies that $\sigma(gh)\sigma(g)^{-1}\sigma(h)^{-1}=\tau(g,h)\tau(g\triangleleft x,h\triangleleft x)$\ for $g,h \in G$.
\end{itemize}
We need remark that here (i) is not necessary for the definition of $\Bbbk^G\#_{\sigma,\tau}\Bbbk \mathbb{Z}_{2}$ and our aim for adding this additional requirement is just to avoid making a commutative algebra (in such case all quasitriangular structures are given by bicharacters and thus is known).
\begin{definition}\cite[Section 2.2]{AA}\label{def2.1.2}
As an algebra, the Hopf algebra $\Bbbk^G\#_{\sigma,\tau}\Bbbk \mathbb{Z}_{2}$ is generated by $\{ e_{g},x \}_{g \in G}$  satisfying
 \begin{equation*}
 e_{g}e_{h}=\delta_{g,h}e_{g},\ xe_{g}=e_{g\triangleleft x}x,\ x^2=\sum\limits_{g \in G}\sigma(g)e_{g}, \;\;\;\;g,h\in G.
 \end{equation*}
 The coproduct, counit and antipode are given by
 \begin{align*}
 &\Delta (e_{g})=\sum_{ h,k \in G,\ hk=g} e_{h}\otimes e_{k},\ \Delta(x)=[\sum\limits_{g,h \in G}\tau(g,h)e_{g}\otimes e_{h}](x\otimes x),\\
  &\epsilon(x)=1,\ \epsilon(e_{g})=\delta_{g,1}1,\\
   &\mathcal{S}(x)=\sum_{g\in G}\sigma(g)^{-1}\tau(g,g^{-1})^{-1}e_{g\triangleleft x}x,\ \mathcal{S}(e_g)=e_{g^{-1}},\;\;g\in G.
 \end{align*}
\end{definition}
\subsection{Examples} The following are some examples of $\Bbbk^G\#_{\sigma,\tau}\Bbbk \mathbb{Z}_2$ and we will discuss them in next sections.
\begin{example}\label{ex2.1.5}
\emph{Let $n \in \mathbb{N}$ and assume that $w$ is a primitive $n$th root of 1 in $\Bbbk$. Then \emph{the generalized Kac-Paljutkin algebra} $H_{2n^2}$ \cite[Section 2.2]{D} belongs to $\Bbbk^G\#_{\sigma,\tau}\Bbbk \mathbb{Z}_2$. By definition, the data $(G,\triangleleft,\sigma,\tau)$ of $H_{2n^2}$ is given by the following way
\begin{itemize}
             \item[(i)] $G=\mathbb{Z}_{n}\times \mathbb{Z}_{n}=\langle a,b|a^n=b^n=1,ab=ba\rangle$ and $a\triangleleft x=b,b\triangleleft x=a$.
              \item[(ii)] $\sigma(a^i b^j)=w^{ij}$ for $1 \leq i,j\leq n$.
              \item[(iii)] $\tau(a^i b^j ,a^k b^l)=(w)^{jk}$ for $1 \leq i,j,k,l\leq n$.
\end{itemize}}
\emph{Among of them, if we take $n=2$ then the resulting Hopf algebra is just the well-known Kac-Paljutkin 8-dimensional algebra $K_8.$ That's the reason why we call $H_{2n^2}$ the generalized Kac-Paljutkin algebra.}
\end{example}

We give another kind of generalization of $K_8$, which is defined as follows and we denote it by $K(8n,\sigma,\tau)$.

\begin{example}\label{def2.1.3}
\emph{ Let $n$ be a natural number. A Hopf algebra $H$ belonging to $\Bbbk^G\#_{\sigma,\tau}\Bbbk \mathbb{Z}_{2}$ is denoted by $K(8n,\sigma,\tau)$ if the data $(G,\triangleleft,\sigma,\tau)$ of $H$ satisfies}
  \begin{itemize}
  \item[(i)] $G=\mathbb{Z}_{2n}\times \mathbb{Z}_{2}=\langle a,b|a^{2n}=b^2=1,ab=ba\rangle;$
       \item[(ii)] $a\triangleleft x=ab,b\triangleleft x=b$.
  \end{itemize}
\end{example}
If we take $n=1$ and let $\sigma(a^ib^j)=(-1)^{(i-j)j}$ and $\tau(a^ib^j,a^kb^l)=(-1)^{j(k-l)}$ for $1\leq i,j,k,l \leq 2$, then we can easily check that the resulting 8-dimensional Hopf algebra is just the Kac-Paljutkin $8$-dimensional algebra $K_8$. Therefore, we give another kind of generalization of $K_8.$  Among of these Hopf algebras $K(8n,\sigma,\tau)$, the following class of Hopf algebras are particularly interesting for us since at least they will provide us a number of minimal quasitriangular Hopf algebras.
\begin{example}\label{ex2.1.1}
\emph{Let $n \in \mathbb{N}$ such that $n\geq 2$ and assume that $\zeta$ is a primitive $2n$th root of 1. A Hopf algebra $H$ belonging to $\Bbbk^G\#_{\sigma,\tau}\Bbbk \mathbb{Z}_{2}$ is denoted by $K(8n,\zeta)$ if the data $(G,\triangleleft,\sigma,\tau)$ of $H$ satisfies the following conditions
\begin{itemize}
             \item[(i)] $G=\mathbb{Z}_{2n}\times \mathbb{Z}_{2}=\langle a,b|a^{2n}=b^2=1,ab=ba\rangle$ and $a\triangleleft x=ab,b\triangleleft x=b$.
              \item[(ii)] $\sigma(a^i b^j)=(-1)^{\frac{i(i-1)}{2}}\zeta^i$ for $1 \leq i\leq 2n$ and $1 \leq j\leq 2$.
              \item[(iii)] $\tau(a^i b^j ,a^k b^l)=(-1)^{jk}$ for $1 \leq i,k\leq 2n$ and $1 \leq j,l\leq 2$.
\end{itemize}}
\end{example}

This also recover some familiar examples of semisimple Hopf algebras. For example, $K(16,\zeta)$ is the 16 dimensional Hopf algebra $H_{c:\sigma_1}$ in \cite[Section 3.1]{K} \label{ex2.1.3}.

At last, we recall another kind of semisimple Hopf algebras for our research.
\begin{example}\label{ex2.1.6}
\emph{Take an odd number $n$ and let $t$ be a primitive $n$th root of 1 in $\Bbbk$, then the Hopf algebras $A_{2n^2,t}$ were defined in \cite[Definition 1.2]{M2}. By definition, they belong to $\Bbbk^G\#_{\sigma,\tau}\Bbbk \mathbb{Z}_2$ and the data $(G,\triangleleft,\sigma,\tau)$ of $A_{2n^2,t}$ can be described as follows (see \cite[Section 2.3.4]{AA})
\begin{itemize}
             \item[(i)] $G=\mathbb{Z}_n \times \mathbb{Z}_n=\langle a,b|a^n=b^n=1,ab=ba\rangle$ and $a\triangleleft x=a^{-1},\; b\triangleleft x=b$.
              \item[(ii)] $\sigma(g)=1$, $g \in G$.
              \item[(iii)] $\tau(a^i b^j ,a^k b^l)=t^{jk}$ for $1 \leq i,j,k,l\leq n$.
            \end{itemize}}
\end{example}

\section{Forms of universal $\mathcal{R}$-matrices}\label{sec2.2}
In this section, we will prove that for $\Bbbk^G\#_{\sigma,\tau}\Bbbk \mathbb{Z}_2$ there are at most two forms of universal $\mathcal{R}$-matrices. Based on this observation, we determine all possible quasitriangular structures on generalized Kac-Paljutkin algebras $H_{2n^2}$(see Example \ref{ex2.1.5}) and semisimple Hopf algebras $A_{2n^2,t}$ (see Example \ref{ex2.1.6}).
\subsection{Forms of Universal $\mathcal{R}$-matrices}
Recall that a quasitriangular Hopf algebra is a pair $(H, R)$ where $H$ is a Hopf algebra and $R=\sum R^{(1)} \otimes R^{(2)}$ is an invertible element in $H\otimes H$ such that
\begin{equation*}
 (\Delta \otimes \id)(R)=R_{13}R_{23},\; (\id \otimes \Delta)(R)=R_{13}R_{12},\;\Delta^{op}(h)R=R\Delta(h),
 \end{equation*}
for $h\in H$. Here by definition $R_{12}= \sum R^{(1)} \otimes R^{(2)}\otimes 1,\; R_{13}= \sum R^{(1)}\otimes 1 \otimes R^{(2)}$ and
$R_{23}=\sum 1 \otimes R^{(1)}\otimes R^{(2)}$. The element $R$ is called a universal $\mathcal{R}$-matrix of $H$ or a quasitriangular structure on $H$.

To find the possible forms of universal $\mathcal{R}$-matrices, we need the following Lemmas \ref{lem2.2.1}-\ref{lem2.2.3} which will help us to check the braiding conditions. The first lemma is well-known.
\begin{lemma}\cite[Proposition 12.2.11]{R}\label{lem2.2.1}
Let $H$ be a Hopf algebra and $R \in H\otimes H$. For $f \in H^*$, if we denote $l(f):=(f \otimes \id)(R)$ and $r(f):=(\id \otimes f)(R)$, then the following statements are equivalent
\begin{itemize}
\item[(i)] $(\Delta \otimes \id)(R)=R_{13}R_{23}$ and $(\id \otimes \Delta)(R)=R_{13}R_{12}$.
\item[(ii)] $l(f_1)l(f_2)=l(f_1f_2)$ and $r(f_1)r(f_2)=r(f_2f_1) $ for $f_1,f_2 \in H^*$.
\end{itemize}
\end{lemma}

\begin{lemma}\label{lem2.2.2}
 Denote the dual basis of $\{e_g,e_gx\}_{g\in G}$ by $\{E_g,X_g\}_{g\in G}$, that is,
$E_g(e_h)=\delta_{g,h},\;E_g(e_hx)=0,\;X_g(e_h)=0,\;X_g(e_hx)=\delta_{g,h}$ for $g,h\in G$. Then the following equations hold in the dual Hopf algebra $(\Bbbk^G\#_{\sigma,\tau}\Bbbk \mathbb{Z}_2)^{*}$:
\begin{equation*}
E_gE_h=E_{gh},\ E_gX_h=X_hE_g=0,\ X_gX_h=\tau(g,h)X_{gh}, \;\; g,h\in G.
\end{equation*}
\end{lemma}
\begin{proof} Direct computations show that
$$E_gE_h(e_k)=E_{gh}(e_k)=\delta_{gh,k},\;\;E_gE_h(e_kx)=E_{gh}(e_kx)=0$$
 for $g,h,k\in G$. As a result, we have $E_gE_h=E_{gh}$. Similarly, one can get the last two equations.
\end{proof}

Let $\Bbbk^G\#_{\sigma,\tau}\Bbbk \mathbb{Z}_2$ as before. We need following two notions which will be used freely throughout this paper. Let
      $$S:=\{g\;|\;g\in G,\; g\triangleleft x=g\},\;\;T:=\{g\;|\; g\in G,\; g\triangleleft x\neq g\}.$$
A very basic observation is:

\begin{lemma}\label{lem2.2.3}
We have $S \subseteq TT$ where $TT=\{gh\;|\;g,h\in T\}$.
\end{lemma}

\begin{proof}
Clearly, for $s\in S, t\in T$, we have $ts\in T$. From the Definition \ref{def2.1.2} we know that the action $\triangleleft$ is injective, therefore $T \neq \emptyset$. Let $t\in T$ and it is obvious that $S=t(t^{-1}S)$ and hence $S \subseteq TT$.
\end{proof}

With the help of $S,T$, we find that

\begin{lemma}\label{lem2.2.4}
Let $w^1:G\times G\rightarrow \Bbbk$, $w^2:G\times G\rightarrow \Bbbk$, $w^3:G\times G\rightarrow \Bbbk$, $w^4:G\times G\rightarrow \Bbbk$ be four maps and define $R$ as follows
\begin{align*}
R&\colon=\sum\limits_{g,h \in G}w^1(g,h)e_{g} \otimes e_{h}+ \sum\limits_{g,h \in G}w^2(g,h)e_{g}x \otimes e_{h}+ \\
&\ \ \ \ \sum\limits_{g,h \in G}w^3(g,h)e_{g} \otimes e_{h}x+\sum\limits_{g,h \in G}w^4(g,h)e_{g}x \otimes e_{h}x.
\end{align*}
If $R$ satisfies $\Delta(e_g)R=R\Delta(e_g)$ for $ g\in G$, then
 \begin{itemize}
  \item[(i)] \label{E} $w^2(t,g)=0,\;t \in T,g\in G$.
  \item[(ii)] \label{F} $w^3(g,t)=0,\;t \in T,g\in G$.
  \item[(iii)] \label{G} $w^4(s,t)=w^4(t,s)=0,\;s \in S,t \in T$.
\end{itemize}
\end{lemma}
\begin{proof}
Since $\Delta(e_g)R=R\Delta(e_g)$ for $g\in G$, we have the following equations
\begin{gather}
       \Delta(e_g) [\sum\limits_{h,k \in G}w^2(h,k)e_{h}x \otimes e_{k}]=[\sum\limits_{h,k \in G}w^2(h,k)e_{h}x \otimes e_{k}] \Delta(e_g),\label{38}\\
		\Delta(e_g) [\sum\limits_{h,k \in G}w^3(h,k)e_{h} \otimes e_{k}x]=[\sum\limits_{h,k
 \in G}w^3(h,k)e_{h} \otimes e_{k}x] \Delta(e_g), \label{39}\\
		\Delta(e_g) [\sum\limits_{h,k \in G}w^4(h,k)e_{h}x \otimes e_{k}x]=[\sum\limits_{h,k \in G}w^4(h,k)e_{h}x \otimes e_{k}x] \Delta(e_g).\label{40}
\end{gather}
Firstly, we analyze equation \eqref{38} as follows
\begin{align}
            \Delta(e_g) [\sum\limits_{h,k \in G}w^2(h,k)e_{h}x \otimes e_{k}]&=\sum\limits_{\begin{subarray}{l}  h,k \in G \\
                             hk=g  \\
        \end{subarray}}w^2(h,k)e_{h}x \otimes e_{k}, \label{41}\\
             [\sum\limits_{h,k \in G}w^2(h,k)e_{h}x \otimes e_{k}] \Delta(e_g)&=\sum\limits_{\begin{subarray}{l}  h,k \in G  \\
                             hk=g  \\
        \end{subarray}}w^2(h\triangleleft x,k)e_{h\triangleleft x}x \otimes e_{k}.  \label{42}
\end{align}
Note that if $h \in T,k \in G$ such that $hk=g$, then $e_{h}x \otimes e_{k}$ will appear in \eqref{41} while not in \eqref{42}. As a result $w^2(h,k)=0\ $ for $ h \in T,k \in G$ and thus (i) has been proved.

Similarly, for equation \eqref{39}, there are the following equations
\begin{align}
            \Delta(e_g) [\sum\limits_{h,k \in G}w^3(h,k)e_{h} \otimes e_{k}x]&=\sum\limits_{\begin{subarray}{l}  h,k \in G   \\
                             hk=g  \\
        \end{subarray}}w^3(h,k)e_{h} \otimes e_{k}x,\   \label{43}  \\
              [\sum\limits_{h,k \in G}w^3(h,k)e_{h} \otimes e_{k}x] \Delta(e_g)&=\sum\limits_{\begin{subarray}{l}  h,k \in G   \\
                             hk=g  \\
        \end{subarray}}w^3(h,k \triangleleft x)e_{h} \otimes e_{k \triangleleft x}x.\   \label{44}
\end{align}
Observe that if $h \in G, k \in T$ such that $hk=g$, then $e_{h} \otimes e_{k}x$ will appear in \eqref{43} while not in \eqref{44}. Therefore $w^3(h,k)=0\ $ for $h \in G,k \in T$ and so (ii) is proved.

 For  equation \eqref{40}, we obtain the following equations
 \begin{align}
  \Delta(e_g) [\sum\limits_{h,k \in G}w^4(h,k)e_{h}x \otimes e_{k}x]&=\sum\limits_{\begin{subarray}{l}  h,k \in G  \\
                             hk=g  \\
        \end{subarray}}w^4(h,k)e_{h}x \otimes e_{k}x,\   \label{45} \\
  [\sum\limits_{h,k \in G}w^4(h,k)e_{h}x \otimes e_{k}x] \Delta(e_g)&=\sum\limits_{\begin{subarray}{l}  h,k \in G   \\
                             hk=g  \\
        \end{subarray}}w^4(h \triangleleft x,k \triangleleft x)e_{h \triangleleft x} \otimes e_{k \triangleleft x}x.\    \label{46}
\end{align}
Note that if $h \in S,k\in T$, then $e_{h}x \otimes e_{k}x$ and $e_{k}x \otimes e_{h}x$ will appear in \eqref{45} and not in \eqref{46}. This implies that $w^4(h,k)=0$ for $h \in S,k\in T$. Similarly, one can find that $w^4(h,k)=0$ for $h \in T,k\in S$. Therefore (iii) has been proved.
\end{proof}
\begin{lemma}\label{lem2.2.5}
 Let $R$ be the element given in Lemma \ref{lem2.2.4} and assume that $(\Delta \otimes \Id)(R)=R_{13}R_{23},\;(\Id \otimes \Delta)(R)=R_{13}R_{12}$. Then the following equations hold
\begin{itemize}
  \item[(i)] \label{H} $w^2(s_1,s_2)=w^3(s_1,s_2)=w^4(s_1,s_2)=0,\; s_1,s_2\in S$.
  \item[(ii)] \label{K} $w^1(g,t_2)w^4(t_1,t_2)=0,\; g\in G,\;t_1,t_2 \in T$.
  \item[(iii)] \label{L} $w^1(t_1,g)w^4(t_1,t_2)=0,\; g \in G,\;t_1,t_2 \in T$.
\end{itemize}
\end{lemma}
\begin{proof}
We have known $l(X_g)l(X_h)=l(X_gX_h)$ for $g,h\in G$ due to Lemma \ref{lem2.2.1}. Let $s \in S$ and we can find $t_1,t_2 \in T$ such that $t_1t_2=s$ because of Lemma \ref{lem2.2.3} and hence the following equation holds
\begin{equation*}
l(X_{t_1}X_{t_2})=\tau(t_1,t_2)l(X_{t_1t_2})=\tau(t_1,t_2)[\sum\limits_{g \in G}w^2(t_1t_2,g)e_g+\sum\limits_{s \in S}w^4(t_1t_2,s)e_sx].
\end{equation*}
At the same time,
\begin{align*}
l(X_{t_1})l(X_{t_2})&=(\sum\limits_{t \in T}w^4(t_1,t)e_tx)(\sum\limits_{t \in T}w^4(t_2,t)e_tx)\\
                &=\sum\limits_{t \in T}w^4(t_1,t)w^4(t_2,t\triangleleft x)e_tx^2  \\
        &=\sum\limits_{t \in T}w^4(t_1,t)w^4(t_2,t\triangleleft x)\sigma(t)e_t.
\end{align*}
Since $l(X_{t_1})l(X_{t_2})=l(X_{t_1}X_{t_2})$, we get that $w^4(s,s')=w^2(s,s')=0$ for $s' \in S$ and thus $w^4(s,s')=w^2(s,s')=0$ for $s,s' \in S$. Similarly by $r(X_{t_1})r(X_{t_2})=r(X_{t_2}X_{t_1})$ one can get that $w^3(s,s')=0$ for $s,s' \in S.$ Therefore, (i) is proved.

It remains to show (ii) and (iii). We have known $l(E_g)l(X_{t_1})=0$ due to Lemma \ref{lem2.2.2}. However a direct computation shows that $l(E_g)l(X_{t_1})=\sum_{t \in T}w^1(g,t)w^4(t_1,t)e_tx$. Therefore $w^1(g,t)w^4(t_1,t)=0$ for $ g\in G,t_1,t\in T$.  Similarly, by $r(E_g)r(X_{t_1})=0$ we get that $w^1(t,g)w^4(t,t_1)=0$ for $g\in G,t_1,t\in T$. These are exactly (ii), (iii).
\end{proof}
The following proposition shows that universal $\mathcal{R}$-matrices of $\Bbbk^G\#_{\sigma,\tau}\Bbbk \mathbb{Z}_{2}$ has only two possible forms.
\begin{proposition}\label{pro2.2.1}
Let $R$ be the element given in Lemma \ref{lem2.2.4} and assume that it is a universal $\mathcal{R}$-matrix of $\Bbbk^G\#_{\sigma,\tau}\Bbbk \mathbb{Z}_{2}$. Then $R$ must belong to one of the following two cases:
\begin{itemize}
             \item[Case 1:]$R=\sum\limits_{g,h\in G}w^1(g,h)e_g \otimes e_h$;
              \item[Case 2:] $R=\sum\limits_{s_1,s_2 \in S}w^1(s_1,s_2)e_{s_1} \otimes e_{s_2}+ \sum\limits_{s \in S, t \in T}w^2(s,t)e_{s}x \otimes e_{t}+
        \sum\limits_{t \in T,s \in S}w^3(t,s)e_{t} \otimes e_{s}x+
         \sum\limits_{t_1,t_2 \in T}w^4(t_1,t_2)e_{t_1}x \otimes e_{t_2}x$.
            \end{itemize}
\end{proposition}
\begin{proof}
Owing to Lemmas \ref{lem2.2.4} and \ref{lem2.2.5}, we can assume that $R$ has the following form:
\begin{align*}
R&=\sum_{g,h \in G}w^1(g,h)e_{g} \otimes e_{h}+ \sum_{s \in S,t \in T}w^2(s,t)e_{s}x \otimes e_{t}+\\
&\ \ \   \sum_{t \in T,s \in S}w^3(t,s)e_{t} \otimes e_{s}x+
         \sum_{t_1,t_2 \in T}w^4(t_1,t_2)e_{t_1}x \otimes e_{t_2}x.
\end{align*}
If $w^4(t_1,t_2)=0$ for all $t_1,t_2 \in T$, then $l(X_{t_1})=l(X_{t_2})=0$. Using Lemma \ref{lem2.2.2} we know that $l(X_{t_1})l(X_{t_2})=l(X_{t_1}X_{t_2})$ and as a result $l(X_{t_1}X_{t_2})=0$ for all $t_1,t_2 \in T$. For $s\in S$ , we can take $t_1,t_2\in T$ such that $s=t_1t_2$. Hence we have that $l(X_{t_1}X_{t_2})=\tau(t_1,t_2)(\sum_{t \in T}w^2(s,t)e_t)=0$ which implies that $w^2(s,t)=0$ for $s\in S, \;t \in T$. Similarly, by $r(X_{t_1})=r(X_{t_2})=0$ and $r(X_{t_2}X_{t_1})=\sum_{t \in T}\tau(t_2,t_1)w^3(t,s)e_t$, we have $w^3(t,s)=0$ for $s\in S,t \in T$. Since $w^2(s,t)=w^3(t,s)=0$ for $s\in S, t\in T$, we know that $R=\sum_{g,h\in G}w^1(g,h)e_g \otimes e_h$ and therefore we get the first case.

If there are $t_0,t_0' \in T$ such that $w^4(t_0,t_0')\neq 0$, then we will show that $w^1(t,g)=w^1(g,t)=0$ for all $g\in G, t\in T$.
For any $g\in G$, we have $w^1(g,t_0')w^4(t_0,t_0')=0$ by (ii) of Lemma \ref{lem2.2.5} and as a result $w^1(g,t_0')=0$. Since $R$ is invertible and $(e_t\otimes e_{t'_0})R=w^4(t,t_0')e_tx\otimes e_{t_0'}x$, we know that $w^4(t,t_0') \neq 0$ for $t\in T$. Next, we use (ii) and (iii) of Lemma \ref{lem2.2.5} repeatedly. We have $w^1(t,g)w^4(t,t_0')=0$ due to (iii) of Lemma \ref{lem2.2.5}. Thus $w^1(t,g)=0$ for $t\in T ,g\in G$. Since $R$ is invertible and $(e_{t_1}\otimes e_{t_2})R=w^4(t_1,t_2)e_{t_1}x\otimes e_{t_2}x$  for $t_1,t_2\in T$,  we get that $w^4(t_1,t_2)\neq 0$ for $t_1,t_2\in T$. Because $w^1(g,t)w^4(t_1,t)=0$ by (ii) of Lemma \ref{lem2.2.5}, we know that $w^1(g,t)=0$ for $g\in G,t \in T$ and hence we get the second case.
\end{proof}

\begin{remark}\emph{For simple, we will call a universal $\mathcal{R}$-matrix $R$ in Case 1 (resp. Case 2) of Proposition \ref{pro2.2.1} by a \emph{trivial} (resp. \emph{non-trivial}) quasitriagular structure.}\end{remark}

Recall that a minimal quasitriangular structure (see \cite[Definition 12.2.14]{R}) can be equivalently defined by: Assume that $R$ is a universal $\mathcal{R}$-matrix on Hopf algebra $H$, and if we let $H_l:=\{l(f)\;|\;f\in H^*\}$ and $H_r:=\{r(f)\;|\;f\in H^*\}$, then $(H,R)$ is minimal if $H=H_lH_r$ (see \cite[Proposition 12.2.13]{R}). In this case, we will say that $R$ is a minimal quasitriangular structure on $H$ and $H$ is a minimal quasitriangular Hopf algebra. Above proposition clearly implies the following corollary.

\begin{corollary}\label{coro2.2.1}
Every minimal quasitriangular structure on $\Bbbk^G\#_{\sigma,\tau}\Bbbk \mathbb{Z}_{2}$ is non-trivial.
\end{corollary}

\subsection{Universal $\mathcal{R}$-matrices of $H_{2n^2}$, $A_{2n^2,t}$}
To determine all universal $\mathcal{R}$-matrices of $H_{2n^2}$, $A_{2n^2,t}$, we give necessary conditions for $\Bbbk^G\#_{\sigma,\tau}\Bbbk \mathbb{Z}_{2}$ preserving a non-trivial quasitriangular structure firstly. For any finite set $X$, we use $|X|$ to denote the number of elements in $X$.
\begin{proposition}\label{pro2.3.1}
If there is a non-trivial quasitrianglar structure on $\Bbbk^G\#_{\sigma,\tau}\Bbbk \mathbb{Z}_{2}$, then
\begin{itemize}
  \item[(i)] $|S|=|T|$;
  \item[(ii)]  there is $b\in S$ such that $b^2=1$ and $t\triangleleft x=tb$ for $t\in T$;
  \item[(iii)] $|G|=4m$ for some $m\in \mathbb{N^+}$.
\end{itemize}
\end{proposition}

\begin{proof}
Assume that $R$ is a non-trivial quasitriangular structure on $\Bbbk^G\#_{\sigma,\tau}\Bbbk \mathbb{Z}_{2}$, then we have $l(E_{t_1})l(E_{t_2})=\sum_{s\in S}w^3(t_1,s)w^3(t_2,s)\sigma(s)e_s$ for $t_1,t_2 \in T$. In this situation, we claim that $TT=S$. In fact, suppose that there are $t_1,t_2 \in T$ satisfying $t_1t_2\in T$. Then it is easy to see that $l(E_{t_1t_2})=\sum_{s \in S}w^3(t_1t_2,s)e_{s}x$ which contradicts to the fact $l(E_{t_1})l(E_{t_2})=l(E_{t_1t_2})$ (Lemma \ref{lem2.2.1}).  Thus we have $TT=S$. Take a $t\in T$. We get that $tT\subseteq S$ and thus $|T|\leq |S|$. Since $tS\subseteq T$, $|T|\geq |S|$.
As a result we have $|T|=|S|$ and thus (i) has been proved. Next we will show (ii). Take a $t_0\in T,$ then we have $T=t_0S$. Let $t_0 \triangleleft x=t_1$ and denote $b=t_0^{-1}t_1$, then we have $b\in S$ by $TT=S$. Since $t_0S\subseteq T$ and $(t_0s)\triangleleft x=(t_0s)b$, we have $t\triangleleft x=tb$ for $t\in T$. It is easy to know that $b^2=1$ since $\triangleleft x$ is a group automorphism with order 2 and thus (ii) has been proved. Now let's show (iii).  By definition, $S$ is a subgroup of $G$ and $b^2=1$. Therefore we know that $2\;|\;|S|$. Since $|G|=|S|+|T|$ and $|S|=|T|$, we can see that $|G|=4m$ for some $m\in \mathbb{N^+}$. This completes the proof of (iii).
\end{proof}
\begin{corollary} \label{coro2.3.1}
If $|G|$ is an odd number or there are $t_1,t_2\in T$ such that $t_1^{-1}(t_1\triangleleft x)\neq t_2^{-1}(t_2\triangleleft x)$, then $\Bbbk^G\#_{\sigma,\tau}\Bbbk \mathbb{Z}_{2}$ has no non-trivial quasitriangular structures.
\end{corollary}

The following proposition determine all possible trivial quasitriangular structures.

\begin{proposition}\label{pro2.3.2}
If $R$ is a trivial quasitriangular structure on $\Bbbk^G\#_{\sigma,\tau}\Bbbk \mathbb{Z}_{2}$, then $R$ must be given by the following way
\begin{itemize}
\item[(i)] $R=\sum_{g,h \in G}w(g,h)e_{g} \otimes e_{h}$ for some bicharacter $w$ on $G$;
\item[(ii)]$w(g\triangleleft x,h\triangleleft x)=w(g,h)\eta (g,h)$ where $\eta(g,h)=\tau(g,h)\tau(h,g)^{-1}$ for $g,h \in G$.
\end{itemize}
\end{proposition}

\begin{proof}
We can assume that $R=\sum_{g,h \in G}w(g,h)e_{g} \otimes e_{h}$ is a trivial quasitriangular structure on it. Owing to $(\Delta \otimes \id)(R)=R_{13}R_{23}$ and $(\id \otimes \Delta)(R)=R_{13}R_{12}$, we know (i).
Expanding $\Delta^{op}(x)R=R\Delta(x)$, one can get (ii).
\end{proof}

The following examples are applications of above results and we get all universal $\mathcal{R}$-matrices of $H_{2n^2}$, $A_{2n^2,t}$ respectively.

\begin{example}
\emph{Let $H_{2n^2}$ as before in Example \ref{ex2.1.5}. Then by definition we find that $a^{-1}(a\triangleleft x)=a^{-1}b$ and $b^{-1}(b\triangleleft x)=b^{-1}a$. We divide our consideration into two cases.
\begin{itemize}
\item[1) ] If $n=2$, then $H_8$ is the 8-dimensional Kac-Paljutkin algebra $K_8$. All possible quasitrigular structures on $K_8$ were given in \cite{W} (see see \cite[Lemma 5.4]{W}).
\item[2) ] If $n>2$, then $a^{-1}(a\triangleleft x)\neq b^{-1}(b\triangleleft x)$. Therefore $H_{2n^2}$ has no non-trivial quasitriangular structure by Corollary \ref{coro2.3.1}. Assume that $R=\sum_{g,h \in G}w(g,h)e_{g} \otimes e_{h}$ is a trivial quasitriangular structure on $H_{2n^2}$, then $w$ is a bicharacter on $G$ and it satisfies the following equations by Proposition \ref{pro2.3.2}
\begin{align}
&w(a,a)^n=1, \;\;\;\;w(a,b)^n=1, \label{equ7}\\
&w(b,a)=w(a,b), \;\;\;\;w(b,b)=w(a,a). \notag
\end{align}
Using the above series of equations \eqref{equ7}, we can get all universal $\mathcal{R}$-matrices of $H_{2n^2}$ ($n\geq 3$) easily: let $R$ be a universal $\mathcal{R}$-matrix of $H_{2n^2}$, then
$$R=\sum_{1\leq i,j,k,l \leq n}\alpha^{ik+jl}\beta^{il+jk}e_{a^ib^j} \otimes e_{a^kb^l}$$
for some $\alpha, \beta\in \Bbbk$ satisfying $\alpha^n=\beta^n=1.$
\end{itemize} }
\end{example}

\begin{example}
\emph{Let $A_{2n^2,t}$ as before in Example \ref{ex2.1.6}. Since $|G|=n^2$ is odd, we know that $A_{2n^2,t}$ only has trivial quasitriangular structures on it by Corollary \ref{coro2.3.1}. Assume that $R=\sum_{g,h \in G}w(g,h)e_{g} \otimes e_{h}$ is a trivial quasitriangular structure on it, then we get that $w$ is a bicharacter on $G$. It is easy to see that the condition (ii) of Proposition \ref{pro2.3.2} is equivalent to the following equations
\begin{align}
&w(a,a)^n=1,\;\;\;\;
w(a,b)=t^m, \;m\in \mathbb{N},\;n|(2m-1),\label{equ3}\\
&w(b,a)=w(a,b)^{-1},\;\;\;
w(b,b)^n=1. \notag
\end{align}
Using the above series of equations \eqref{equ3}, it is easy to see that  all universal $\mathcal{R}$-matrices of $A_{2n^2,t}$ are given by the following way: Let $R$ be a universal $\mathcal{R}$-matrix of $A_{2n^2,t}$, then
$$R=\sum_{1\leq i,j,k,l \leq n}\alpha^{ik}\beta^{jl}t^{m(il-jk)}e_{a^ib^j} \otimes e_{a^kb^l}$$ for some $\alpha, \beta\in \Bbbk,\; m \in \mathbb{N}$ satisfying $\alpha^n=\beta^n=1,\; n|(2m-1) $.}
\end{example}

\section{Quasitriangular structures on $K(8n,\sigma,\tau)$}
As another continuation of our Proposition \ref{pro2.2.1}, we want to determine all universal $\mathcal{R}$-matrices of $K(8n,\sigma,\tau)$ in this section. As a consequence, we get a class of minimal quaistriangular semisimple Hopf algebras.

\subsection{Analysis of $\eta$}
For a Hopf algebra $K(8n,\sigma,\tau)$, if we let $\eta(g,h)=\tau(g,h)\tau(h,g)^{-1}$ for $g,h \in G$, then  $\eta$ is a bicharacter on $G$ by $\tau$ is a 2-cocycle on the abelian group $G$. Becuase $b^2=1$ and $\eta$ is a bicharacter, we know that $\eta(a,b)^2=1$ and $\eta(a,b)=\eta(b,a)$. As a result, we have two cases for the value of $\eta(a,b)$, that is, $\eta(a,b)=1$ or $\eta(a,b)=-1$. Both of them can occur. For example, for the dihedral group algebra $\k D_8$ (a special case of $K(8,\sigma,\tau)$) we have $\eta(a,b)=1$. For the $8$-dimensional Kac-Paljutkin algebra $K_8$, we have $\eta(a,b)=-1$. We found that the quasitriangular structures in the two cases have similar expressions. To present our results conveniently, we assume that $\eta(a,b)=-1$ in subsections \ref{sec4.2}-\ref{sec4.3} and then we will give
all universal $\mathcal{R}$-matrices of $K(8n,\sigma,\tau)$ for the case $\eta(a,b)=1$ in subsection \ref{sec4.4}.

\subsection{Trivial form}\label{sec4.2} The trivial quasitriangular structures on $K(8n,\sigma,\tau)$ can be determined easily.

\begin{proposition}\label{trivial}
Assume that $R$ is a trivial quasitriangular structure on $K(8n,\sigma,\tau)$, then $$R=\sum_{1\leq i,k \leq n,\;1\leq j,l \leq 2}\alpha^{ik}\beta^{il-jk}(-1)^{j(k+l)}e_{a^ib^j} \otimes e_{a^kb^l}$$ for some $\alpha, \beta\in \Bbbk$ satisfying $\alpha^{2n}=\beta^2=1.$
\end{proposition}

\begin{proof}
Assume that $R=\sum_{g,h \in G}w(g,h)e_{g} \otimes e_{h}$ is a trivial quasitriangular structure on $K(8n,\sigma,\tau)$, then $w$ is a bicharacter on $G$ and satisfies the following equations by Proposition \ref{pro2.3.2}
\begin{align}
&w(a,a)^{2n}=1, \;\;\;\;\;\;\;\;\;\;\;\;w(a,b)^2=1.\label{equ.4.3}\\
&w(ab,ab)=w(a,a), \;\;\;\;w(ab,b)=-w(a,b).\notag
\end{align}
If we let $w(a,a):=\alpha$, $w(a,b):=\beta$, then it is not hard to see that the above series of equations \eqref{equ.4.3} hold if and only if $w(a^ib^j,a^kb^l)=\alpha^{ik}\beta^{il-jk}(-1)^{j(k+l)}$. This implies our desired result.
\end{proof}

\subsection{Nontrivial form}\label{sec4.3}
In this subsection, we will give all nontrivial universal $\mathcal{R}$-matrices of $K(8n,\sigma,\tau)$. By Corollary \ref{coro2.2.1} and Proposition \ref{pro2.3.1}, if there is a non-trivial quasitriangular structure $R$ on $\Bbbk^G\#_{\sigma,\tau}\Bbbk \mathbb{Z}_{2}$, then we have $|S|=|T|$ and $R$ has the following form
\begin{gather}
              R=\sum\limits_{s_1,s_2 \in S}w^1(s_1,s_2)e_{s_1} \otimes e_{s_2}+ \sum\limits_{s \in S, t \in T}w^2(s,t)e_{s}x \otimes e_{t}+
        \sum\limits_{t \in T,s \in S}w^3(t,s)e_{t} \otimes \label{r}\\
         e_{s}x+\sum\limits_{t_1,t_2 \in T}w^4(t_1,t_2)e_{t_1}x \otimes e_{t_2}x. \notag
\end{gather}

Observe that if we let $S=\{s_1,\cdots, s_m\}$ and $T=\{t_1,\cdots ,t_m\}$, then the functions $w^i(1\leq i \leq 4)$ can be viewed as 4 matrices, which are
$(w^1(s_i,s_j))_{1\leq i,j \leq m}$, $(w^2(t_i,s_j))_{1\leq i,j \leq m}$, $(w^3(s_i,t_j))_{1\leq i,j \leq m}$, $(w^4(t_i,t_j))_{1\leq i,j \leq m}$. So when we say $w^i(1\leq i \leq 4)$ we mean that they are matrices in the following content.

To construct all non-trivial quasitriangular structures on $K(8n,\sigma,\tau)$, we introduce the following notations to simplify the calculation.

\begin{notation}
\emph{Let $K(8n,\sigma,\tau)$ as before, then we introduce symbols $P_i,\;\lambda_{2i,j},\lambda_{2i+1,j}$ and the function $h$ as follows
\begin{itemize}
             \item[(i)] $ P_i:=\left\{
\begin{array}{lll}
\prod_{k=1}^{i-1}\tau(a,a^k)  & {i\geq 2,\;i\in \mathbb{N}}\\
1     & {i=0\text{\;or\;}i=1}
\end{array} \right.$;
              \item[(ii)] $\lambda_{2i,j}:=P_{2i}^{-1}\sigma^i(a^j),\;\lambda_{2i+1,j}:=P_{2i+1}^{-1}\sigma^i(a^j)$ $,\;i,j\in \mathbb{N}$;
              \item[(iii)] $h(t_1,t_2):=\frac{\tau(t_1,t_2)}{\tau(t_2\triangleleft x, t_1\triangleleft x)}$ $\text{for\;} t_1,t_2\in T$;
            \end{itemize}}
\end{notation}%
%

We use these notations to construct universal $\mathcal{R}$-matrices of $K(8n,\sigma,\tau)$ as follows. Let $\alpha,\beta \in \Bbbk$ such that $$(\alpha\beta)^n\lambda_{2n,1}=1,\;\;\frac{\beta^2}{\alpha^2}
=\frac{\tau(b,b)}{\tau(b,a)^2},$$  and $S_{j,0}:=\lambda_{2j+1,1}\alpha^{j+1}\beta^j,\;S_{j,1}:=h(a,a^{2j+1}b)\lambda_{2j+1,1}\alpha^j\beta^{j+1}$ for $j\in \mathbb{N}$. Now we can construct    $$R_{\alpha,\beta}$$
 in the form of \eqref{r} through letting
\begin{itemize}
\item[(i)] $w^1$ be given by \\\\
$ \left\{
\begin{array}{l}
w^1(a^{2i},a^{2j})=w^1(a^{2i}b,a^{2j})
=(\lambda_{2j,1})^{2i}(\alpha\beta)^{2ij}[\sigma(a^{2j})]^i \\
w^1(a^{2i},a^{2j}b)=-w^1(a^{2i}b,a^{2j}b)
=(\lambda_{2j,1})^{2i}(\alpha\beta)^{2ij}[\sigma(a^{2j})]^i
\end{array}\right.$,\\

\item[(ii)] $w^2$ be given by  \\\\
$ \left\{
\begin{array}{l}
w^2(a^{2i},a^{2j+1})=w^2(a^{2i},a^{2j+1}b)=\lambda_{2i,2j+1}[S_{j,0}S_{j,1}]^{i}\\
w^2(a^{2i}b,a^{2j+1})=-w^2(a^{2i}b,a^{2j+1}b)=\frac{\tau(b,a)\beta}{\tau(b,a^{2i})\alpha}
\lambda_{2i,2j+1}[S_{j,0}S_{j,1}]^{i}
\end{array} \right.$,\\

\item[(iii)] $w^3$ be given by \\\\
$ \left\{
\begin{array}{l}
w^3(a^{2i+1},a^{2j})=w^3(a^{2i+1}b,a^{2j})=\lambda_{2j,2i+1}[S_{i,0}S_{i,1}]^{j}\\
w^3(a^{2i+1},a^{2j}b)=-w^3(a^{2i+1}b,a^{2j}b)=-\frac{\tau(b,a)\beta}{\tau(b,a^{2j})\alpha}
\lambda_{2j,2i+1}[S_{i,0}S_{i,1}]^{j}
\end{array} \right.$,\\

 \item[(iv)] $w^4$ be given by \\\\
$ \left\{
\begin{array}{l}
w^4(a^{2i+1},a^{2j+1})=\lambda_{2i+1,2j+1}S_{j,0}^{i+1}S_{j,1}^{i} \\
w^4(a^{2i+1},a^{2j+1}b)=\lambda_{2i+1,2j+1}S_{j,0}^{i}S_{j,1}^{i+1} \\
w^4(a^{2i+1}b,a^{2j+1})=h(a^{2i+1}b,a^{2j+1})\lambda_{2i+1,2j+1}S_{j,0}^{i}S_{j,1}^{i+1}\\
w^4(a^{2i+1}b,a^{2j+1}b)=h(a^{2i+1}b,a^{2j+1}b)\lambda_{2i+1,2j+1}S_{j,0}^{i+1}S_{j,1}^{i}
\end{array} \right.$,
            \end{itemize}
for $0 \leq i,j \leq (n-1)$. Then we have the following theorem.
\begin{theorem}\label{pro4.2}
The set of elements $\{R_{\alpha,\beta}\;|\;\alpha,\beta \in \Bbbk \;,(\alpha\beta)^n\lambda_{2n,1}=1\; \text{and}\; \frac{\beta^2}{\alpha^2}=\frac{\tau(b,b)}{\tau(b,a)^2}\}$ gives all non-trivial quasitriangular structures on $K(8n,\sigma,\tau)$.
\end{theorem}
\begin{remark} \emph{(1) Note the set $S':=\{(\alpha,\beta)|\alpha, \beta\in \k, (\alpha\beta)^n\lambda_{2n,1}=1\;\text{and}\;  \frac{\beta^2}{\alpha^2}=\frac{\tau(b,b)}{\tau(b,a)^2}\}$ is not empty. Actually, we can show that $|S'|=4n.$ This means that we have $4n$-number of non-trivial quasitriangular structures on a $K(8n,\sigma,\tau)$.}

\emph{(2) The general idea of the proof of Theorem \ref{pro4.2} is: }

\emph{Part 1:} we need to show that $R_{\alpha,\beta}$ is a universal $\mathcal{R}$-matrix of $K(8n,\sigma,\tau)$. \emph{This is given in Proposition \ref{pro4.1}. To show Proposition \ref{pro4.1}, we need Lemma \ref{lem3.1} which gives an equivalent description of a universal $\mathcal{R}$-matrix. Then we use Lemmas \ref{lemm4.2},\ref{lemm4.3},\ref{lemm4.6} and \ref{lemm4.8}  to verify this equivalent description.
  The following diagram illustrates the relations between these lemmas and Proposition \ref{pro4.1}:
$$\text{Proposition}\; \ref{pro4.1} \Leftarrow \text{Lemma}\; \ref{lem3.1} \Leftarrow \begin{cases}
    \text{Lemma}\;\ref{lemm4.2}\\
    \text{Lemma}\;\ref{lemm4.3}\\
    \text{Lemma}\;\ref{lemm4.6} \Leftarrow \left\{
    \begin{array}{l}
        \text{Lemma}\;\ref{lemm4.5} \Leftarrow \text{Lemma}\;\ref{lemm4.4}\\
        \text{Lemma}\;\ref{lemm4.1}
	\end{array}
	\right.\\
    \text{Lemma}\;\ref{lemm4.8}&
\end{cases}
$$
Part 2: } we show that if $R$ is a universal $\mathcal{R}$-matrix of $K(8n,\sigma,\tau)$, then $R=R_{\alpha,\beta}$ for some $\alpha,\beta\in \k$ satisfying $(\alpha\beta)^n\lambda_{2n,1}=1,\;\frac{\beta^2}{\alpha^2}
=\frac{\tau(b,b)}{\tau(b,a)^2}.$ \emph{This is the Proposition \ref{addpro}. The basic observation to prove Proposition \ref{addpro} is Lemma \ref{lem3.2} which states that $R$ is \emph{essentially} determined by the fourth matrix $w^4$. And we use Lemma \ref{all.1} and Lemma \ref{all.2} to compute the $w^4$ of $R$. }
\end{remark}

The following content of this subsection is designed to prove Theorem \ref{pro4.2} and we start with the properties of the notations we introduced.
\begin{lemma}\label{lem4.1}
If we use the notation $P_i,\;\lambda_{2i,j},\;\lambda_{2i+1,j},\;h$ as above, then the following equations hold:
\begin{gather}
\label{4.b1} \frac{P_{2j}^{2i}}{P_{2i}^{2j}}=\frac{\sigma(a^{2j})^{i}}{\sigma(a^{2i})^{j}},\; i,j\geq 0,\\
\label{4.b3} P_{2i+1}^{2j}\sigma(a)^j\sigma(a^{2j})^i=P_{2j}^{2i}\sigma(a^{2i+1})^j
[\frac{\tau(b,a)}{\tau(b,a^{2i+1})}]^j,\;i,j\geq 0,\\
\label{4.b4} P_{2i+1}^{2j} \frac{\sigma(a^{2j+1})^i}{\sigma(a)^i}[\frac{\tau(b,a)}{\tau(b,a^{2j+1})}]^i=
P_{2j+1}^{2i} \frac{\sigma(a^{2i+1})^j}{\sigma(a)^j}[\frac{\tau(b,a)}{\tau(b,a^{2i+1})}]^j,\;i,j\geq 0,\\
\label{4.b5}
h(t_1,t_2)h(t_1\triangleleft x, t_2\triangleleft x)=1,\;h(t_1,t_2)=h(t_2,t_1),\;t_1,t_2\in T,\\
\label{4.b6}
h(a,a^{2i+1}b)=\frac{\tau(b,a)}{\tau(b,a^{2i+1})},\;i\geq 0.
\end{gather}
Moreover if there are $\alpha,\beta \in \Bbbk$ such that $(\alpha\beta)^n \lambda_{2n,1}=1$ and $\frac{\beta^2}{\alpha^2}=\frac{\tau(b,b)}{\tau(b,a)^2}$, then we have $[\frac{\tau(b,a)\beta}{\tau(b,a^{2i})\alpha}]^2=\frac{\sigma(a^{2i})}{\sigma(a^{2i}b)},\; i\geq 0$.
\end{lemma}
\begin{proof}
We use induction to prove \eqref{4.b1}. If $i,j\in \{0,1\}$, then $\frac{P_{2j}^{2i}}{P_{2i}^{2j}}=\frac{\sigma(a^{2j})^{i}}{\sigma(a^{2i})^{j}}=1$. Assume that \eqref{4.b1} hold for $(k,l)\leq (i,j)$ (we mean that $k\leq i$ and $l\leq j$). We consider that case $(i,j+1)$ at first. By
\begin{align*}
\frac{P_{2j+2}^{2i}}{P_{2i}^{2j+2}}&=\frac{P_{2j}^{2i}[\tau(a,a^{2j})\tau(a,a^{2j+1})]^{2i}}{P_{2i}^{2j+2}}\\
                &=\frac{P_{2j}^{2i}}{P_{2i}^{2j}} \frac{[\tau(a^2,a^{2j})\tau(a,a)]^{2i}}{P_{2i}^{2}}\;\;(\tau \;\text{is a 2-cocycle})\\
                &=\frac{P_{2j}^{2i}}{P_{2i}^{2j}} \frac{P_2^{2i}}{P_{2i}^2} \tau(a^2,a^{2j})^{2i}
\end{align*}
and
\begin{align*}
\frac{\sigma(a^{2j+2})^i} {\sigma(a^{2i})^{j+1}}&=\frac{[\sigma(a^{2j}) \sigma(a^2) \tau(a^2,a^{2j})^2]^{i}}{\sigma(a^{2i})^j\sigma(a^{2i})}\\
                &=\frac{P_{2j}^{2i}}{P_{2i}^{2j}} \frac{\sigma(a^2)^i}{\sigma(a^{2i})} \tau(a^2,a^{2j})^{2i}\;\;(\frac{P_{2j}^{2i}}{P_{2i}^{2j}}=\frac{\sigma(a^{2j})^{i}}{\sigma(a^{2i})^{j}} \text{\;by induction})\\
                &=\frac{P_{2j}^{2i}}{P_{2i}^{2j}} \frac{P_2^{2i}}{P_{2i}^2}
                 \tau(a^2,a^{2j})^{2i}, \;\;(\frac{P_{2}^{2i}}{P_{2i}^{2}}=\frac{\sigma(a^{2})^{i}}{\sigma(a^{2i})} \text{\;by induction})
\end{align*}
we have $\frac{P_{2j+2}^{2i}}{P_{2i}^{2j+2}}=\frac{\sigma(a^{2j+2})^i} {\sigma(a^{2i})^{j+1}}.$ Since the equation \eqref{4.b1} is symmetric for $i,j$, we know that it also holds for the case $(i+1,j)$.

We turn to the equation \eqref{4.b3}.  By definition,
\begin{align*}
P_{2i+1}^{2j}\sigma(a)^j\sigma(a^{2j})^i&=[P_{2i}^{2j}\tau(a,a^{2i})^{2j}]
\sigma(a)^j\sigma(a^{2j})^i.
\end{align*}
Since
\begin{align*}
P_{2j}^{2i} \sigma(a^{2i+1})^j
[\frac{\tau(b,a)}{\tau(b,a^{2i+1})}]^j&=P_{2i}^{2j} \frac{\sigma(a^{2j})^i}{\sigma(a^{2i})^j}
\sigma(a^{2i+1})^j [\frac{\tau(b,a)}{\tau(b,a^{2i+1})}]^j\\
&=P_{2i}^{2j} \sigma(a^{2j})^i [\frac{\sigma(a^{2i+1})}{\sigma(a^{2i})}]^j [\frac{\tau(b,a)}{\tau(b,a^{2i+1})}]^j\\
&=P_{2i}^{2j} \sigma(a^{2j})^i [\sigma(a) \tau(a,a^{2i}) \tau(ab,a^{2i})]^j [\frac{\tau(b,a)}{\tau(b,a^{2i+1})}]^j\\
&=[P_{2i}^{2j}\tau(a,a^{2i})^{2j}]\sigma(a)^j\sigma(a^{2j})^i [\frac{\tau(ba,a^{2i}) \tau(b,a)}{\tau(b,a^{2i+1})\tau(a,a^{2i})}]^j\\
&=[P_{2i}^{2j}\tau(a,a^{2i})^{2j}]\sigma(a)^j\sigma(a^{2j})^i,
\end{align*}
the equation \eqref{4.b3} holds.

For the equation \eqref{4.b4}, direct computations show that
\begin{align*}
P_{2i+1}^{2j} \frac{\sigma(a^{2j+1})^i}{\sigma(a)^i} [\frac{\tau(b,a)}{\tau(b,a^{2j+1})}]^i&=
[\tau(a,a^{2i})^{2j}P_{2i}^{2j}] \frac{\sigma(a^{2j+1})^i}{\sigma(a)^i} [\frac{\tau(b,a)}{\tau(b,a^{2j+1})}]^i\\
&=\frac{\tau(a,a^{2i})^{2j}}{\sigma(a)^i} P_{2i}^{2j} \sigma(a^{2j+1})^i [\frac{\tau(b,a)}{\tau(b,a^{2j+1})}]^i\\
&=\frac{\tau(a,a^{2i})^{2j}}{\sigma(a)^i}[P_{2j+1} \sigma(a)^i \sigma(a^{2i})^j]\;\;(\text{by}\; \eqref{4.b3} )\\
&=\tau(a,a^{2i})^{2j} P_{2j+1}^{2i} \sigma(a^{2i})^j\\
&=\tau(a,a^{2i})^{2j} [P_{2j}^{2i} \tau(a,a^{2j})^{2i}] \sigma(a^{2i})^j\\
&=[\tau(a,a^{2i})^{2j}\tau(a,a^{2j})^{2i}] [P_{2j}^{2i}\sigma(a^{2i})^j]
\end{align*}
and
\begin{align*}
P_{2j+1}^{2i} \frac{\sigma(a^{2i+1})^j}{\sigma(a)^j}[\frac{\tau(b,a)}{\tau(b,a^{2i+1})}]^j&=
[\tau(a,a^{2j})^{2i}\tau(a,a^{2i})^{2j}] [P_{2i}^{2j}\sigma(a^{2j})^i]\\
&=[\tau(a,a^{2i})^{2j}\tau(a,a^{2j})^{2i}] [P_{2i}^{2j}\sigma(a^{2j})^i]\\
&=[\tau(a,a^{2i})^{2j}\tau(a,a^{2j})^{2i}] [P_{2j}^{2i}\sigma(a^{2i})^j]\;\;(\text{by}\;\eqref{4.b1}).
\end{align*}
Therefore the equation \eqref{4.b4} holds.

To show the equation \eqref{4.b5}, we do the following calculations. By
\begin{align*}
h(t_1,t_2)h(t_1\triangleleft x, t_2\triangleleft x)&=\frac{\tau(t_1,t_2)}{\tau(t_2 \triangleleft x,t_1\triangleleft x)} h(t_1\triangleleft x, t_2\triangleleft x)\\
&=\frac{\tau(t_1,t_2)}{\tau(t_2 \triangleleft x,t_1\triangleleft x)} \frac{\tau(t_1 \triangleleft x,t_2\triangleleft x)}{\tau(t_2,t_1)}\\
&=\frac{\tau(t_1,t_2) \tau(t_1 \triangleleft x,t_2\triangleleft x)}{\tau(t_2,t_1) \tau(t_2 \triangleleft x,t_1\triangleleft x)}\\
&=\frac{\sigma(t_1t_2) \sigma(t_1)^{-1} \sigma(t_2)^{-1}}{\sigma(t_2t_1) \sigma(t_2)^{-1} \sigma(t_1)^{-1}}\\
&=1,
\end{align*}
we have $h(t_1,t_2)h(t_1\triangleleft x, t_2\triangleleft x)=1$. Since
\begin{align*}
h(t_1\triangleleft x, t_2\triangleleft x)&=\frac{\tau(t_1 \triangleleft x,t_2\triangleleft x)}{\tau(t_2,t_1)}\\
&=[\frac{\tau(t_2,t_1)}{\tau(t_1 \triangleleft x,t_2\triangleleft x)}]^{-1}\\
&=h(t_2,t_1)^{-1}
\end{align*}
and $h(t_1,t_2)h(t_1\triangleleft x, t_2\triangleleft x)=1$, we have $h(t_1,t_2)=h(t_2,t_1)$. Therefore the equation \eqref{4.b5} holds.

For the last equation \eqref{4.b6}, we have
\begin{align*}
h(a,a^{2i+1}b)&=\frac{\tau(a,a^{2i+1}b)}{\tau(a^{2i+1},ab)}
=\frac{\tau(a,a^{2i+1}b)}{\tau(a^{2i+1},ba)}=\frac{\tau(a,a^{2i+1}b) \tau(b,a)}{\tau(a^{2i+1},ba) \tau(b,a)}\\
&=\frac{\tau(a,a^{2i+1}b) \tau(b,a)}{\tau(a^{2i+1}b,a) \tau(a^{2i+1},b)}=\eta(a,a^{2i+1}b) \frac{\tau(b,a)}{\tau(a^{2i+1},b)} \\
&=-\frac{\tau(b,a)}{\tau(a^{2i+1},b)}\;\; (\text{by}\; \eta(a,a^{2i+1}b)=-1)\\
&=\frac{\tau(b,a)}{\tau(b,a^{2i+1})} \;\; (\text{by}\; \eta(b,a^{2i+1})=-1).
\end{align*}
This means that we get the equation \eqref{4.b6}.

Furthermore, if there are $\alpha,\beta \in \Bbbk$ such that $(\alpha\beta)^n \lambda_{2n,1}=1$ and $\frac{\beta^2}{\alpha^2}=\frac{\tau(b,b)}{\tau(b,a)^2}$, then we will show that $\frac{\tau(b,a)\beta}{\tau(b,a^{2i})\alpha}=\frac{\sigma(a^{2i})}{\sigma(a^{2i}b)},\; i\geq 0$.
Firstly, we claim that $\sigma(b)^{-1}=\tau(b,b)$. In fact, by $\sigma(ab)(\sigma(a)\sigma(b))^{-1}=\tau(a,b)\tau(ab,b)$ and $\sigma(a)=\sigma(a\triangleleft x)=\sigma(ab)$, we have $\sigma(b)^{-1}=\tau(a,b)\tau(ab,b)$. Thanks to $\tau$ is a 2-cocycle, we have $\tau(a,b)\tau(ab,b)=\tau(b,b)\tau(a,1)=\tau(b,b)$ and thus $\sigma(b)^{-1}=\tau(b,b)$. Secondly, since
\begin{align*}
[\frac{\tau(b,a)\beta}{\tau(b,a^{2i})\alpha}]^2&=[\frac{\tau(b,a)}{\tau(b,a^{2i})}]^2 \frac{\tau(b,b)}{\tau(b,a)^2}=\frac{\tau(b,b)}{\tau(b,a^{2i})^2}
\end{align*}
and
\begin{align*}
\frac{\sigma(a^{2i})}{\sigma(a^{2i}b)}&=\frac{\sigma(a^{2i})}{\sigma(a^{2i}) \sigma(b) \tau(b,a^{2i})^2}=\frac{1}{\sigma(b) \tau(b,a^{2i})^{2}}
\end{align*}
 together with $\sigma(b)^{-1}=\tau(b,b)$, we know that $[\frac{\tau(b,a)\beta}{\tau(b,a^{2i})\alpha}]^2=\frac{\sigma(a^{2i})}{\sigma(a^{2i}b)}$.
\end{proof}

The following Lemma  is used to prove Proposition \ref{pro4.1}.
\begin{lemma}\label{lem3.1}
Denote the dual of $K(8n,\sigma,\tau)$ by $H^*$, then $R$ is a universal $\mathcal{R}$-matrix of $K(8n,\sigma,\tau)$ if and only if the following equations hold
\begin{gather}
\label{e3.11} \tau(s_1,s_2)=\tau(s_2,s_1),\; s_1,s_2\in S,\\
\label{e3.12} w^2(s,t\triangleleft x)=w^2(s,t)\eta(s,t),\; s \in S,t\in T,\\
\label{e3.13} w^3(t\triangleleft x,s)=w^3(t,s)\eta(t,s),\;s \in S,t\in T,\\
\label{e3.14} \tau(t_2,t_1)w^4(t_1\triangleleft x,t_2\triangleleft x)=\tau(t_1\triangleleft x,t_2\triangleleft x)w^4(t_1,t_2),\;t_1,t_2 \in T,\\
\label{e3.15} l(f_1)l(f_2)=l(f_1f_2),\;r(f_1)r(f_2)=r(f_2f_1),\;f_1,f_2\in H^*.
\end{gather}
\end{lemma}

\begin{proof}
On the one hand, we have the following equation
\begin{align*}
\Delta^{op}(x)R&=[\sum\limits_{g,h \in G}\tau(h,g)e_g \otimes e_h](x \otimes x)R\\
                &=[\sum_{s_1,s_2 \in S}\tau(s_2,s_1)w^1(s_1,s_2)e_{s_1} \otimes e_{s_2}+ \sum_{s \in S, t \in T}\tau(t,s)w^2(s,t\triangleleft x)e_{s}x \otimes e_{t}+ \\
        &\sum_{t \in T, s \in S}\tau(s,t)w^3(t\triangleleft x,s)e_{t} \otimes e_{s}x+\\
        &\sum\limits_{t_1,t_2 \in T}\tau(t_2,t_1)w^4(t_1\triangleleft x,t_2 \triangleleft x)e_{t_1}x \otimes e_{t_2}x](x\otimes x),
\end{align*}
On the other hand, the following equation hold
\begin{align*}
R\Delta(x)&=R[\sum\limits_{g,h \in G}\tau(g,h)e_g \otimes e_h](x \otimes x)\\
                &=[\sum\limits_{s_1,s_2 \in S}\tau(s_1,s_2)w^1(s_1,s_2)e_{s_1} \otimes e_{s_2}+ \sum\limits_{s \in S, t \in T}\tau(s,t)w^2(s,t)e_{s}x \otimes e_{t}+\\
        &\sum\limits_{t \in T,  s \in S}\tau(t,s)w^3(t,s)e_{t} \otimes e_{s}x+\\
         &\sum\limits_{t_1,t_2 \in T}\tau(t_1\triangleleft x,t_2\triangleleft x)w^4(t_1,t_2)e_{t_1}x \otimes e_{t_2}x](x\otimes x).
\end{align*}
Therefore, $\Delta^{op}(x)R=R\Delta(x)$ holds if and only if equations \eqref{e3.11}-\eqref{e3.14} hold. The last equation is just Lemma \ref{lem2.2.1}.
\end{proof}

We use the following Lemmas \ref{lemm4.2} and \ref{lemm4.3} to check that $R_{\alpha,\beta}$ satisfies the first four equalities of Lemma \ref{lem3.1}.
\begin{lemma}\label{lemm4.2}
Let $K(8n,\sigma,\tau)$ as above, then $\tau(s_1,s_2)=\tau(s_2,s_1)$ for all $s_1,s_2\in S$.
\end{lemma}
\begin{proof}
Recall the $\eta$ was defined by $\eta(g,h)=\tau(g,h) \tau(h,g)^{-1}$ and $\eta$ is a bicharacter on $G$, thus $\eta(a,b)^2=\eta(a,1)=1$ and $\eta(b,a)^2=\eta(1,a)=1$. Since $S=\{a^{2i}b^j\;|\;i,j \geq 0\}$ and $\eta(a^{2i}b^j,a^{2k}b^l)=\eta(a,b)^{2il} \eta(b,a)^{2jk}=1$, we have $\tau(s_1,s_2)=\tau(s_2,s_1)$ for all $s_1,s_2\in S$.
\end{proof}

\begin{lemma}\label{lemm4.3}
The $w^i (1\leq i \leq4)$ of $R_{\alpha,\beta}$ satisfy the following equations
\begin{gather}
\label{w2} w^2(s,t\triangleleft x)=w^2(s,t)\eta(s,t),\; s \in S,t\in T,\\
\label{w3} w^3(t\triangleleft x,s)=w^3(t,s)\eta(t,s),\;s \in S,t\in T,\\
\label{w4} w^4(t_1,t_2)=h(t_1,t_2)w^4(t_1 \triangleleft x,t_2 \triangleleft x),\;t_1,t_2 \in T.
\end{gather}
\end{lemma}
\begin{proof}
To simplify the calculation, we analyze $\eta$ furthermore. Recall that $\eta$ is bicharacter on $G$ and we have assumed $\eta(a,b)=\eta(b,a)=-1$ in this subsection. Since
\begin{align*}
\eta(s,t)\eta(s,t\triangleleft x)&=\frac{\tau(s,t)}{\tau(t,s)}\eta(s,t\triangleleft x)=\frac{\tau(s,t)}{\tau(t,s)} \frac{\tau(s,t\triangleleft x)}{\tau(t\triangleleft x,s)}\\
&=\frac{\tau(s,t) \tau(s\triangleleft x,t\triangleleft x)}{\tau(t,s) \tau(t\triangleleft x,s \triangleleft x)}=\frac{\sigma(st) \sigma(s)^{-1} \sigma(t)^{-1}}{\sigma(ts) \sigma(t)^{-1} \sigma(s)^{-1}}\\
&=1,
\end{align*}
we know that $\eta(s,t\triangleleft x)=\eta(s,t)^{-1}$.
This observation can help us to simply the proof. Indeed, we find that if $w^2(s,t\triangleleft x)=w^2(s,t)\eta(s,t)$ then we have $w^2(s,t)=w^2(s,t\triangleleft x)\eta(s,t\triangleleft x)$ automatically. By $S=\{a^{2i}b^j\;|\;i,j\geq 0\}$ and $T=\{a^{2i+1}b^j\;|\;i,j\geq 0\}$, this discussion tells us that to show the equation \eqref{w2} we only need to show the following two special cases:
\begin{align}
\label{w2.1} &w^2(a^{2i},a^{2j+1}\triangleleft x)=w^2(a^{2i},a^{2j+1})\eta(a^{2i},a^{2j+1}),\; i,j\geq 0,\\
\label{w2.2} &w^2(a^{2i}b,a^{2j+1}\triangleleft x)=w^2(a^{2i}b,a^{2j+1})\eta(a^{2i}b,a^{2j+1}),\; i,j\geq 0.
\end{align} Using the same arguments (we have $\eta(t\triangleleft x,s)=\eta(t,s)^{-1}$ similarly and $h(t_1\triangleleft x, t_2\triangleleft x)=h(t_1,t_2)^{-1}$ by \eqref{4.b5}), to show the equations \eqref{w3} and \eqref{w4} we only need to show the following equations
\begin{align}
\label{w3.1} &w^3(a^{2i+1}\triangleleft x,a^{2j})=w^3(a^{2i+1},a^{2j})\eta(a^{2i+1},a^{2j}),\;i,j\geq 0,\\
\label{w3.2} &w^3(a^{2i+1}\triangleleft x,a^{2j}b)=w^3(a^{2i+1},a^{2j}b)\eta(a^{2i+1},a^{2j}b),\;i,j\geq 0,\\
\label{w4.1} &w^4(a^{2i+1}b,a^{2j+1})=h(a^{2i+1}b,a^{2j+1})w^4((a^{2i+1}b) \triangleleft x,a^{2j+1} \triangleleft x),\;i,j\geq 0,\\
\label{w4.2} &w^4(a^{2i+1}b,a^{2j+1}b)=h(a^{2i+1}b,a^{2j+1}b)w^4((a^{2i+1}b) \triangleleft x,(a^{2j+1}b) \triangleleft x),\;i,j\geq 0.
\end{align}

We will check them one by one. Since
\begin{align*}
&w^2(a^{2i},a^{2j+1}\triangleleft x)=
w^2(a^{2i},a^{2j+1}b)
=\lambda_{2i,2j+1}[S_{j,0}S_{j,1}]^{i}\\
&w^2(a^{2i},a^{2j+1})\eta(a^{2i},a^{2j+1})=w^2(a^{2i},a^{2j+1})
=\lambda_{2i,2j+1}[S_{j,0}S_{j,1}]^{i},
\end{align*}
we have $w^2(a^{2i},a^{2j+1}\triangleleft x)=w^2(a^{2i},a^{2j+1})\eta(a^{2i},a^{2j+1})$ and therefore the equation \eqref{w2.1} holds.  Since
\begin{align*}
w^2(a^{2i}b,a^{2j+1}\triangleleft x)&=
w^2(a^{2i}b,a^{2j+1}b)\\
&=\frac{\tau(b,a)\beta}{\tau(b,a^{2i})\alpha}\lambda_{2i,2j+1}[S_{j,0}S_{j,1}]^{i}
\end{align*}
and
\begin{align*}
w^2(a^{2i}b,a^{2j+1})\eta(a^{2i}b,a^{2j+1})&=w^2(a^{2i}b,a^{2j+1})(-1)\\
&=-\frac{\tau(b,a)\beta}{\tau(b,a^{2i})\alpha}\lambda_{2i,2j+1}[S_{j,0}S_{j,1}]^{i}(-1)\\
&=\frac{\tau(b,a)\beta}{\tau(b,a^{2i})\alpha}\lambda_{2i,2j+1}[S_{j,0}S_{j,1}]^{i},
\end{align*}
we have $w^2(a^{2i}b,a^{2j+1}\triangleleft x)=w^2(a^{2i}b,a^{2j+1})\eta(a^{2i}b,a^{2j+1})$ and therefore the equation \eqref{w2.2} holds. Since
\begin{align*}
w^3(a^{2i+1}\triangleleft x,a^{2j})&=
w^3(a^{2i+1}b,a^{2j})\\
&=\lambda_{2j,2i+1}[S_{i,0}S_{i,1}]^{j}
\end{align*}
and
\begin{align*}
w^3(a^{2i+1},a^{2j})\eta(a^{2i+1},a^{2j})&=w^3(a^{2i+1},a^{2j})\\
&=\lambda_{2j,2i+1}[S_{i,0}S_{i,1}]^{j},
\end{align*}
we have $w^3(a^{2i+1}\triangleleft x,a^{2j})=w^3(a^{2i+1},a^{2j})\eta(a^{2i+1},a^{2j})$ and therefore the equation \eqref{w3.1} holds. Since
\begin{align*}
w^3(a^{2i+1}\triangleleft x,a^{2j}b)&=
w^3(a^{2i+1}b,a^{2j}b)\\
&=\frac{\tau(b,a)\beta}{\tau(b,a^{2j})\alpha} \lambda_{2j,2i+1}[S_{i,0}S_{i,1}]^{j}
\end{align*}
and
\begin{align*}
w^3(a^{2i+1},a^{2j}b)\eta(a^{2i+1},a^{2j}b)&=w^3(a^{2i+1},a^{2j}b) (-1)\\
&=\frac{-\tau(b,a)\beta}{\tau(b,a^{2j})\alpha} \lambda_{2j,2i+1}[S_{i,0}S_{i,1}]^{j} (-1)\\
&=\frac{\tau(b,a)\beta}{\tau(b,a^{2j})\alpha} \lambda_{2j,2i+1}[S_{i,0}S_{i,1}]^{j},
\end{align*}
we have $w^3(a^{2i+1}\triangleleft x,a^{2j}b)=w^3(a^{2i+1},a^{2j}b)\eta(a^{2i+1},a^{2j}b)$ and therefore the equation \eqref{w3.2} holds. Since
\begin{align*}
w^4(a^{2i+1}b,a^{2j+1})&=h(a^{2i+1}b,a^{2j+1})\lambda_{2i+1,2j+1}S_{j,0}^{i}S_{j,1}^{i+1}\\
&=h(a^{2i+1}b,a^{2j+1}) w^4(a^{2i+1},a^{2j+1}b)
\end{align*}
and
\begin{align*}
w^4(a^{2i+1}b,a^{2j+1}b)&=h(a^{2i+1}b,a^{2j+1}b)\lambda_{2i+1,2j+1}S_{j,0}^{i+1}S_{j,1}^{i}\\
&=h(a^{2i+1}b,a^{2j+1}b) w^4(a^{2i+1},a^{2j+1}),
\end{align*}
the equation \eqref{w4.1} and the equation \eqref{w4.2} hold.
\end{proof}

The following Lemma \ref{lemm4.4} is completely prepared to prove Lemma \ref{lemm4.5}.
\begin{lemma}\label{lemm4.4}
Let $R_{\alpha,\beta}$ as above, then
\begin{gather}
\label{w.1} [w^4(a,t)w^4(a,t\triangleleft x)\sigma(t)]^n=P_{2n},\;w^2(b,t)^2=\tau(b,b),\\
\label{w.2} w^4(a,t)w^2(b,t\triangleleft x)=\tau(a,b)w^4(ab,t),\;w^4(a,t)w^2(b,t)=\tau(b,a)w^4(ab,t),\\
\label{w.3} [w^4(a,t) w^4(a,t\triangleleft x)\sigma(t)]^i=P_{2i}w^2(a^{2i},t),\\
\label{w.4} w^2(a^{2i},t) w^2(b,t)=\tau(a^{2i},b) w^2(a^{2i}b,t),\\
\label{w.5} w^4(a,t)^{i+1} w^4(a,t\triangleleft x)^i \sigma(t)^i=P_{2i+1} w^4(a^{2i+1},t),\\
\label{w.6} w^4(a^{2i+1},t) w^2(b,t \triangleleft x)=\tau(a^{2i+1},b) w^4(a^{2i+1}b,t),\\
\label{w.7} w^3(a,s)^{2n}\sigma(s)^n=1,\;w^1(b,s)^2=1,\\
\label{w.8} w^1(b,s) w^3(a,s)=w^3(ab,s),\\
\label{w.9} w^3(a,s)^{2i} \sigma(s)^i=w^1(a^{2i},s),\;w^1(a^{2i},s) w^1(b,s)=w^1(a^{2i}b,s)\\
\label{w.10} w^3(a,s)^{2i+1}\sigma(s)^i=w^3(a^{2i+1},s),\;w^3(a^{2i+1},s) w^1(b,s)=w^3(a^{2i+1}b,s).
\end{gather}
\end{lemma}
\begin{proof}
We will show  \eqref{w.1} at first. Since $T=\{a^{2i+1},a^{2i+1}b\;|\;i\geq 0\}$, we need to show the following equations:
\begin{gather*}
[w^4(a,a^{2j+1}) w^4(a,a^{2j+1}b) \sigma(a^{2j+1})]^n=P_{2n},\\
[w^4(a,a^{2j+1}b) w^4(a,a^{2j+1}) \sigma(a^{2j+1}b)]^n=P_{2n}.
\end{gather*}
Since $\sigma(a^{2j+1})=\sigma(a^{2j+1}\triangleleft x)$ by the definition of $\sigma$ and $\sigma(a^{2j+1}\triangleleft x)=\sigma(a^{2j+1}b)$, we have $\sigma(a^{2j+1})=\sigma(a^{2j+1}b)$. This implies that we only need to show the first equation. Since
\begin{align*}
[w^4(a,a^{2j+1}) w^4(a,a^{2j+1}b) \sigma(a^{2j+1})]^n&=[S_{j,0} S_{j,1} \sigma(a^{2j+1})]^{n},
\end{align*}
\begin{align*}
[S_{j,0} S_{j,1} \sigma(a^{2j+1})]^{n}&=[(\lambda_{2j+1,1}\alpha^{j+1}\beta^j) S_{j,1} \sigma(a^{2j+1})]^{n}\\
&=[(\lambda_{2j+1,1}\alpha^{j+1}\beta^j) (h(a,a^{2j+1}b)\lambda_{2j+1,1}\alpha^j\beta^{j+1}) \sigma(a^{2j+1})]^{n}\\
&=h(a,a^{2j+1}b)^n \sigma(a^{2j+1})^{n} \lambda_{2j+1,1}^{2n} (\alpha \beta)^{(2j+1)n}
\end{align*}
and
\begin{align*}
\lambda_{2j+1,1}^{2n} (\alpha \beta)^{(2j+1)n} &=\lambda_{2j+1,1}^{2n} [(\alpha \beta)^{n}]^{2j+1}=\lambda_{2j+1,1}^{2n} [\lambda_{2n,1}^{-1}]^{2j+1} \\
&=[P_{2j+1}^{-1} \sigma(a)^j]^{2n} [\lambda_{2n,1}]^{-2j-1}=[P_{2j+1}^{-2n} \sigma(a)^{2jn}] [\lambda_{2n,1}]^{-2j-1}\\
&=[P_{2j+1}^{-2n} \sigma(a)^{2jn}] [P_{2n}^{-1} \sigma(a)^n]^{-2j-1}=P_{2j+1}^{-2n} P_{2n}^{2j+1} \sigma(a)^{-n},
\end{align*}
we have
\begin{align*}
[w^4(a,a^{2j+1}) w^4(a,a^{2j+1}b) \sigma(a^{2j+1})]^n&=h(a,a^{2j+1}b)^n \sigma(a^{2j+1})^{n} P_{2j+1}^{-2n} P_{2n}^{2j+1} \sigma(a)^{-n}.
\end{align*}
By the equation \eqref{4.b3}, we have
\begin{align}
\label{t1} P_{2k+1}^{2l}\sigma(a)^l \sigma(a^{2l})^k=P_{2l}^{2k}\sigma(a^{2k+1})^l
[\frac{\tau(b,a)}{\tau(b,a^{2k+1})}]^l.
\end{align}
Let $k=j$ and $l=n$, then the equation \eqref{t1} above becomes (by noting that $a^{2n}=1$)
\begin{align}
P_{2j+1}^{2n}\sigma(a)^n=P_{2n}^{2j}\sigma(a^{2j+1})^n
[\frac{\tau(b,a)}{\tau(b,a^{2j+1})}]^n.
\end{align}
Therefore we have
\begin{align*}
[w^4(a,a^{2j+1}) w^4(a,a^{2j+1}b) \sigma(a^{2j+1})]^n&=h(a,a^{2j+1}b)^n P_{2j+1}^{-2n} P_{2n}^{2j} \sigma(a^{2j+1})^n \sigma(a)^{-n} P_{2n}\\
&=h(a,a^{2j+1}b)^n [\frac{\tau(b,a^{2j+1})}{\tau(b,a)}]^n P_{2n}\\
&=P_{2n}
\end{align*}
where for the last equality we use the equation \eqref{4.b6}. Since
\begin{align*}
[w^2(b,a^{2j+1})]^2&=[-\tau(b,a)\frac{\beta}{\alpha}]^2=\tau(b,a)^2 [\frac{\beta}{\alpha}]^2\\
&=\tau(b,a)^2 \frac{\tau(b,b)}{\tau(b,a)^2}=\tau(b,b),
\end{align*}
\begin{align*}
[w^2(b,a^{2j+1}b)]^2&=[\tau(b,a)\frac{\beta}{\alpha}]^2=\tau(b,a)^2 [\frac{\beta}{\alpha}]^2\\
&=\tau(b,a)^2 \frac{\tau(b,b)}{\tau(b,a)^2}=\tau(b,b)
\end{align*}
and $T=\{a^{2i+1},a^{2i+1}b\;|\;i\geq 0\}$, we have $w^2(b,t)^2=\tau(b,b)$ by noting that $T=\{a^{2i+1},a^{2i+1}b\;|\;i\geq 0\}$. Therefore, we get the equations \eqref{w.1}.

Now let us show the equations \eqref{w.2}. Since
\begin{align*}
\tau(a,b) w^4(ab,a^{2j+1})&=\tau(a,b) [h(ab,a^{2j+1})S_{j,1}]\\
&=\tau(a,b) h(ab,a^{2j+1}) [h(a,a^{2j+1}b) \lambda_{2j+1,1} \alpha^j \beta^{j+1}]\\
&=\tau(a,b) h(ab,a^{2j+1}) [h(a,a^{2j+1}b) \frac{\beta}{\alpha} S_{j,0}]\\
&=\tau(a,b) [h(ab,a^{2j+1}) h(a,a^{2j+1}b)] \frac{\beta}{\alpha} S_{j,0}\\
&=\tau(a,b) [h(ab,a^{2j+1}) h((ab)\triangleleft x,a^{2j+1}\triangleleft x)] \frac{\beta}{\alpha} S_{j,0} \\
&=\tau(a,b) \frac{\beta}{\alpha} S_{j,0} \;\;(\text{by}\;\eqref{4.b5})\\
&=-\tau(b,a) \frac{\beta}{\alpha} S_{j,0}\;\;(\text{by} \;\eta(a,b)=-1)\\
&=(S_{j,0}) (-\tau(b,a) \frac{\beta}{\alpha})\\
&=w^4(a,a^{2j+1}) w^2(b,a^{2j+1}b),
\end{align*}
we have
\begin{align}
\label{w4.1.1} w^4(a,a^{2j+1}) w^2(b,a^{2j+1}b)=\tau(a,b) w^4(ab,a^{2j+1}).
\end{align}
Since
\begin{align*}
w^4(a,a^{2j+1}b) w^2(b,a^{2j+1})&=S_{j,1} w^2(b,a^{2j+1})=S_{j,1}[\tau(b,a)\frac{\beta}{\alpha}] \\
&=[h(a,a^{2j+1}b) \lambda_{2j+1,1} \alpha^j \beta^{j+1}] [\tau(b,a)\frac{\beta}{\alpha}]\\
&=[h(a,a^{2j+1}b) \frac{\beta}{\alpha} S_{j,0}] [\tau(b,a)\frac{\beta}{\alpha}]\\
&=\tau(b,a) h(a,a^{2j+1}b) (\frac{\beta}{\alpha})^2 S_{j,0}\\
&=\tau(b,a) \frac{\tau(b,a)}{\tau(b,a^{2j+1})} (\frac{\beta}{\alpha})^2 S_{j,0}\\
&=\tau(b,a) \frac{\tau(b,a)}{\tau(b,a^{2j+1})} \frac{\tau(b,b)}{\tau(b,a)^2} S_{j,0}\\
&=\frac{\tau(b,b)}{\tau(b,a^{2j+1})}  S_{j,0}
\end{align*}
and
\begin{align*}
\tau(a,b) w^4(ab,a^{2j+1}b)&=\tau(a,b)[h(ab,a^{2j+1}b) S_{j,0}]\\
&=\tau(a,b) \frac{\tau(ab,a^{2j+1}b)}{\tau(a^{2j+1},a)} S_{j,0}  \\
&=\frac{\tau(a,b) \tau(ab,a^{2j+1}b)}{\tau(a^{2j+1},a)} S_{j,0}\\
&=\frac{\tau(b,a^{2j+1}b) \tau(a,a^{2j+1})}{\tau(a^{2j+1},a)} S_{j,0}\\
&=\tau(b,a^{2j+1}b) S_{j,0}\\
&=\tau(b,ba^{2j+1}) S_{j,0}\\
&=\frac{\tau(b,ba^{2j+1}) \tau(b,a^{2j+1})}{\tau(b,a^{2j+1})} S_{j,0}\\
&=\frac{\tau(b,b)}{\tau(b,a^{2j+1})} S_{j,0},
\end{align*}
we have
\begin{align}
\label{w4.1.2} w^4(a,a^{2j+1}b) w^2(b,a^{2j+1})=\tau(a,b) w^4(ab,a^{2j+1}b).
\end{align}
 By equations \eqref{w4.1.1} and \eqref{w4.1.2}, we have $w^4(a,t)w^4(b,t\triangleleft x)=\tau(a,b) w^4(ab,t)$ for $t\in T=\{a^{2i+1},a^{2i+1}b\;|\;i\geq 0\}$. Since
\begin{align*}
w^2(b,a^{2j+1}) w^4(a,a^{2j+1})&=(\tau(b,a) \frac{\beta}{\alpha}) w^4(a,a^{2j+1})\\
&=(-w^2(b,a^{2j+1}b))w^4(a,a^{2j+1}) \\
&=-w^4(a,a^{2j+1}) w^2(b,a^{2j+1}b)\\
&=-\tau(a,b) w^4(ab,a^{2j+1})\;\;(\text{by }\;\eqref{w4.1.1})\\
&=\tau(b,a) w^4(ab,a^{2j+1})
\end{align*}
and
\begin{align*}
w^2(b,a^{2j+1}b) w^4(a,a^{2j+1}b)&=(-\tau(b,a) \frac{\beta}{\alpha}) w^4(a,a^{2j+1}b)\\
&=(-w^2(b,a^{2j+1}))w^4(a,a^{2j+1}b) \\
&=-w^4(a,a^{2j+1}b) w^2(b,a^{2j+1})\\
&=-\tau(a,b) w^4(ab,a^{2j+1}b)\;\;(\text{by }\;\eqref{w4.1.2})\\
&=\tau(b,a) w^4(ab,a^{2j+1}b),
\end{align*}
we have $w^2(b,t)w^4(a,t)=\tau(b,a)w^4(ab,t)$ and thus the equations \eqref{w.2} hold.

Next, we will show  \eqref{w.3}. Since
\begin{align*}
[w^4(a,a^{2j+1}) w^4(a,a^{2j+1}b) \sigma(a^{2j+1})]^i&=[S_{j,0} S_{j,1} \sigma(a^{2j+1})]^i
\end{align*}
and
\begin{align*}
P_{2i} w^2(a^{2i}, a^{2j+1})&=P_{2i}[\lambda_{2i,2j+1} (S_{j,0} S_{j,1})^i]\\
&=P_{2i} \lambda_{2i,2j+1} (S_{j,0} S_{j,1})^i\\
&=P_{2i} (P_{2i}^{-1} \sigma(a^{2j+1})^i) (S_{j,0} S_{j,1})^i\\
&=\sigma(a^{2j+1})^i (S_{j,0} S_{j,1})^i\\
&=[S_{j,0} S_{j,1} \sigma(a^{2j+1})]^i,
\end{align*}
we have
\begin{align}
\label{w4.2.1} [w^4(a,a^{2j+1}) w^4(a,a^{2j+1}b) \sigma(a^{2j+1})]^i=P_{2i} w^2(a^{2i},a^{2j+1}).
\end{align}
 Similarly, by
\begin{align*}
[w^4(a,a^{2j+1}b) w^4(a,a^{2j+1}) \sigma(a^{2j+1}b)]^i&=[w^4(a,a^{2j+1}b) w^4(a,a^{2j+1}) \sigma(a^{2j+1})]^i\\
&=P_{2i} w^2(a^{2i},a^{2j+1})\;\;(\text{by }\;\eqref{w4.2.1})\\
&=P_{2i} w^2(a^{2i},a^{2j+1}b),
\end{align*}
we have $[w^4(a,a^{2j+1}b) w^4(a,a^{2j+1}) \sigma(a^{2j+1}b)]^i=P_{2i} w^2(a^{2i},a^{2j+1}b)$. This means that we get the equation \eqref{w.3}.

We turn to the proof of the equation \eqref{w.4}. Since
\begin{align*}
w^2(a^{2i},a^{2j+1}) w^2(b,a^{2j+1})&=[\lambda_{2i,2j+1} (S_{j,0}S_{j,1})^i] w^2(b,a^{2j+1})\\
&=[\lambda_{2i,2j+1} (S_{j,0}S_{j,1})^i] [\tau(b,a)\frac{\beta}{\alpha}]\\
&=\tau(b,a) \frac{\beta}{\alpha} \lambda_{2i,2j+1} (S_{j,0}S_{j,1})^i
\end{align*}
and
\begin{align*}
\tau(a^{2i},b) w^2(a^{2i}b,a^{2j+1})&=\tau(a^{2i},b)[\frac{\tau(b,a)}{\tau(b,a^{2i})} \frac{\beta}{\alpha} \lambda_{2i,2j+1} (S_{j,0}S_{j,1})^i]\\
&=\frac{\tau(a^{2i},b)}{\tau(b,a^{2i})} [\tau(b,a) \frac{\beta}{\alpha} \lambda_{2i,2j+1} (S_{j,0}S_{j,1})^i]\\
&=\eta(a^{2i},b)[\tau(b,a) \frac{\beta}{\alpha} \lambda_{2i,2j+1} (S_{j,0}S_{j,1})^i]\\
&=\tau(b,a) \frac{\beta}{\alpha} \lambda_{2i,2j+1} (S_{j,0}S_{j,1})^i,
\end{align*}
we have
\begin{align}
\label{w4.3.1} w^2(a^{2i},a^{2j+1}) w^2(b,a^{2j+1})=\tau(a^{2i},b) w^2(a^{2i}b,a^{2j+1}).
\end{align}
Since
\begin{align*}
w^2(a^{2i},a^{2j+1}b) w^2(b,a^{2j+1}b)&=w^2(a^{2i},a^{2j+1}) w^2(b,a^{2j+1}b)\\
&=w^2(a^{2i},a^{2j+1}) [-w^2(b,a^{2j+1})]\\
&=-w^2(a^{2i},a^{2j+1}) w^2(b,a^{2j+1})\\
&=-\tau(a^{2i},b) w^2(a^{2i}b,a^{2j+1})\;\;(\text{by }\;\eqref{w4.3.1})\\
&=\tau(a^{2i},b) [-w^2(a^{2i}b,a^{2j+1})]\\
&=\tau(a^{2i},b) w^2(a^{2i}b,a^{2j+1}b),
\end{align*}
we have $w^2(a^{2i},a^{2j+1}b) w^2(b,a^{2j+1}b)=\tau(a^{2i},b) w^2(a^{2i}b,a^{2j+1}b)$. Therefore the equation \eqref{w.4} is proved.

We will show the equation \eqref{w.5}. Since
\begin{align*}
w^4(a,a^{2j+1})^{i+1} w^4(a,a^{2j+1}b)^i \sigma(a^{2j+1})^i&=S_{j,0}^{i+1} S_{j,1}^i \sigma(a^{2j+1})^i\\
&=\sigma(a^{2j+1})^i S_{j,0}^{i+1} S_{j,1}^i
\end{align*}
and
\begin{align*}
P_{2i+1} w^4(a^{2i+1},a^{2j+1})&=P_{2i+1} [\lambda_{2i+1,2j+1} S_{j,0}^{i+1} S_{j,1}^i]\\
&=P_{2i+1} [P_{2i+1}^{-1} \sigma(a^{2j+1})^i] S_{j,0}^{i+1} S_{j,1}^i\\
&=\sigma(a^{2j+1})^i S_{j,0}^{i+1} S_{j,1}^i,
\end{align*}
we have
\begin{align}
\label{w4.4.1} w^4(a,a^{2j+1})^{i+1} w^4(a,a^{2j+1}b)^i \sigma(a^{2j+1})^i=P_{2i+1}w^4(a^{2i+1},a^{2j+1}).
\end{align}
Since
\begin{align*}
w^4(a,a^{2j+1}b)^{i+1} w^4(a,a^{2j+1})^i \sigma(a^{2j+1}b)^i&=S_{j,1}^{i+1} S_{j,0}^i \sigma(a^{2j+1}b)^i\\
&=\sigma(a^{2j+1})^i S_{j,1}^{i+1} S_{j,0}^i
\end{align*}
and
\begin{align*}
P_{2i+1} w^4(a^{2i+1},a^{2j+1}b)&=P_{2i+1} [\lambda_{2i+1,2j+1} S_{j,0}^{i} S_{j,1}^{i+1}]\\
&=P_{2i+1} [P_{2i+1}^{-1} \sigma(a^{2j+1})^i] S_{j,0}^{i} S_{j,1}^{i+1}\\
&=\sigma(a^{2j+1})^i  S_{j,1}^{i+1} S_{j,0}^{i},
\end{align*}
we have
\begin{align}
\label{w4.4.2} w^4(a,a^{2j+1}b)^{i+1} w^4(a,a^{2j+1})^i \sigma(a^{2j+1}b)^i=P_{2i+1}w^4(a^{2i+1},a^{2j+1}b).
\end{align}
 Combining the equation \eqref{w4.4.1} together with the equation \eqref{w4.4.2}, we get the equation \eqref{w.5}.

Now let us go to the proof of the equation \eqref{w.6}. Directly we have
\begin{align*}
w^4(a^{2i+1},a^{2j+1}) w^2(b,a^{2j+1}b)&=[\lambda_{2i+1,2j+1} S_{j,0}^{i+1} S_{j,1}^i] w^2(b,a^{2j+1}b)\\
&=[\lambda_{2i+1,2j+1} S_{j,0}^{i+1} S_{j,1}^i] [-\tau(b,a)\frac{\beta}{\alpha}]\\
&=[\lambda_{2i+1,2j+1} S_{j,0}^{i+1} S_{j,1}^i] [\tau(a,b)\frac{\beta}{\alpha}]\\
&=\tau(a,b)\frac{\beta}{\alpha} [\lambda_{2i+1,2j+1} S_{j,0}^{i+1} S_{j,1}^i]
\end{align*}
and
\begin{align*}
\tau(a^{2i+1},b) w^4(a^{2i+1}b,a^{2j+1})&=\tau(a^{2i+1},b) [h(a^{2i+1}b,a^{2j+1}) \lambda_{2i+1,2j+1} S_{j,0}^{i} S_{j,1}^{i+1}]\\
&=\tau(a^{2i+1},b) [h(a^{2i+1}b,a^{2j+1}) \lambda_{2i+1,2j+1} S_{j,0}^{i+1} S_{j,1}^{i} \frac{S_{j,1}}{S_{j,0}}]\\
&=[\tau(a^{2i+1},b) h(a^{2i+1}b,a^{2j+1}) \frac{S_{j,1}}{S_{j,0}}] [\lambda_{2i+1,2j+1} S_{j,0}^{i+1} S_{j,1}^{i}].
\end{align*}
To prove $w^4(a,a^{2j+1}) w^2(b,a^{2j+1}b)=\tau(a^{2i+1},b) w^4(a^{2i+1}b,a^{2j+1})$, we need only to show that $\tau(a,b)\frac{\beta}{\alpha}=\tau(a^{2i+1},b) h(a^{2i+1}b,a^{2j+1}) \frac{S_{j,1}}{S_{j,0}}$. In fact, by
\begin{align*}
\frac{S_{j,1}}{S_{j,0}}&=\frac{h(a,a^{2j+1}b) \lambda_{2j+1,1} \alpha^j \beta^{j+1}}{\lambda_{2j+1,1} \alpha^{j+1} \beta^{j}}\\
&=h(a,a^{2j+1}b) \frac{\beta}{\alpha},
\end{align*}
we have
\begin{align*}
\tau(a^{2i+1},b) h(a^{2i+1}b,a^{2j+1}) \frac{S_{j,1}}{S_{j,0}}&=\tau(a^{2i+1},b) h(a^{2i+1}b,a^{2j+1}) h(a,a^{2j+1}b) \frac{\beta}{\alpha}\\
&=\tau(a^{2i+1},b) \frac{\tau(a^{2i+1}b,a^{2j+1})}{\tau(a^{2j+1}b,a^{2i+1})} h(a,a^{2j+1}b) \frac{\beta}{\alpha}\\
&=\frac{\tau(a^{2i+1},b) \tau(a^{2i+1}b,a^{2j+1})}{\tau(a^{2j+1}b,a^{2i+1})} h(a,a^{2j+1}b) \frac{\beta}{\alpha}\\
&=\frac{\tau(a^{2i+1},ba^{2j+1})) \tau(b,a^{2j+1})}{\tau(a^{2j+1}b,a^{2i+1})} h(a,a^{2j+1}b) \frac{\beta}{\alpha}\\
&=\eta(a^{2i+1},ba^{2j+1}) \tau(b,a^{2j+1}) h(a,a^{2j+1}b) \frac{\beta}{\alpha}\\
&=-\tau(b,a^{2j+1}) h(a,a^{2j+1}b) \frac{\beta}{\alpha}\\
&=-\tau(b,a^{2j+1}) \frac{\tau(a,a^{2j+1}b)}{\tau(a^{2j+1},ab)} \frac{\beta}{\alpha}\\
&=-\frac{\tau(b,a^{2j+1}) \tau(a,a^{2j+1}b)}{\tau(a^{2j+1},ab)} \frac{\beta}{\alpha}\\
&=-\frac{\tau(ab,a^{2j+1}) \tau(a,b)}{\tau(a^{2j+1},ab)} \frac{\beta}{\alpha}\\
&=-\eta(ab,a^{2j+1}) \tau(a,b)\frac{\beta}{\alpha}\\
&=\tau(a,b)\frac{\beta}{\alpha}.
\end{align*}
Therefore, we get
\begin{align}
\label{light1} w^4(a^{2i+1},a^{2j+1}) w^2(b,a^{2j+1}b)&=\tau(a^{2i+1},b) w^4(a^{2i+1}b,a^{2j+1}).
\end{align}
Now,
\begin{align*}
w^4(a^{2i+1},a^{2j+1}b) w^2(b,a^{2j+1})&=[h(a^{2i+1},a^{2j+1}b) \lambda_{2i+1,2j+1} S_{j,0}^{i+1} S_{j,1}^i] w^2(b,a^{2j+1}b)\\
&=h(a^{2i+1},a^{2j+1}b) \frac{S_{j,1}}{S_{j,0}} w^4(a^{2i+1},a^{2j+1}) w^2(b,a^{2j+1}b)\\
&=h(a^{2i+1},a^{2j+1}b)\frac{S_{j,1}}{S_{j,0}} w^4(a^{2i+1},a^{2j+1}b) [-\tau(b,a) \frac{\beta}{\alpha}]\\
&=h(a^{2i+1},a^{2j+1}b) \frac{S_{j,1}}{S_{j,0}} w^4(a^{2i+1},a^{2j+1}b) [-w^2(b,a^{2j+1})]\\
&=-h(a^{2i+1},a^{2j+1}b) \frac{S_{j,1}}{S_{j,0}} w^4(a^{2i+1},a^{2j+1}b) w^2(b,a^{2j+1})\\
&=-h(a^{2i+1},a^{2j+1}b) \frac{S_{j,1}}{S_{j,0}} \tau(a^{2j+1},b) w^4(a^{2i+1}b,a^{2j+1})\\
&=\tau(a^{2i+1},b)[-h(a^{2i+1},a^{2j+1}b) \frac{S_{j,1}}{S_{j,0}} w^4(a^{2i+1}b,a^{2j+1})]\\
\end{align*}
and
\begin{align*}
\tau(a^{2i+1},b) w^4(a^{2i+1}b,a^{2j+1}b)&=\tau(a^{2i+1},b) [h(a^{2i+1}b,a^{2j+1}b) \lambda_{2i+1,2j+1} S_{j,0}^{i+1} S_{j,1}^{i}]\\
&=\tau(a^{2i+1},b) [h(a^{2i+1}b,a^{2j+1}b) w^4(a^{2i+1},a^{2j+1})].
\end{align*}
Therefore similarly to get our desired equation we need only to prove the following
\begin{align}
\label{light2} -h(a^{2i+1},a^{2j+1}b) \frac{S_{j,1}}{S_{j,0}} w^4(a^{2i+1}b,a^{2j+1})=h(a^{2i+1}b,a^{2j+1}b) w^4(a^{2i+1},a^{2j+1}).
\end{align}
As a matter of fact, by
\begin{align*}
h(a^{2i+1},a^{2j+1}b)w^4(a^{2i+1}b,a^{2j+1})&=h(a^{2i+1},a^{2j+1}b) h(a^{2i+1}b,a^{2j+1}) \lambda_{2i+1,2j+1} S_{j,0}^i S_{j,1}^{i+1}\\
&=\lambda_{2i+1,2j+1} S_{j,0}^i S_{j,1}^{i+1}\;\;(\text{by}\;\eqref{4.b5} )\\
&=\frac{S_{j,1}}{S_{j,0}} w^4(a^{2i+1}b,a^{2j+1}),
\end{align*}
we have
\begin{align*}
-h(a^{2i+1},a^{2j+1}b) \frac{S_{j,1}}{S_{j,0}} w^4(a^{2i+1}b,a^{2j+1})&=-\frac{S_{j,1}}{S_{j,0}} [h(a^{2i+1},a^{2j+1}b) w^4(a^{2i+1}b,a^{2j+1})]\\
&=-\frac{S_{j,1}}{S_{j,0}} \frac{S_{j,1}}{S_{j,0}} w^4(a^{2i+1}b,a^{2j+1})\\
&=-[\frac{S_{j,1}}{S_{j,0}}]^2 w^4(a^{2i+1}b,a^{2j+1})\\
&=-[h(a,a^{2j+1}b) \frac{\beta}{\alpha}]^2 w^4(a^{2i+1}b,a^{2j+1})\\
&=-h(a,a^{2j+1}b)^2 \frac{\tau(b,b)}{\tau(b,a)^2} w^4(a^{2i+1}b,a^{2j+1})\\
&=-[\frac{\tau(b,a)}{\tau(b,a^{2j+1})}]^2 \frac{\tau(b,b)}{\tau(b,a)^2} w^4(a^{2i+1}b,a^{2j+1})\\
&=-\frac{\tau(b,b)}{\tau(b,a^{2j+1})^2} w^4(a^{2i+1}b,a^{2j+1}).
\end{align*}
Moreover, since
\begin{align*}
h(a^{2i+1}b,a^{2j+1}b) w^4(a^{2i+1},a^{2j+1})&=\frac{\tau(a^{2i+1}b,a^{2j+1}b)}{\tau(a^{2j+1},a^{2i+1})} w^4(a^{2i+1},a^{2j+1})\\
&=\frac{\tau(a^{2i+1}b,ba^{2j+1})}{\tau(a^{2j+1},a^{2i+1})} w^4(a^{2i+1},a^{2j+1})\\
&=\frac{\tau(a^{2i+1}b,ba^{2j+1}) \tau(b,a^{2j+1})}{\tau(a^{2j+1},a^{2i+1}) \tau(b,a^{2j+1})} w^4(a^{2i+1},a^{2j+1})\\
&=\frac{\tau(a^{2i+1},a^{2j+1}) \tau(a^{2i+1}b,b)}{\tau(a^{2j+1},a^{2i+1}) \tau(b,a^{2j+1})} w^4(a^{2i+1},a^{2j+1})\\
&=\eta(a^{2i+1},a^{2j+1}) \frac{\tau(a^{2i+1}b,b)}{\tau(b,a^{2j+1})} w^4(a^{2i+1},a^{2j+1})\\
&=\frac{\tau(a^{2i+1}b,b)}{\tau(b,a^{2j+1})} w^4(a^{2i+1},a^{2j+1})\\
&=\frac{\tau(a^{2i+1}b,b) \tau(a^{2j+1},b)}{\tau(b,a^{2j+1}) \tau(a^{2j+1},b)} w^4(a^{2i+1},a^{2j+1})\\
&=\frac{\tau(b,b)}{\tau(b,a^{2j+1}) \tau(a^{2j+1},b)} w^4(a^{2i+1},a^{2j+1})\\
&=\frac{\tau(b,b)}{\tau(b,a^{2j+1}) (-\tau(b,a^{2j+1}))} w^4(a^{2i+1},a^{2j+1})\\
&=-\frac{\tau(b,b)}{\tau(b,a^{2j+1})^2} w^4(a^{2i+1},a^{2j+1}),
\end{align*}
we have
\begin{align*}
-h(a^{2i+1},a^{2j+1}b) \frac{S_{j,1}}{S_{j,0}} w^4(a^{2i+1}b,a^{2j+1})=h(a^{2i+1}b,a^{2j+1}b) w^4(a^{2i+1},a^{2j+1}).
\end{align*}
Therefore we have proved the following equation
\begin{align}
\label{light3} w^4(a^{2i+1},a^{2j+1}b) w^2(b,a^{2j+1})&=\tau(a^{2i+1},b) w^4(a^{2i+1}b,a^{2j+1}b).
\end{align}
Combining the equation \eqref{light1} together with the equation \eqref{light3}, the equation \eqref{w.6} is proved.

We will show the equation \eqref{w.7}. Since
\begin{align*}
[w^3(a,a^{2i})]^{2n}&=[\lambda_{2i,1}(S_{0,0} S_{0,1})^i]^{2n}=[\lambda_{2i,1}(\alpha \beta)^i]^{2n}\\
&=\lambda_{2i,1}^{2n} (\alpha \beta)^{2 i n}=\lambda_{2i,1}^{2n} [(\alpha \beta)^{n}]^{2i}\\
&=\lambda_{2i,1}^{2n} [\lambda_{2n,1}^{-1}]^{2i}=\lambda_{2i,1}^{2n} [P_{2n} \sigma(a)^{-n}]^{2i}\\
&=\lambda_{2i,1}^{2n} P_{2n}^{2i} \sigma(a)^{-2 i n}=(P_{2i}^{-1} \sigma(a)^i)^{2n} P_{2n}^{2i} \sigma(a)^{-2 i n}\\
&=P_{2i}^{-2n} \sigma(a)^{2 i n} P_{2n}^{2i} \sigma(a)^{-2 i n}=P_{2i}^{-2n} P_{2n}^{2i}\\
&=\frac{\sigma(a^{2n})^i}{\sigma(a^{2i})^n} \;\;(\text{by}\;\eqref{4.b1})\\
&=\frac{1}{\sigma(a^{2i})^n},
\end{align*}
we have $[w^3(a,a^{2i})]^{2n}=\frac{1}{\sigma(a^{2i})^n}$ and therefore
\begin{align}
\label{light4}[w^3(a,a^{2i})]^{2n} \sigma(a^{2i})^n=1.
\end{align}
Since
\begin{align*}
[w^3(a,a^{2i}b)]^{2n}&=[\frac{-\tau(b,a)\beta }{\tau(b,a^{2i}) \alpha}\lambda_{2i,1}(S_{0,0} S_{0,1})^i]^{2n}\\
&=[\frac{-\tau(b,a)\beta }{\tau(b,a^{2i}) \alpha}]^{2n} [\lambda_{2i,1}(S_{0,0} S_{0,1})^i]^{2n}\\
&=[\frac{-\tau(b,a)}{\tau(b,a^{2i})}]^{2n} [\frac{\tau(b,b)}{\tau(b,a)^2}]^{n} [\lambda_{2i,1}(S_{0,0} S_{0,1})^i]^{2n}\\
&=[\frac{\tau(b,b)}{\tau(b,a^{2i})^2}]^{n} [\lambda_{2i,1}(S_{0,0} S_{0,1})^i]^{2n}\\
&=[\frac{\tau(b,b)}{\tau(b,a^{2i})^2}]^{n} w^3(a,a^{2i})^{2n},
\end{align*}
we have
\begin{align*}
[w^3(a,a^{2i}b)]^{2n} \sigma(a^{2i}b)^n&=[\frac{\tau(b,b)}{\tau(b,a^{2i})^2}]^{n} w^3(a,a^{2i})^{2n} \sigma(a^{2i}b)^n\\
&=[\frac{\tau(b,b)}{\tau(b,a^{2i})^2}]^{n} \sigma(a^{2i})^{-n} \sigma(a^{2i}b)^n\;\;(\text{by}\;\eqref{light4})\\
&=\tau(b,b)^n [\frac{\sigma(a^{2i}b)}{\sigma(a^{2i}) \tau(b,a^{2i})^2}]^n\\
&=\tau(b,b)^n \sigma(b)^n\\
&=[\tau(b,b) \sigma(b)]^n.
\end{align*}
To show that $[w^3(a,a^{2i}b)]^{2n} \sigma(a^{2i}b)^n=1$, we need only to prove that $\tau(b,b) \sigma(b)=1$. Since $\sigma(ab)\sigma(a)^{-1}\sigma(b)^{-1}=\tau(a,b)\tau(ab,b)$ and $\sigma(a)=\sigma(a\triangleleft x)=\sigma(ab)$, we have $\sigma(b)^{-1}=\tau(a,b)\tau(ab,b)$. Thanks to $\tau$ being a 2-cocycle, we have $\tau(a,b)\tau(ab,b)=\tau(b,b)\tau(a,1)=\tau(b,b)$ and thus
\begin{align}
\label{sig}\sigma(b)^{-1}=\tau(b,b).
\end{align}
Therefore
\begin{align}
\label{light5}[w^3(a,a^{2i}b)]^{2n} \sigma(a^{2i}b)^n=1.
\end{align}
 Combining the equation \eqref{light4} and the equation \eqref{light5} we get the first part of \eqref{w.7} by nothing that $S=\{a^{2i},a^{2i}b\;|\;i\geq 0\}$. Since by definition $w^1(b,a^{2i})=1$ and $w^1(b,a^{2i}b)=-1$, we get the rest of \eqref{w.7}.

 The proof the equation \eqref{w.8} is easy. Since by definition $w^1(b,a^{2i})=1$, $w^1(b,a^{2i}b)=-1$, we have
 $w^3(a,a^{2i}) w^1(b,a^{2i})=w^3(a,a^{2i}) $
\begin{align*}
&w^3(a,a^{2i}) w^1(b,a^{2i})=w^3(a,a^{2i})= w^3(ab,a^{2i}),\\
&w^3(a,a^{2i}b) w^1(b,a^{2i}b)=-w^3(a,a^{2i})(-1)=w^3(ab,a^{2i}b).
\end{align*}
Therefore, we get the equation \eqref{w.8} by $S=\{a^{2i},a^{2i}b\;|\;i\geq 0\}$.

Now let's prove the equation \eqref{w.9}. Since
\begin{align*}
w^3(a,a^{2j})^{2i} \sigma(a^{2j})^i&=[\lambda_{2j,1} (S_{0,0} S_{0,1})^j]^{2i} \sigma(a^{2j})^i=[\lambda_{2j,1} (\alpha\beta)^j]^{2i} \sigma(a^{2j})^i\\
&=\lambda_{2j,1}^{2i} (\alpha\beta)^{2ij} \sigma(a^{2j})^i=w^1(a^{2i},a^{2j}),
\end{align*}
we have
\begin{align}
\label{light7} w^3(a,a^{2j})^{2i} \sigma(a^{2j})^i=w^1(a^{2i},a^{2j}).
\end{align}
Since
\begin{align*}
w^3(a,a^{2j}b)^{2i} \sigma(a^{2j}b)^i&=[-\frac{\tau(b,a)\beta}{\tau(b,a^{2j})\alpha}\lambda_{2j,1} (S_{0,0} S_{0,1})^j]^{2i} \sigma(a^{2j}b)^i\\
&=[(-\frac{\tau(b,a) \beta}{\tau(b,a^{2j}) \alpha}) w^3(a,a^{2j})]^{2i} \sigma(a^{2j}b)^i\\
&=[-\frac{\tau(b,a) \beta}{\tau(b,a^{2j}) \alpha}]^{2i} w^3(a,a^{2j})^{2i} \sigma(a^{2j}b)^i\\
&=[-\frac{\tau(b,a)\beta}{\tau(b,a^{2j})\alpha}]^{2i} w^3(a,a^{2j})^{2i} \sigma(a^{2j})^i \frac{\sigma(a^{2j}b)^i}{\sigma(a^{2j})^i}\\
&=[-\frac{\tau(b,a) \beta}{\tau(b,a^{2j}) \alpha}]^{2i} w^1(a^{2i},a^{2j}) \frac{\sigma(a^{2j}b)^i}{\sigma(a^{2j})^i}\\
&=[-\frac{\tau(b,a) \beta}{\tau(b,a^{2j}) \alpha}]^{2i} \frac{\sigma(a^{2j}b)^i}{\sigma(a^{2j})^i} w^1(a^{2i},a^{2j})\\
&=[-\frac{\tau(b,a) \beta}{\tau(b,a^{2j}) \alpha}]^{2i} \frac{\sigma(a^{2j}b)^i}{\sigma(a^{2j})^i} w^1(a^{2i},a^{2j}b),
\end{align*}
we have $w^3(a,a^{2j}b)^{2i} \sigma(a^{2j}b)^i=[-\frac{\tau(b,a) \beta}{\tau(b,a^{2j}) \alpha}]^{2i} \frac{\sigma(a^{2j}b)^i}{\sigma(a^{2j})^i} w^1(a^{2i},a^{2j}b)$, and thus we need only to
prove $[-\frac{\tau(b,a) \beta}{\tau(b,a^{2j}) \alpha}]^{2i} \frac{\sigma(a^{2j}b)^i}{\sigma(a^{2j})^i}=1$ if we want to show that $w^3(a,a^{2j}b)^{2i} \sigma(a^{2j}b)^i=w^1(a^{2i},a^{2j}b)$. In fact, we have
\begin{align*}
[\frac{\tau(b,a) \beta}{\tau(b,a^{2j}) \alpha}]^2 \frac{\sigma(a^{2j}b)}{\sigma(a^{2j})}&=\frac{\tau(b,a)^2}{\tau(b,a^{2j})^2} \frac{\tau(b,b)}{\tau(b,a)^2} \frac{\sigma(a^{2j}b)}{\sigma(a^{2j})}\\
&=\frac{\tau(b,b)}{\tau(b,a^{2j})^2} \frac{\sigma(a^{2j}b)}{\sigma(a^{2j})}\\
&=\tau(b,b) \frac{\sigma(a^{2j}b)}{\sigma(a^{2j}) \tau(b,a^{2j})^2}\\
&=\tau(b,b) \sigma(b)\\
&=1,\;\;(\text{by }\;\eqref{sig}).
\end{align*}
Therefore
\begin{align}
\label{light6} w^3(a,a^{2j}b)^{2i} \sigma(a^{2j}b)^i=w^1(a^{2i},a^{2j}b).
\end{align}
 Combining the equation \eqref{light7} and the equation \eqref{light6}, we get that
\begin{align}
\label{light8}w^3(a,s)^{2i} \sigma(s)^i=w^1(a^{2i},s)
\end{align}
for $s\in S.$ By
\begin{align*}
&w^1(a^{2i},a^{2j}) w^1(b,a^{2j})=w^1(a^{2i},a^{2j})=w^1(a^{2i},a^{2j})=w^1(a^{2i}b,a^{2j}),\\
&w^1(a^{2i},a^{2j}b) w^1(b,a^{2j}b)=-w^1(a^{2i},a^{2j}b)=w^1(a^{2i}b,a^{2j}b),
\end{align*}
 we know that
\begin{align}
\label{light9}w^1(a^{2i},s) w^1(b,s)=w^1(a^{2i}b,s).
\end{align}
Since the equation \eqref{light8} and the equation \eqref{light9} hold, the proof of  \eqref{w.9} is done.

Now we turn to the proof of the last equation \eqref{w.10} hold. Direct computations show that
\begin{align*}
w^3(a,a^{2j})^{2i+1} \sigma(a^{2j})^i&=[\lambda_{2j,1} (S_{0,0} S_{0,1})^j]^{2i+1} \sigma(a^{2j})^i\\
&=[P_{2j}^{-1} \sigma(a)^j (S_{0,0} S_{0,1})^j]^{2i+1} \sigma(a^{2j})^i\\
&=[P_{2j}^{-1} \sigma(a)^j (\alpha \beta)^j]^{2i+1} \sigma(a^{2j})^i\\
&=P_{2j}^{-(2i+1)} \sigma(a)^{j(2i+1)} (\alpha \beta)^{j(2i+1)} \sigma(a^{2j})^i\\
&=[P_{2j}^{-(2i+1)} \sigma(a)^{j} \sigma(a^{2j})^i] [\sigma(a)^{2ij}(\alpha \beta)^{j(2i+1)}].
\end{align*}
Since
\begin{align*}
w^3(a^{2i+1},a^{2j})&=\lambda_{2j,2i+1}[S_{i,0}S_{i,1}]^j\\
&=P_{2j}^{-1} \sigma(a^{2i+1})^j [S_{i,0}S_{i,1}]^j\\
\end{align*}
and
\begin{align*}
S_{i,0}S_{i,1}&=(\lambda_{2i+1,1} \alpha^{i+1} \beta^{i}) S_{i,1}\\
&=(\lambda_{2i+1,1} \alpha^{i+1} \beta^{i}) (h(a,a^{2i+1}b) \lambda_{2i+1,1} \alpha^{i} \beta^{i+1})\\
&=h(a,a^{2i+1}b) \lambda_{2i+1,1}^2 \alpha^{2i+1} \beta^{2i+1}\\
&=h(a,a^{2i+1}b) [P_{2i+1}^{-1} \sigma(a)^i]^2 \alpha^{2i+1} \beta^{2i+1}\\
&=h(a,a^{2i+1}b) P_{2i+1}^{-2} \sigma(a)^{2i} \alpha^{2i+1} \beta^{2i+1},
\end{align*}
we have
\begin{align*}
w^3(a^{2i+1},a^{2j})=P_{2j}^{-1} \sigma(a^{2i+1})^j h(a,a^{2i+1}b)^j P_{2i+1}^{-2j} [\sigma(a)^{2ij}(\alpha \beta)^{j(2i+1)}].
\end{align*}
Since $P_{2i+1}^{2j}\sigma(a)^j\sigma(a^{2j})^i=P_{2j}^{2i}\sigma(a^{2i+1})^j
[\frac{\tau(b,a)}{\tau(b,a^{2i+1})}]^j$  and $h(a,a^{2i+1}b)=\frac{\tau(b,a)}{\tau(b,a^{2i+1})}$ by Lemma \ref{lem4.1}, we have $P_{2j}^{-(2i+1)} \sigma(a)^{j} \sigma(a^{2j})^i=P_{2j}^{-1} \sigma(a^{2i+1})^j h(a,a^{2i+1}b)^j P_{2i+1}^{-2j}$ and therefore
\begin{align}
\label{light10} w^3(a,a^{2j})^{2i+1} \sigma(a^{2j})^i=w^3(a^{2i+1},a^{2j}).
\end{align}
Since
\begin{align*}
w^3(a,a^{2j}b)&=\frac{-\tau(b,a)\beta}{\tau(b,a^{2j}) \alpha} \lambda_{2j,1} (S_{0,0} S_{0,1})^j=\frac{-\tau(b,a)\beta}{\tau(b,a^{2j})\alpha} w^3(a,a^{2j})
\end{align*}
and
\begin{align*}
\sigma(a^{2j}b)&=\sigma(a^{2j}) \frac{\sigma(a^{2j}b)}{\sigma(a^{2j})},
\end{align*}
we have
\begin{align*}
w^3(a,a^{2j}b)^{2i+1} \sigma(a^{2j}b)^i&=[\frac{-\tau(b,a)\beta}{\tau(b,a^{2j}) \alpha} w^3(a,a^{2j})]^{2i+1} [\sigma(a^{2j}) \frac{\sigma(a^{2j}b)}{\sigma(a^{2j})}]^{i}\\
&=(\frac{-\tau(b,a)\beta}{\tau(b,a^{2j}) \alpha})^{2i+1} [w^3(a,a^{2j})^{2i+1} \sigma(a^{2j})^i] [\frac{\sigma(a^{2j}b)}{\sigma(a^{2j})}]^{i}\\
&=(\frac{-\tau(b,a)\beta}{\tau(b,a^{2j}) \alpha})^{2i+1} w^3(a^{2i+1},a^{2j}) [\frac{\sigma(a^{2j}b)}{\sigma(a^{2j})}]^{i}\\
&=(\frac{-\tau(b,a)\beta}{\tau(b,a^{2j}) \alpha})^{2i} [\frac{-\tau(b,a)\beta}{\tau(b,a^{2j}) \alpha} w^3(a^{2i+1},a^{2j})] [\frac{\sigma(a^{2j}b)}{\sigma(a^{2j})}]^{i}\\
&=(\frac{-\tau(b,a)\beta}{\tau(b,a^{2j}) \alpha})^{2i} w^3(a^{2i+1},a^{2j}b) [\frac{\sigma(a^{2j}b)}{\sigma(a^{2j})}]^{i}\\
&=(\frac{-\tau(b,a)\beta}{\tau(b,a^{2j}) \alpha})^{2i}  [\frac{\sigma(a^{2j}b)}{\sigma(a^{2j})}]^{i} w^3(a^{2i+1},a^{2j}b)\\
&=w^3(a^{2i+1},a^{2j}b)\;(\text{by} \;[\frac{\tau(b,a)\beta}{\tau(b,a^{2j})\alpha}]^2=
\frac{\sigma(a^{2j})}{\sigma(a^{2j}b)}\;\text{in Lemma}\; \ref{lem4.1}).
\end{align*}
Therefore we have
\begin{align}
\label{light11} w^3(a,a^{2j}b)^{2i+1} \sigma(a^{2j}b)^i=w^3(a^{2i+1},a^{2j}b).
\end{align}
By \eqref{light10}, \eqref{light11} and $S=\{a^{2i},a^{2i}b\;|\;i\geq 0\}$,  we have
\begin{align}
\label{light12} w^3(a,s)^{2i+1} \sigma(s)^i=w^3(a^{2i+1},s).
\end{align}
Since
\begin{align*}
&w^3(a^{2i+1},a^{2j}) w^1(b,a^{2j})=w^3(a^{2i+1},a^{2j})=w^3(a^{2i+1}b,a^{2j}),\\
&w^3(a^{2i+1},a^{2j}b) w^1(b,a^{2j}b)=-w^3(a^{2i+1},a^{2j}b)=w^3(a^{2i+1}b,a^{2j}b)
\end{align*}
and $S=\{a^{2i},a^{2i}b\;|\;i\geq 0\}$, we have
\begin{align}
\label{light13} w^3(a^{2i+1},s) w^1(b,s)=w^3(a^{2i+1}b,s).
\end{align}
 Combining the equation \eqref{light12} together with the equation \eqref{light13}, we prove the last equation \eqref{w.10}.
\end{proof}

The following Lemma \ref{lemm4.5} and Lemma \ref{lemm4.1} are prepared to prove Lemma \ref{lemm4.6}.

\begin{lemma}\label{lemm4.5}
Let $R_{\alpha,\beta}$ as above, then
\begin{gather}
\label{el.1} l(X_a)^{2n}=P_{2n}l(X_1),\;l(X_b)^2=\tau(b,b)l(X_1),\\
\label{el.2} l(X_a)l(X_b)=\tau(a,b)l(X_{ab}),\;l(X_b)l(X_a)=\tau(b,a)l(X_{ab}),\\
\label{el.3} l(X_a)^{2i}=P_{2i}l(X_{a^{2i}}),\; l(X_{a^{2i}})l(X_b)=\tau(a^{2i},b)l(X_{a^{2i}b}),\\
\label{el.4} l(X_a)^{2i+1}=P_{2i+1}l(X_{a^{2i+1}}),\; l(X_{a^{2i+1}})l(X_b)=\tau(a^{2i+1},b)l(X_{a^{2i+1}b}),\\
\label{el.5} l(X_a)l(X_1)=l(X_a),\;l(X_b)l(X_1)=l(X_b),\\
\label{el.6} l(E_a)^{2n}=l(E_1),\;l(E_b)^2=l(E_1),\\
\label{el.7} l(E_a)l(E_b)=l(E_{ab}),\; l(E_b)l(E_a)=l(E_{ab}),\\
\label{el.8} l(E_a)^{2i}=l(E_{a^{2i}}),\; l(E_{a^{2i}})l(E_b)=l(E_{a^{2i}b}),\\
\label{el.9} l(E_a)^{2i+1}=l(E_{a^{2i+1}}),\; l(E_{a^{2i+1}})l(E_b)=l(E_{a^{2i+1}b}),\\
\label{el.10} l(E_a)l(E_1)=l(E_a),\;l(E_b)l(E_1)=l(E_b)\\
\label{el.11} l(X_1)l(E_1)=0,\;l(E_1)l(X_1)=0.
\end{gather}
\end{lemma}
\begin{proof}
We will show the equation \eqref{el.1}. Since
\begin{align*}
l(X_a)^{2n}&=[\sum\limits_{t \in T}w^4(a,t)e_t x]^{2n}=[(\sum\limits_{t \in T}w^4(a,t)e_t x)^2]^{n}\\
&=[\sum\limits_{t \in T}w^4(a,t)w^4(a,t\triangleleft x)e_t x^2]^{n}=[\sum\limits_{t \in T}w^4(a,t)w^4(a,t\triangleleft x)\sigma(t)e_t]^{n}\\
&=\sum\limits_{t \in T}[w^4(a,t)w^4(a,t\triangleleft x)\sigma(t)]^n e_t\\
&=\sum\limits_{t \in T} P_{2n}e_t\;\;(\text{by}\;\eqref{w.1})
\end{align*}
and
\begin{align*}
P_{2n}l(X_1)&=P_{2n}\sum\limits_{t \in T}w^2(1,t)e_t=\sum\limits_{t \in T}P_{2n}e_t,
\end{align*}
we have $l(X_a)^{2n}=P_{2n}l(X_1)$. Since
\begin{align*}
l(X_b)^{2}&=[\sum\limits_{t \in T}w^2(b,t)e_t]^{2}=\sum\limits_{t \in T}w^2(b,t)^2e_t\\
&=\sum\limits_{t \in T}\tau(b,b)e_t \;\;(\text{by}\;\eqref{w.1})
\end{align*}
and
\begin{align*}
\tau(b,b)l(X_1)&=\tau(b,b)\sum\limits_{t \in T}w^2(1,t)e_t=\sum\limits_{t \in T}\tau(b,b)e_t,
\end{align*}
 we have $l(X_b)^{2}=\tau(b,b)l(X_1)$. This means that the equation \eqref{el.1} hold.

Next, we want to show the equation \eqref{el.2}. Since
\begin{align*}
l(X_a)l(X_b)&=(\sum\limits_{t \in T}w^4(a,t)e_t x)l(X_b)=(\sum\limits_{t \in T}w^4(a,t)e_t x) (\sum\limits_{t \in T}w^2(b,t)e_t)\\
&=\sum\limits_{t \in T}w^4(a,t) w^2(b,t\triangleleft x)e_t x\\
&=\sum\limits_{t \in T}\tau(a,b) w^4(ab,t)e_t x \;\;(\text{by}\;\eqref{w.2})
\end{align*}
and
\begin{align*}
\tau(a,b)l(X_{ab})&=\tau(a,b) \sum\limits_{t \in T}w^4(ab,t)e_t x=\sum\limits_{t \in T}\tau(a,b)w^4(ab,t)e_t x,
\end{align*}
we have $l(X_a)l(X_b)=\tau(a,b)l(X_{ab})$. By
\begin{align*}
l(X_b)l(X_a)&=(\sum\limits_{t \in T}w^2(b,t)e_t)l(X_a)=(\sum\limits_{t \in T}w^2(b,t)e_t)(\sum\limits_{t \in T}w^4(a,t)e_t x)\\
&=\sum\limits_{t \in T}w^2(b,t) w^4(a,t)e_t x\\
&=\sum\limits_{t \in T}\tau(b,a) w^4(ab,t)e_t x  \;\;(\text{by}\;\eqref{w.2})
\end{align*}
and
\begin{align*}
\tau(a,b)l(X_{ab})&=\tau(a,b) \sum\limits_{t \in T}w^4(ab,t)e_t x\\&=\sum\limits_{t \in T}\tau(b,a) w^4(ab,t)e_t x \;\;(\text{by}\;\eqref{w.2}),
\end{align*}
 we have $l(X_b)l(X_a)=\tau(b,a)l(X_{ab})$ and the equation \eqref{el.2} hold.

We turn to the proof of the equation \eqref{el.3}. Due to
\begin{align*}
l(X_a)^{2i}&=[\sum\limits_{t \in T}w^4(a,t)e_t x]^{2i}\\
&=[(\sum\limits_{t \in T}w^4(a,t)e_t x)^2]^{i}=[\sum\limits_{t \in T}w^4(a,t) w^4(a,t\triangleleft x) e_t x^2]^{i}\\
&=[\sum\limits_{t \in T}w^4(a,t) w^4(a,t\triangleleft x) e_t \sigma(t)]^{i}=\sum\limits_{t \in T}[w^4(a,t) w^4(a,t\triangleleft x)\sigma(t)]^i e_t\\
&=\sum\limits_{t \in T} P_{2i}w^2(a^{2i},t)e_t \;\;(\text{by}\;\eqref{w.3})
\end{align*}
and
\begin{align*}
P_{2i}l(X_{a^{2i}})&=P_{2i}\sum\limits_{t \in T} w^2(a^{2i},t)e_t=\sum\limits_{t \in T} P_{2i}w^2(a^{2i},t)e_t,
\end{align*}
we have $l(X_a)^{2i}=P_{2i}l(X_{a^{2i}})$. Since
\begin{align*}
l(X_{a^{2i}})l(X_b)&=(\sum\limits_{t \in T}w^2(a^{2i},t)e_t) l(X_b)=(\sum\limits_{t \in T}w^2(a^{2i},t)e_t) (\sum\limits_{t \in T}w^2(b,t)e_t)\\
&=\sum\limits_{t \in T}w^2(a^{2i},t) w^2(b,t) e_t\\
&=\sum\limits_{t \in T}\tau(a^{2i},b) w^2(a^{2i}b,t) e_t \;\;(\text{by}\;\eqref{w.4})
\end{align*}
and
\begin{align*}
\tau(a^{2i},b) l(X_{a^{2i}b})&=\tau(a^{2i},b) \sum\limits_{t \in T}w^2(a^{2i}b,t) e_t\\
&=\sum\limits_{t \in T}\tau(a^{2i},b) w^2(a^{2i}b,t) e_t,
\end{align*}
we have $l(X_{a^{2i}})l(X_b)=\tau(a^{2i},b) l(X_{a^{2i}b})$ and thus \eqref{el.3} holds.

For the equation \eqref{el.4}, we find that
\begin{align*}
l(X_a)^{2i+1}&=[\sum\limits_{t \in T}w^4(a,t)e_t x]^{2i+1}=[(\sum\limits_{t \in T}w^4(a,t)e_t x)^2]^{i} [\sum\limits_{t \in T}w^4(a,t)e_t x]\\
&=[(\sum\limits_{t \in T}w^4(a,t) w^4(a,t\triangleleft x) e_t x^2)]^{i} [\sum\limits_{t \in T}w^4(a,t)e_t x]\\
&=[\sum\limits_{t \in T}w^4(a,t) w^4(a,t\triangleleft x) \sigma(t) e_t]^{i} [\sum\limits_{t \in T}w^4(a,t)e_t x]\\
&=[\sum\limits_{t \in T}w^4(a,t)^i w^4(a,t\triangleleft x)^i \sigma(t)^i e_t] [\sum\limits_{t \in T}w^4(a,t)e_t x]\\
&=\sum\limits_{t \in T}w^4(a,t)^{i+1} w^4(a,t\triangleleft x)^i \sigma(t)^i e_t x\\
&=\sum\limits_{t \in T}P_{2i+1} w^4(a^{2i+1},t)e_t x \;\;(\text{by}\;\eqref{w.5})
\end{align*}
and
\begin{align*}
P_{2i+1}l(X_{a^{2i+1}})&=P_{2i+1}\sum\limits_{t \in T} w^4(a^{2i+1},t)e_t x=\sum\limits_{t \in T} P_{2i+1} w^4(a^{2i+1},t)e_t x.
\end{align*}
Therefore, $l(X_a)^{2i+1}=P_{2i+1}l(X_{a^{2i+1}})$. By
\begin{align*}
l(X_{a^{2i+1}})l(X_b)&=(\sum\limits_{t \in T}w^4(a^{2i+1},t)e_t x) l(X_b)=(\sum\limits_{t \in T}w^4(a^{2i+1},t)e_t x) (\sum\limits_{t \in T}w^2(b,t)e_t)\\
&=\sum\limits_{t \in T}w^4(a^{2i+1},t) w^2(b,t \triangleleft x) e_t x\\
&=\sum\limits_{t \in T}\tau(a^{2i+1},b) w^4(a^{2i+1}b,t) e_t \;\;(\text{by}\;\eqref{w.6})
\end{align*}
and
\begin{align*}
\tau(a^{2i+1},b) l(X_{a^{2i+1}b})&=\tau(a^{2i+1},b) \sum\limits_{t \in T}w^4(a^{2i+1}b,t) e_t\\
&=\sum\limits_{t \in T}\tau(a^{2i+1},b) w^4(a^{2i+1}b,t) e_t,
\end{align*}
we have $l(X_{a^{2i+1}})l(X_b)=\tau(a^{2i+1},b) l(X_{a^{2i+1}b})$ and the equation \eqref{el.4} holds.

Now we prove the equation \eqref{el.5}. From
\begin{align*}
l(X_a)l(X_1)&=(\sum\limits_{t \in T}w^4(a,t)e_t x) l(X_1)=(\sum\limits_{t \in T}w^4(a,t)e_t x) (\sum\limits_{t \in T}w^2(1,t)e_t)\\
&=(\sum\limits_{t \in T}w^4(a,t)e_t x) (\sum\limits_{t \in T}e_t)=\sum\limits_{t \in T}w^4(a,t)e_t x\\
&=l(X_a),
\end{align*}
we have $l(X_a)l(X_1)=l(X_a)$. Due to
\begin{align*}
l(X_b)l(X_1)&=(\sum\limits_{t \in T}w^2(b,t)e_t) l(X_1)=(\sum\limits_{t \in T}w^2(b,t)e_t) (\sum\limits_{t \in T}w^2(1,t)e_t)\\
&=(\sum\limits_{t \in T}w^2(b,t)e_t)=l(X_b),
\end{align*}
we have $l(X_b)l(X_1)=l(X_b)$. Thus we have the equation \eqref{el.5}.

Now we consider the equation \eqref{el.6}. Based on
\begin{align*}
l(E_a)^{2n}&=(\sum\limits_{s \in S}w^3(a,s)e_s x)^{2n}=[(\sum\limits_{s \in S}w^3(a,s)e_s x)^2]^{n}\\
&=[\sum\limits_{s \in S}w^3(a,s)^2 e_s x^2]^{n}=[\sum\limits_{s \in S}w^3(a,s)^2 \sigma(s) e_s]^{n}\\
&=\sum\limits_{s \in S}w^3(a,s)^{2n} \sigma(s)^n e_s\\
&=\sum\limits_{s \in S}e_s \;\;(\text{by}\;\eqref{w.7})
\end{align*}
and
\begin{align*}
l(E_1)&=\sum\limits_{s \in S} w^1(1,s)e_s=\sum\limits_{s \in S} e_s,
\end{align*}
we get that $l(E_{a})^{2n}=l(E_1)$. Owning to
\begin{align*}
l(E_b)^{2}&=[\sum\limits_{s \in S}w^1(b,s)e_s]^{2}=\sum\limits_{s \in S}w^1(b,s)^2e_s\\
&=\sum\limits_{s \in S}e_s \;\;(\text{by}\;\eqref{w.7})
\end{align*}
and
\begin{align*}
l(E_1)&=\sum\limits_{s \in S} w^1(1,s)e_s=\sum\limits_{s \in S} e_s,
\end{align*}
we have $l(E_{b})^{2}=l(E_1)$ and thus the equation \eqref{el.6} holds.

For the equation \eqref{el.7}, we find that
\begin{align*}
l(E_a)l(E_b)&=(\sum\limits_{s \in S}w^3(a,s)e_s x) l(E_b)=(\sum\limits_{s \in S}w^3(a,s)e_s x) (\sum\limits_{s \in S}w^1(b,s)e_s)\\
&=\sum\limits_{s \in S}w^3(a,s) w^1(b,s)e_s x\\
&=\sum\limits_{s \in S}w^3(ab,s) e_s x \;\;(\text{by}\;\eqref{w.8})\\
&=l(E_{ab}),
\end{align*}
and
\begin{align*}
l(E_b)l(E_a)&=(\sum\limits_{s \in S}w^1(b,s)e_s) l(E_a)=(\sum\limits_{s \in S}w^1(b,s)e_s) (\sum\limits_{s \in S}w^3(a,s)e_s x)\\
&=\sum\limits_{s \in S}w^1(b,s) w^3(a,s) e_s x\\
&=\sum\limits_{s \in S}w^3(ab,s) e_s x \;\;(\text{by}\;\eqref{w.8})\\
&=l(E_{ab}).
\end{align*}
We get the equation \eqref{el.7}.

For the equation \eqref{el.8}, direct computations show that
\begin{align*}
l(E_a)^{2i}&=(\sum\limits_{s \in S}w^3(a,s)e_s x)^{2i}=[(\sum\limits_{s \in S}w^3(a,s)e_s x)^2]^{i}\\
&=[\sum\limits_{s \in S}w^3(a,s)^2 e_s x^2]^{i}=[\sum\limits_{s \in S}w^3(a,s)^2 \sigma(s) e_s]^{i}\\
&=\sum\limits_{s \in S}w^3(a,s)^{2i} \sigma(s)^i e_s\\
&=\sum\limits_{s \in S}w^1(a^{2i},s)e_s \;\;(\text{by}\;\eqref{w.9})\\
&=l(E_{a^{2i}}),
\end{align*}
and
\begin{align*}
l(E_{a^{2i}})l(E_b)&=(\sum\limits_{s \in S}w^1(a^{2i},s)e_s x) l(E_b)=(\sum\limits_{s \in S}w^1(a^{2i},s)e_s) (\sum\limits_{s \in S}w^1(b,s)e_s)\\
&=\sum\limits_{s \in S}w^1(a^{2i},s) w^1(b,s)e_s\\
&=\sum\limits_{s \in S}w^1(a^{2i}b,s)e_s \;\;(\text{by}\;\eqref{w.9})\\
&=l(E_{a^{2i}b}).
\end{align*}
Therefore the equation \eqref{el.8} holds.

Now we are going to prove the equation \eqref{el.9}. By
\begin{align*}
l(E_a)^{2i+1}&=(\sum\limits_{s \in S}w^3(a,s)e_s x)^{2i+1}=[(\sum\limits_{s \in S}w^3(a,s)e_s x)^2]^{i} (\sum\limits_{s \in S}w^3(a,s)e_s x)\\
&=[\sum\limits_{s \in S}w^3(a,s)^2 e_s x^2]^{i} (\sum\limits_{s \in S}w^3(a,s)e_s x)\\
&=[\sum\limits_{s \in S}w^3(a,s)^{2}\sigma(s) e_s]^{i} (\sum\limits_{s \in S}w^3(a,s)e_s x)\\
&=[\sum\limits_{s \in S}w^3(a,s)^{2i}\sigma(s)^i e_s] (\sum\limits_{s \in S}w^3(a,s)e_s x)\\
&=\sum\limits_{s \in S}w^3(a,s)^{2i+1}\sigma(s)^i e_s x\\
&=\sum\limits_{s \in S}w^3(a^{2i+1},s) e_s x  \;\;(\text{by}\;\eqref{w.10}\; \text{of}\; \text{Lemma}\; \ref{lemm4.4})\\
&=l(E_{a^{2i+1}}),
\end{align*}
and
\begin{align*}
l(E_{a^{2i+1}})l(E_b)&=(\sum\limits_{s \in S}w^3(a^{2i+1},s)e_s x) l(E_b)\\
&=(\sum\limits_{s \in S}w^3(a^{2i+1},s)e_s x) (\sum\limits_{s \in S}w^1(b,s)e_s)\\
&=\sum\limits_{s \in S}w^3(a^{2i+1},s) w^1(b,s)e_s x\\
&=\sum\limits_{s \in S}w^3(a^{2i+1}b,s)e_s x \;\;(\text{by}\;\eqref{w.10}\; \text{of}\; \text{Lemma}\; \ref{lemm4.4})\\
&=l(E_{a^{2i+1}b}),
\end{align*}
we get the desired equation \eqref{el.9}.

For the equation \eqref{el.10}, we have
\begin{align*}
l(E_a)l(E_1)&=(\sum\limits_{s \in S}w^3(a,s)e_s x) l(E_1)=(\sum\limits_{s \in S}w^3(a,s)e_s x) (\sum\limits_{s \in S}w^1(1,s)e_s)\\
&=(\sum\limits_{s \in S}w^3(a,s)e_s x) (\sum\limits_{s \in S}e_s)=\sum\limits_{s \in S}w^3(a,s)e_s x\\
&=l(E_a)
\end{align*}
and
\begin{align*}
l(E_b)l(E_1)&=(\sum\limits_{s \in S}w^3(a,s)e_s x) l(E_1)=(\sum\limits_{s \in S}w^1(b,s)e_s x) (\sum\limits_{s \in S}w^1(1,s)e_s)\\
&=(\sum\limits_{s \in S}w^1(b,s)e_s x) (\sum\limits_{s \in S}e_s)=\sum\limits_{s \in S}w^1(b,s)e_s x\\
&=l(E_b).
\end{align*}
This implies the equation \eqref{el.10}.

For the last equation \eqref{el.11}, we find that
\begin{align*}
l(X_1)l(E_1)&=(\sum\limits_{s \in S}w^2(1,s)e_t) l(E_1)\\
&=(\sum\limits_{s \in S}w^2(1,s)e_t) (\sum\limits_{s \in S}w^1(1,s)e_s)\\
&=0
\end{align*}
and
\begin{align*}
l(E_1)l(X_1)&=(\sum\limits_{s \in S}w^1(1,s)e_s) l(X_1)\\
&=(\sum\limits_{s \in S}w^1(1,s)e_s) (\sum\limits_{s \in S}w^2(1,s)e_t)\\
&=0.
\end{align*}
This implies that the equation \eqref{el.11} holds.
\end{proof}

Let $K(8n,\sigma,\tau)$ as before, we associate a free object with it as follows. Recall that the data $G$ of $K(8n,\sigma,\tau)$ is $\langle a,b|a^{2n}=b^2=1,ab=ba\rangle$. We define $A_G$ as a
free $\Bbbk$ algebra generated by set $\{x_1,x_a,x_b,e_1,e_a,e_b\}$, and let $I_G$ be the ideal generated by $\{x_a^{2n}-\Pi_{i=1}^{2n-1}\tau(a,a^i)x_1,x_b^2-\tau(b,b)x_1,x_bx_a-\eta(a,b)x_ax_b,x_ax_1-x_a,x_bx_1-x_b,e_a^{2n}-e_1,
e_b^2-e_1,e_be_a-e_ae_b,e_ae_1-e_a,e_be_1-e_b,x_1e_1,e_1x_1\}$. Then we have the following lemma.
\begin{lemma}\label{lemm4.1}
Denote the dual Hopf algebra of $K(8n,\sigma,\tau)$ by $H^*$, then $H^*\cong A_G/I_G$ as an algebra.
\end{lemma}

\begin{proof}
Since Lemma \ref{lem2.2.2}, we have
\begin{gather*}
X_a^{2n}=\Pi_{i=1}^{2n-1}\tau(a,a^i)X_1,\;X_b^2=\tau(b,b)X_1,\\
X_bX_a=\eta(a,b)X_aX_b,\;X_a X_1=X_a,\;X_bX_1=X_b,\\
E_a^{2n}=E_1,\;E_b^2=E_1,\;E_bE_a=E_aE_b,\\
E_a E_1=E_a,\;E_b E_1=E_b,\;X_1 E_1=0,\;E_1X_1=0.
\end{gather*}
then we can define an algebra map $\pi:A_G/I_G\rightarrow H^*$ by setting
\begin{gather*}
\pi(x_1)=X_1,\; \pi(x_a)=X_a,\; \pi(x_b)=X_b,\; \pi(e_1)=E_1,\; \pi(e_a)=E_a,\; \pi(e_b)=E_b,
\end{gather*}
by using the definition of $A_G/I_G$, we will show that $\pi$ is bijective.

Firstly we show that $\pi$ is surjective. Since definition of $H^*$, it is linear spanned by $\{X_g,E_g\;|\;g\in G\}$. By Lemma \ref{lem2.2.2}, $\{X_g,E_g\;|\;g\in G\}$ are contained in the linear space spanned by $\{X_a^iX_b^j,E_a^iE_b^j\;|\;i,j\in \mathbb{N}\}$. Therefore $H^*$ is generated by $\{X_a,X_b,E_a,E_b\}$ as an algebra. Since $\pi$ is an algebra map and $\{X_a,X_b,E_a,E_b\}\subseteq \textrm{Im}\pi$, we know $\pi$ is surjective.

Secondly we show that $\pi$ is injective. Note that $x_g x_1=x_1x_g=x_g$ and $e_g e_1=e_1e_g=e_g$ for all $g\in G$, we have $e_gx_h=(e_g e_1)(x_1 x_h)=e_g(e_1 x_1)x_h=0$ and $x_he_g=( x_h x_1)(e_1 e_g )=x_h (x_1 e_1) e_g =0$. Then we can see that $A_G/I_G$ is linear spanned by $\{X_a^iX_b^j,E_a^iE_b^j\;|\;1\leq i \leq 2n, 1\leq j \leq 2\}$ and as a result we know that $\text{dim}(A_G/I_G)\leq 8n$. Since $\text{dim}(H^*)=8n$, we have $\text{dim}(A_G/I_G)\leq \text{dim}(H^*)$. But we have proved that $\pi$ is surjective and thus $\text{dim}(A_G/I_G)\geq \text{dim}(H^*)$. Then we have $\text{dim}(A_G/I_G)= \text{dim}(H^*)$. Since $\pi$ is surjective, we know that $\pi$  is injective.
\end{proof}

We use the following Lemmas \ref{lemm4.6} and  \ref{lemm4.8} to show that $R_{\alpha,\beta}$ satisfies the last equation of Lemma \ref{lem3.1}.
\begin{lemma}\label{lemm4.6}
Let $R_{\alpha,\beta}$ as above, then $l(f_1)l(f_2)=l(f_1 f_2)$ for $f_1,f_2\in H^*$ where $H^*$ is the dual of $K(8n,\sigma,\tau)$.
\end{lemma}
\begin{proof}
Due to Lemmas \ref{lemm4.5} and \ref{lemm4.1}, the following map is an algebra map:
 \begin{align*}\pi\colon H^*\rightarrow H,\;\;&X_1\mapsto l(X_1),\;X_a\mapsto l(X_a),\; X_b\mapsto l(X_b),\; \\
 & E_1\mapsto l(E_1),\; E_a\mapsto l(E_a),\; E_b\mapsto l(E_b).\end{align*}
To show that $l(f_1)l(f_2)=l(f_1 f_2)$ for $f_1,f_2\in H^*$, we only need to show that $\pi(f)=l(f)$ for all $f\in H^*$. By Lemma \ref{lemm4.5}, we know that
\begin{gather*}
l(X_a)^{2i}=P_{2i}l(X_{a^{2i}}),\; l(X_{a^{2i}})l(X_b)=\tau(a^{2i},b)l(X_{a^{2i}b}),\\
l(X_a)^{2i+1}=P_{2i+1}l(X_{a^{2i+1}}),\; l(X_{a^{2i+1}})l(X_b)=\tau(a^{2i+1},b)l(X_{a^{2i+1}b}),\\
l(E_a)^{2i}=l(E_{a^{2i}}),\; l(E_{a^{2i}})l(E_b)=l(E_{a^{2i}b}),\\
l(E_a)^{2i+1}=l(E_{a^{2i+1}}),\; l(E_{a^{2i+1}})l(E_b)=l(E_{a^{2i+1}b}).
\end{gather*}
Due to $\pi$ is an algebra map, we have
\begin{gather*}
\pi(X_a)^{2i}=P_{2i}\pi(X_{a^{2i}}),\; \pi(X_{a^{2i}})\pi(X_b)=\tau(a^{2i},b)\pi(X_{a^{2i}b}),\\
\pi(X_a)^{2i+1}=P_{2i+1}\pi(X_{a^{2i+1}}),\; \pi(X_{a^{2i+1}})\pi(X_b)=\tau(a^{2i+1},b)\pi(X_{a^{2i+1}b}),\\
\pi(E_a)^{2i}=\pi(E_{a^{2i}}),\; \pi(E_{a^{2i}})\pi(E_b)=\pi(E_{a^{2i}b}),\\
\pi(E_a)^{2i+1}=\pi(E_{a^{2i+1}}),\; \pi(E_{a^{2i+1}})\pi(E_b)=\pi(E_{a^{2i+1}b})
\end{gather*}
But we have the following equations by the definition of $\pi$
\begin{gather*}
\pi(x_1)=l(X_1),\; \pi(x_a)=l(X_a),\; \pi(x_b)=l(X_b),\\
\pi(e_1)=l(E_1),\; \pi(e_a)=l(E_a),\; \pi(e_b)=l(E_b).
\end{gather*}
Therefore we have $\pi(f)=l(f)$ for all $f\in H^*$ and we have completed the proof.
\end{proof}

In order to get  another side version of above lemma, we need the following observation.
\begin{lemma}\label{lemm4.7}
Let $R_{\alpha,\beta}$ as above, then
\begin{itemize}
\item[(i)] $l(X_t)=r(X_t)$,\;$t\in T$;
\item[(ii)] $l(E_t)=r(E_{t\triangleleft x})$\;$t\in T$;
\item[(iii)] $l(E_s)=r(E_s)$\;$s\in S$.
\end{itemize}
\end{lemma}
\begin{proof}
We will show (i) at first. By definition, $l(X_t)=\sum_{t'\in T} w^4(t,t')e_{t'}x$ and $r(X_t)=\sum_{t'\in T} w^4(t',t)e_{t'}x$. To show (i), we need to prove that $w^4(t,t')=w^4(t',t)$. Direct computations show that
\begin{align*}
w^4(a^{2i+1},a^{2j+1})&=\lambda_{2i+1,2j+1} S_{j,0}^{i+1} S_{j,1}^i=\lambda_{2i+1,2j+1} [\lambda_{2j+1,1} \alpha^{j+1} \beta^j]^{i+1} S_{j,1}^i\\
&=\lambda_{2i+1,2j+1} [\lambda_{2j+1,1} \alpha^{j+1} \beta^j]^{i+1} [h(a,a^{2j+1}b)\lambda_{2j+1,1} \alpha^{j} \beta^{j+1}]^i\\
&=\lambda_{2i+1,2j+1} \lambda_{2j+1,1}^{2i+1} h(a,a^{2j+1}b)^i \alpha^{(i+1)(j+1)+ij} \beta^{j(i+1)+i(j+1)}\\
&=P_{2i+1}^{-1} \sigma(a^{2j+1})^i \lambda_{2j+1,1}^{2i+1} h(a,a^{2j+1}b)^i \alpha^{(i+1)(j+1)+ij} \beta^{j(i+1)+i(j+1)}\\
&=P_{2i+1}^{-1} \sigma(a^{2j+1})^i P_{2j+1}^{-(2i+1)} \sigma(a)^{j(i+1)} h(a,a^{2j+1}b)^i \alpha^{2ij+i+j+1} \beta^{2ij+i+j}
\end{align*}
and similarly
\begin{align*}
w^4(a^{2j+1},a^{2i+1})&=P_{2j+1}^{-1} \sigma(a^{2i+1})^j P_{2i+1}^{-(2j+1)} \sigma(a)^{i(j+1)} h(a,a^{2i+1}b)^j \alpha^{2ij+i+j+1} \beta^{2ij+i+j}.
\end{align*}
By \eqref{4.b4} in Lemma \ref{lem4.1}, we find that
\begin{align}
\label{wt1} w^4(a^{2i+1},a^{2j+1})=w^4(a^{2j+1},a^{2i+1}).
\end{align}
To show (i), we also need consider other cases. By
\begin{align*}
w^4(a^{2i+1},a^{2j+1}b)&=\lambda_{2i+1,2j+1} S_{j,0}^i S_{j,1}^{i+1}=\lambda_{2i+1,2j+1} S_{j,0}^{i+1} S_{j,1}^{i} \frac{S_{j,1}}{S_{j,0}}\\
&=w^4(a^{2i+1},a^{2j+1}) \frac{S_{j,1}}{S_{j,0}}=\frac{S_{j,1}}{S_{j,0}} w^4(a^{2i+1},a^{2j+1})
\end{align*}
and
\begin{align*}
w^4(a^{2j+1}b,a^{2i+1})&=h(a^{2j+1}b,a^{2i+1})\lambda_{2j+1,2i+1} S_{i,0}^j S_{i,1}^{j+1}\\
&=h(a^{2j+1}b,a^{2i+1})[\lambda_{2j+1,2i+1} S_{i,0}^{j+1} S_{i,1}^{j}] \frac{S_{i,1}}{S_{i,0}}\\
&=h(a^{2j+1}b,a^{2i+1}) w^4(a^{2j+1},a^{2i+1}) \frac{S_{i,1}}{S_{i,0}}\\
&=h(a^{2j+1}b,a^{2i+1}) \frac{S_{i,1}}{S_{i,0}} w^4(a^{2j+1},a^{2i+1})
\end{align*}
and
\begin{align*}
w^4(a^{2i+1},a^{2j+1})=w^4(a^{2j+1},a^{2i+1}),
\end{align*}
 to show that $w^4(a^{2i+1},a^{2j+1}b)=w^4(a^{2j+1}b,a^{2i+1})$ we just need to prove that
\begin{align*}
 \frac{S_{j,1}}{S_{j,0}}=h(a^{2j+1}b,a^{2i+1}) \frac{S_{i,1}}{S_{i,0}}.
\end{align*}
In fact,
\begin{align*}
h(a^{2j+1}b,a^{2i+1}) \frac{S_{i,1}}{S_{i,0}}&=h(a^{2j+1}b,a^{2i+1}) \frac{h(a,a^{2i+1}b) \lambda_{2i+1,1} \alpha^i \beta^{i+1}}{\lambda_{2i+1,1} \alpha^{i+1} \beta^{i}}\\
&=h(a^{2j+1}b,a^{2i+1}) h(a,a^{2i+1}b) \frac{\beta}{\alpha}\\
&=h(a^{2j+1}b,a^{2i+1}) \frac{\tau(b,a)}{\tau(b,a^{2i+1})} \frac{\beta}{\alpha}\;\;(\text{by Lemma}\;\ref{lem4.1})\\
&=\frac{\tau(a^{2j+1}b,a^{2i+1})}{\tau(a^{2i+1}b,a^{2j+1})} \frac{\tau(b,a)}{\tau(b,a^{2i+1})} \frac{\beta}{\alpha}\\
&=\frac{\tau(a^{2j+1}b,a^{2i+1}) \tau(b,a)}{\tau(a^{2i+1}b,a^{2j+1}) (-\tau(a^{2i+1},b))} \frac{\beta}{\alpha}\\
&=\frac{\tau(a^{2j+1}b,a^{2i+1}) \tau(b,a)}{\tau(a^{2i+1},ba^{2j+1}) (-\tau(b,a^{2j+1}))} \frac{\beta}{\alpha}\\
&=\frac{\tau(a^{2j+1}b,a^{2i+1}) \tau(b,a)}{\tau(a^{2i+1},ba^{2j+1}) (-\tau(b,a^{2j+1}))} \frac{\beta}{\alpha}\\
&=\eta(a^{2j+1}b,a^{2i+1}) \frac{\tau(b,a)}{-\tau(b,a^{2j+1})} \frac{\beta}{\alpha}\\
&=- \frac{\tau(b,a)}{-\tau(b,a^{2j+1})} \frac{\beta}{\alpha}=\frac{\tau(b,a)}{\tau(b,a^{2j+1})} \frac{\beta}{\alpha}\\
&=h(a,a^{2j+1}b) \frac{\beta}{\alpha}\;\;(\text{by}\;\eqref{4.b6}),\\
\end{align*}
and
\begin{align*}
\frac{S_{j,1}}{S_{j,0}}&=\frac{h(a,a^{2j+1}b) \lambda_{2j+1,1} \alpha^j \beta^{j+1}}{\lambda_{2j+1,1} \alpha^{j+1} \beta^{j}}\\
&=h(a,a^{2j+1}b) \frac{\beta}{\alpha}.
\end{align*}
Therefore we have
\begin{align}
\label{wt2} w^4(a^{2i+1},a^{2j+1}b)=w^4(a^{2j+1}b,a^{2i+1}).
\end{align}
Since
\begin{align*}
w^4(a^{2i+1}b,a^{2j+1}b)&=h(a^{2i+1}b,a^{2j+1}b) \lambda_{2i+1,2j+1} S_{j,0}^{i+1} S_{j,1}^i\\
&=h(a^{2i+1}b,a^{2j+1}b) w^4(a^{2i+1},a^{2j+1}),
\end{align*}
\begin{align*}
w^4(a^{2j+1}b,a^{2i+1}b)&=h(a^{2j+1}b,a^{2i+1}b) \lambda_{2j+1,2i+1} S_{i,0}^{j+1} S_{i,1}^j\\
&=h(a^{2j+1}b,a^{2i+1}b) w^4(a^{2j+1},a^{2i+1})\\
&=h(a^{2j+1}b,a^{2i+1}b) w^4(a^{2i+1},a^{2j+1})\;\;(\text{by}\;\eqref{wt1})
\end{align*}
and
\begin{align*}
h(a^{2i+1}b,a^{2j+1}b)&=h(a^{2j+1}b,a^{2i+1}b)\;\;(\text{by}\;\eqref{4.b5}),
\end{align*}
we have
\begin{align}
\label{wt3} w^4(a^{2i+1}b,a^{2j+1}b)=w^4(a^{2j+1}b,a^{2i+1}b).
\end{align}
Combining the equations \eqref{wt1},\eqref{wt2},\eqref{wt3} and $T=\{a^{2i+1},a^{2i+1}b\;|\;i\geq 0\}$, we know that $w^4(t,t')=w^4(t',t)$ for all $t,t'\in T$. Thus (i) is proved.

To show (ii). By $l(E_t)=\sum_{s\in S} w^3(t,s)e_s x$ and $r(E_{t\triangleleft x})=\sum_{s\in S} w^2(s,t \triangleleft x)e_s x$,  we need to show that $w^3(t,s)=w^2(s,t\triangleleft x)$ for $s\in S,t\in T.$  By definition,
\begin{align*}
w^3(a^{2i+1},a^{2j})&=\lambda_{2j,2i+1} [S_{i,0}S_{i,1}]^j\\
&=w^2(a^{2j},a^{2i+1}b)
\end{align*}
and
\begin{align*}
w^3(a^{2i+1}b,a^{2j})&=\lambda_{2j,2i+1} [S_{i,0}S_{i,1}]^j\\
&=w^2(a^{2j},a^{2i+1})
\end{align*}
and
\begin{align*}
w^3(a^{2i+1},a^{2j}b)&=-\frac{\tau(b,a) \beta}{\tau(b,a^{2j}) \alpha}\lambda_{2j,2i+1} [S_{i,0}S_{i,1}]^j\\
&=w^2(a^{2j}b,a^{2i+1}b)
\end{align*}
and
\begin{align*}
w^3(a^{2i+1}b,a^{2j}b)&=\frac{\tau(b,a) \beta}{\tau(b,a^{2j}) \alpha}\lambda_{2j,2i+1} [S_{i,0}S_{i,1}]^j\\
&=w^2(a^{2j}b,a^{2i+1}).
\end{align*}
 By $S=\{a^{2i},a^{2i}b\;|\;i\geq 0\}$ and $T=\{a^{2i+1},a^{2i+1}b\;|\;i\geq 0\}$,  $w^3(t,s)=w^2(s,t\triangleleft x)$ and (ii) is proved.

 At last, let us show (iii). Similarly, by $l(E_s)=\sum_{s'\in S} w^1(s,s')e_{s'}$ and $r(E_s)=\sum_{s'\in S} w^1(s',s)e_{s'}$, to show (iii) we need to show that $w^1(s,s')=w^1(s',s)$ for $s,s' \in S.$  Since
\begin{align*}
w^1(a^{2i},a^{2j})&=(\lambda_{2j,1})^{2i}(\alpha\beta)^{2ij}[\sigma(a^{2j})]^i\\
&=[P_{2j}^{-1} \sigma(a)^j]^{2i}(\alpha\beta)^{2ij}[\sigma(a^{2j})]^i\\
&=P_{2j}^{-2i} \sigma(a^{2j})^i \sigma(a)^{2ij}(\alpha\beta)^{2ij},
\end{align*}
\begin{align*}
w^1(a^{2j},a^{2i})&=P_{2i}^{-2j} \sigma(a^{2i})^j \sigma(a)^{2ij}(\alpha\beta)^{2ij}
\end{align*}
and
\begin{align*}
\frac{P_{2j}^{2i}}{P_{2i}^{2j}}=\frac{\sigma(a^{2j})^i}{\sigma(a^{2i})^j},\;\;(\text{by Lemma}\;\ref{lem4.1})
\end{align*}
we have
\begin{align}
\label{ws.1} w^1(a^{2i},a^{2j})=w^1(a^{2j},a^{2i}).
\end{align}
By the definition of $w^1$, we have $w^1(a^{2i}b,a^{2j})=w^1(a^{2i},a^{2j})$ and $w^1(a^{2j},a^{2i}b)=w^1(a^{2j},a^{2i})$. Because $w^1(a^{2i},a^{2j})=w^1(a^{2j},a^{2i})$, we know that
\begin{align}
\label{ws.2} w^1(a^{2i}b,a^{2j})=w^1(a^{2j},a^{2i}b).
\end{align}
Similarly, owing to the definition of $w^1$, we have $w^1(a^{2i}b,a^{2j}b)=-w^1(a^{2i},a^{2j})$ and $w^1(a^{2j}b,a^{2i}b)=-w^1(a^{2j},a^{2i})$. Using $w^1(a^{2i},a^{2j})=w^1(a^{2j},a^{2i})$ again, we can get that
\begin{align}
\label{ws.3} w^1(a^{2i}b,a^{2j}b)=w^1(a^{2j}b,a^{2i}b).
\end{align}
Combining these equations \eqref{ws.1},\eqref{ws.2} and \eqref{ws.3}, we obtain that $l(E_s)=r(E_s)$ and thus (iii) has been proved.
\end{proof}

Based on this observation, we have

\begin{lemma}\label{lemm4.8}
Let $R_{\alpha,\beta}$ as above, then $r(f_1)r(f_2)=r(f_2 f_1)$ for $f_1,f_2\in H^*$ where $H^*$ is the dual of $K(8n,\sigma,\tau)$.
\end{lemma}
\begin{proof}
Since $H^*=\langle X_t,X_s,E_t,E_s\;|\;s\in S,t\in T\rangle$ as linear space, we have to show the following equations:
\begin{gather}
\label{re1} r(X_{t_1})r(X_{t_2})=r(X_{t_2} X_{t_1}),\;r(X_t)r(X_s)=r(X_{s}X_{t}),\\
\label{re2} r(X_s)r(X_t)=r(X_{t}X_{s}),\;r(X_{s_1})r(X_{s_2})=r(X_{s_2}X_{s_1}),\\
\label{re3} r(E_{t_1})r(E_{t_2})=r(E_{t_2} E_{t_1}),\;r(E_t)r(E_s)=r(E_{s}E_{t}),\\
\label{re4} r(E_s)r(E_t)=r(E_{t}E_s),\;r(E_{s_1})r(E_{s_2})=r(E_{s_2} E_{s_2}),\\
\label{re5} r(X_g)r(E_h)=0,\;r(E_h)r(X_g)=0
\end{gather}
where $s,s_1,s_2\in S$ and $t,t_1,t_2\in T$ and $g,h\in G$. To show that $r(X_{t_1})r(X_{t_2})=r(X_{t_2} X_{t_1})$, we need to prove the following equation:
\begin{align}
\label{rw.1} w^4(t',t_1) w^4(t'\triangleleft x,t_2) \sigma(t)=\tau(t_2,t_1) w^3(t',t_1 t_2).
\end{align}
Since Lemma \ref{lemm4.6}, we have $l(X_{t_2}) l(X_{t_1})=\tau(t_2,t_1)l(X_{t_1 t_2})$. Because $l(X_{t_2}) l(X_{t_1})=\tau(t_2,t_1)l(X_{t_1 t_2})$,
\begin{align*}
l(X_{t_2})l(X_{t_1})&=(\sum_{t'\in T} w^4(t_2,t')e_{t'}x) (\sum_{t'\in T} w^4(t_1,t')e_{t'}x)\\
&=\sum_{t'\in T} w^4(t_2,t') w^4(t_1,t' \triangleleft x) e_{t'}x^2\\
&=\sum_{t'\in T} w^4(t_2,t') w^4(t_1,t' \triangleleft x) \sigma(t') e_{t'}\\
&=\sum_{t'\in T} w^4(t',t_2) w^4(t' \triangleleft x,t_1) \sigma(t') e_{t'}\;\;(\text{by (i) of Lemma}\;\ref{lemm4.7})\\
&=\sum_{t'\in T} w^4(t' \triangleleft x,t_1) w^4(t',t_2) \sigma(t') e_{t'},
\end{align*}
and
\begin{align*}
\tau(t_2,t_1)l(X_{t_1 t_2})&=\tau(t_2,t_1) (\sum_{t'\in T} w^2(t_1 t_2,t')e_{t'})\\
&=\tau(t_2,t_1) (\sum_{t'\in T} w^3(t'\triangleleft x,t_1 t_2)e_{t'})\\
&= \sum_{t'\in T}\tau(t_2,t_1) w^3(t'\triangleleft x,t_1 t_2)e_{t'},
\end{align*}
we have $w^4(t' \triangleleft x,t_1) w^4(t',t_2) \sigma(t')=\tau(t_2,t_1) w^3(t'\triangleleft x,t_1 t_2)$ for all $t'\in T$. Since $(t'\triangleleft x)\in T$ if $t'\in T$, we know that
$w^4(t',t_1) w^4(t'\triangleleft x,t_2) \sigma(t'\triangleleft x)=\tau(t_2,t_1) w^3(t',t_1 t_2)$, but $\sigma(t'\triangleleft x)=\sigma(t')$ by definition of $\sigma$, so we have
\begin{align*}
w^4(t',t_1) w^4(t'\triangleleft x,t_2) \sigma(t')=\tau(t_2,t_1) w^3(t',t_1 t_2).
\end{align*}
Since
\begin{align*}
r(X_{t_1})r(X_{t_2})&=(\sum_{t'\in T} w^4(t',t_1)e_{t'}x) (\sum_{t'\in T} w^4(t',t_2)e_{t'}x)\\
&=\sum_{t'\in T} w^4(t',t_1) w^4(t' \triangleleft x,t_2) e_{t'}x^2\\
&=\sum_{t'\in T} w^4(t',t_1) w^4(t' \triangleleft x,t_2) \sigma(t') e_{t'}\\
&=\sum_{t'\in T} \tau(t_2,t_1) w^3(t',t_1 t_2) e_{t'}\;\;(\text{by }\eqref{rw.1})\\
&=\tau(t_2,t_1) \sum_{t'\in T} w^3(t',t_1 t_2) e_{t'}\\
&=\tau(t_2,t_1)r(X_{t_1 t_2})\\
&=r(X_{t_2} X_{t_1}),
\end{align*}
we have $r(X_{t_1})r(X_{t_2})=r(X_{t_2} X_{t_1})$. Then we will show that $r(X_t)r(X_s)=r(X_{s}X_{t})$. Since
\begin{align*}
r(X_t)r(X_s)&=(\sum_{t'\in T} w^4(t',t)e_{t'}x) (\sum_{t'\in T} w^3(t',s)e_{t'})
=\sum_{t'\in T} w^4(t',t) w^3(t'\triangleleft x,s)e_{t'}x\\
&=\sum_{t'\in T} w^4(t,t') w^3(t'\triangleleft x,s)e_{t'}x
=\sum_{t'\in T} w^4(t,t') w^2(s,t)e_{t'}x\\
&=(\sum_{t'\in T} w^2(s,t)e_{t'}) (\sum_{t'\in T} w^4(t,t')e_{t'}x)\\
&=l(X_s)l(X_t)
=\tau(s,t)l(X_{st}) \;\;(\text{by Lemma}\;\ref{lemm4.6})\\
&=\tau(s,t)r(X_{st})\;\;(\text{by Lemma}\;\ref{lemm4.7})\\
&=r(X_{s} X_{t}),
\end{align*}
we have $r(X_t)r(X_s)=r(X_{s}X_{t})$ and therefore the equation \eqref{re1} holds.

Then we will show that $r(X_s)r(X_t)=r(X_{t}X_{s})$ and $r(X_{s_1})r(X_{s_2})=r(X_{s_2}X_{s_1})$.
Since
\begin{align*}
r(X_s)r(X_t)&=(\sum_{t'\in T} w^3(t',s)e_{t'}) (\sum_{t'\in T} w^4(t',t)e_{t'}x)\\
&=\sum_{t'\in T} w^3(t',s) w^4(t',t) e_{t'}x
=\sum_{t'\in T} w^2(s,t'\triangleleft x)w^4(t',t) e_{t'}x\\
&=\sum_{t'\in T} w^2(s,t'\triangleleft x)w^4(t,t') e_{t'}x
=\sum_{t'\in T} w^4(t,t')w^2(s,t'\triangleleft x) e_{t'}x\\
&=(\sum_{t'\in T} w^4(t,t') e_{t'}x) (\sum_{t'\in T}w^2(s,t') e_{t'})
=l(X_t)l(X_s)\\
&=\tau(t,s)l(X_{st}) \;\;(\text{by Lemma}\;\ref{lemm4.6})\\
&=\tau(t,s)r(X_{st})\;\;(\text{by Lemma}\;\ref{lemm4.7})\\
&=r(X_{t} X_{s}),
\end{align*}
we have $r(X_s)r(X_t)=r(X_{t}X_{s})$. Since $l(X_{s_2}) l(X_{s_1})=\tau(s_2,s_1)l(X_{s_1 s_2})$,
\begin{align*}
l(X_{s_2}) l(X_{s_1})&=(\sum_{t'\in T} w^2(s_2,t')e_{t'}) (\sum_{t'\in T} w^2(s_1,t')e_{t'})\\
&=\sum_{t'\in T} w^2(s_2,t') w^2(s_1,t') e_{t'},
\end{align*}
and
\begin{align*}
\tau(s_2,s_1)l(X_{s_1 s_2})&=\tau(s_2,s_1)(\sum_{t'\in T} w^2(s_1 s_2,t')e_{t'})\\
&=\sum_{t'\in T} \tau(s_2,s_1) w^2(s_1 s_2,t')e_{t'},
\end{align*}
we have
\begin{align}
\label{x1} w^2(s_2,t') w^2(s_1,t')=\tau(s_2,s_1) w^2(s_1 s_2,t').
\end{align}
Since
\begin{align*}
r(X_{s_1}) r(X_{s_2})&=(\sum_{t'\in T} w^3(t',s_1)e_{t'}) (\sum_{t'\in T} w^3(t',s_2)e_{t'})=\sum_{t'\in T} w^3(t',s_1) w^3(t',s_2) e_{t'}\\
&=\sum_{t'\in T} w^2(s_1,t'\triangleleft x) w^2(s_2,t'\triangleleft x) e_{t'} \;\;(\text{by (ii) of Lemma}\;\ref{lemm4.7})\\
&=\sum_{t'\in T} \tau(s_2,s_1) w^2(s_1 s_2,t'\triangleleft x) e_{t'}\;\;(\text{by }\;\eqref{x1})\\
&=\sum_{t'\in T} \tau(s_2,s_1) w^3(t',s_1 s_2,) e_{t'} \;\;(\text{by (ii) of Lemma}\;\ref{lemm4.7})\\
&=\tau(s_2,s_1)r(X_{s_1 s_2})=r(X_{s_2} X_{s_1}),
\end{align*}
we have $r(X_{s_1}) r(X_{s_2})=r(X_{s_2} X_{s_1})$ and therefore the equation \eqref{re2} has been proved.

Then we will show that $r(E_{t_1})r(E_{t_2})=r(E_{t_2} E_{t_1})$ and $r(E_t)r(E_s)=r(E_{s}E_{t})$. Since
\begin{align*}
r(E_{t_1})r(E_{t_2})&=l(E_{t_1 \triangleleft x}) r(E_{t_2})=l(E_{t_1 \triangleleft x}) l(E_{t_2 \triangleleft x}) \;\;(\text{by Lemma}\;\ref{lemm4.7})\\
&=l(E_{t_1 \triangleleft x} E_{t_2 \triangleleft x}) \;\;(\text{by Lemma}\;\ref{lemm4.6})\\
&=l(E_{(t_1t_2)\triangleleft x})=l(E_{t_1t_2})\\
&=r(E_{t_1t_2})\;\;(\text{by (iii) of Lemma}\;\ref{lemm4.7})
\end{align*}
and
\begin{align*}
r(E_{t})r(E_{s})&=l(E_{t \triangleleft x}) r(E_s)=l(E_{t \triangleleft x}) l(E_s) \;\;(\text{by Lemma}\;\ref{lemm4.7})\\
&=l(E_{t \triangleleft x} E_s) \;\;(\text{by Lemma}\;\ref{lemm4.6})\\
&=l(E_{(ts)\triangleleft x})=r(E_{ts}) \;\;(\text{by Lemma}\;\ref{lemm4.7})\\
&=r(E_{s} E_{t}),
\end{align*}
we have $r(E_{t_1})r(E_{t_2})=r(E_{t_2} E_{t_1})$ and $r(E_t)r(E_s)=r(E_{s}E_{t})$. Therefore the equation \eqref{re3} holds.

Then we will show that $r(E_s)r(E_t)=r(E_{t}E_s)$ and $r(E_{s_1})r(E_{s_2})=r(E_{s_2} E_{s_2})$. Since
\begin{align*}
r(E_{s})r(E_{t})&=r(E_s) l(E_{t \triangleleft x})\\
&=l(E_s)l(E_{t \triangleleft x}) \;\;(\text{by Lemma}\;\ref{lemm4.7})\\
&=l(E_s E_{t \triangleleft x} ) \;\;(\text{by Lemma}\;\ref{lemm4.6})\\
&=l(E_{(st)\triangleleft x})=r(E_{st}) \;\;(\text{by Lemma}\;\ref{lemm4.7})\\
&=r(E_{t} E_{s}),
\end{align*}
and
\begin{align*}
r(E_{s_1}) r(E_{s_2})&=l(E_{s_1}) l(E_{s_2})\;\;(\text{by Lemma}\;\ref{lemm4.7})\\
&=l(E_{s_1} E_{s_2}) \;\;(\text{by Lemma}\;\ref{lemm4.6})\\
&=l(E_{s_1 s_2})=r(E_{s_1 s_2})=r(E_{s_2} E_{s_1}),
\end{align*}
we have $r(E_s)r(E_t)=r(E_{t}E_s)$ and $r(E_{s_1})r(E_{s_2})=r(E_{s_2} E_{s_2})$ and hence the equation \eqref{re4} holds.

Then we will show that $r(X_g)r(E_h)=0$ and $r(E_h)r(X_g)=0$ for $g,h\in G$. Since
\begin{gather*}
r(X_t)=\sum_{t'\in T} w^4(t',t)e_{t'}x,\;\;\;\;
r(X_s)=\sum_{t'\in T} w^3(t',s)e_{t'},\\
r(E_t)=\sum_{s' \in S} w^2(s',t)e_{s'}x,\;\;\;\;
r(E_s)=\sum_{s'\in S} w^1(s',s)e_{s'},
\end{gather*}
we know that $r(X_g)r(E_h)=0$ and $r(E_h)r(X_g)=0$ for $g,h\in G$ and therefore the equation \eqref{re5} holds.
\end{proof}
With the above preparation, we can prove that
\begin{proposition} \label{pro4.1} The element
$R_{\alpha,\beta}$ is a universal $\mathcal{R}$-matrix of $K(8n,\sigma,\tau)$.
\end{proposition}
\begin{proof}
By Lemmas \ref{lemm4.2} and \ref{lemm4.3}, we have
\begin{align*}
&\tau(s_1,s_2)=\tau(s_2,s_1),\; s_1,s_2\in S,\\
&w^2(s,t\triangleleft x)=w^2(s,t)\eta(s,t),\; s \in S,t\in T,\\
&w^3(t\triangleleft x,s)=w^3(t,s)\eta(t,s),\;s \in S,t\in T,\\
&\tau(t_2,t_1)w^4(t_1\triangleleft x,t_2\triangleleft x)=\tau(t_1\triangleleft x,t_2\triangleleft x)w^4(t_1,t_2),\;t_1,t_2 \in T.
\end{align*}
By Lemmas \ref{lemm4.6} and  \ref{lemm4.8}, we have
\begin{gather*}
l(f_1)l(f_2)=l(f_1f_2),\;r(f_1)r(f_2)=r(f_2f_1),\;f_1,f_2\in H^*.
\end{gather*}
 Due to Lemma \ref{lem3.1},  $R_{\alpha,\beta}$ is a universal $\mathcal{R}$-matrix of $K(8n,\sigma,\tau)$.
\end{proof}

Now we turn to the proof of Proposition \ref{addpro}. Let $R$ be a universal $\mathcal{R}$-matrix of $\Bbbk^G\#_{\sigma,\tau}\Bbbk \mathbb{Z}_{2}$, and let $H_l=\{l(f)\;|\;f\in H^*\}$ and $H_r=\{r(f)\;|\;f\in H^*\}$ where $H^*$ is the dual of $\Bbbk^G\#_{\sigma,\tau}\Bbbk \mathbb{Z}_{2}$. Note that both $H_l$ and $H_r$ are subalgebras of $H^\ast.$
\begin{lemma}\label{lemm3.2} We have
$H_l=\langle l(X_t),l(E_t)\;|\;t\in T\rangle $ and $H_r=\langle r(X_t),r(E_t)\;|\;t\in T\rangle $ as algebras.
\end{lemma}

\begin{proof}
By Lemma \ref{lem2.2.2} and $S=TT$, we know that $H^*=\langle X_t,E_t\;|\;t\in T\rangle $ as an algebra. Define two maps
\begin{align*}&\pi:H^*\rightarrow H_l,\;\;\;\; f\mapsto l(f)\\
 &\pi':H^*\rightarrow H_r,\;\;\;\; f\mapsto r(f).\end{align*}
 It is not hard to see that both of them are surjective algebra maps.
\end{proof}

Now let $R_w$ and $R_v$ be two non-trivial universal $\mathcal{R}$-matrices of $\Bbbk^G\#_{\sigma,\tau}\Bbbk \mathbb{Z}_{2}$ which are denoted by the following way:
\begin{itemize}
              \item[] $R_w=\sum\limits_{s_1,s_2 \in S}w^1(s_1,s_2)e_{s_1} \otimes e_{s_2}+ \sum\limits_{s \in S, t \in T}w^2(s,t)e_{s}x \otimes e_{t}+
        \sum\limits_{t \in T,s \in S}w^3(t,s)e_{t} \otimes e_{s}x+
         \sum\limits_{t_1,t_2 \in T}w^4(t_1,t_2)e_{t_1}x \otimes e_{t_2}x$
\end{itemize}
and
\begin{itemize}
              \item[] $R_v=\sum\limits_{s_1,s_2 \in S}v^1(s_1,s_2)e_{s_1} \otimes e_{s_2}+ \sum\limits_{s \in S, t \in T}v^2(s,t)e_{s}x \otimes e_{t}+
        \sum\limits_{t \in T,s \in S}v^3(t,s)e_{t} \otimes e_{s}x+
         \sum\limits_{t_1,t_2 \in T}v^4(t_1,t_2)e_{t_1}x \otimes e_{t_2}x$.
\end{itemize}
We find that
\begin{lemma}\label{lem3.2}
 Given $R_w,\;R_v$ as above, then $R_w=R_v$ if and only if $w^4(t_1,t_2)=v^4(t_1,t_2)$ for $t_1,\;t_2\in T$.
\end{lemma}

\begin{proof}
Let $l_w:H^*\rightarrow H$ defined by $l_w(f)=(f \otimes \id)(R_w)$ and let $l_v:H^*\rightarrow H$ defined by $l_v(f)=(f \otimes \id)(R_v)$, then $R_w=R_v$ if and only if $l_w=l_v$. By Lemma \ref{lem2.2.2} and $S=TT$, we get that $H^*=\langle X_t,E_t\;|\;t\in T\rangle $ as an algebra. This implies that $l_w=l_v$ if and only if $l_w(X_t)=l_v(X_t)$ and $l_w(E_t)=l_v(E_t)$ for $t\in T$. Since
\begin{gather*}
l_w(X_t)=\sum_{t\in T}w^4(t,t')e_{t'}x,\; l_v(X_t)=\sum_{t\in T}v^4(t,t')e_{t'}x,\;t\in T,\\
l_w(E_t)=\sum_{t\in T}w^3(t,s)e_{s}x,\; l_v(E_t)=\sum_{t\in T}v^3(t,s)e_{s}x,\;t\in T,
\end{gather*}
we know that $l_w=l_v$ if and only if $w^4(t,t')=v^4(t,t')$ and $w^3(t,s)=v^3(t,s)$ for $s\in S,t,t'\in T$. To complete the proof, we will show that if $w^4(t,t')=v^4(t,t')$ for all $t,t'\in T$ then $w^3(t,s)=v^3(t,s)$ for all $s\in S,t\in T$. Let $r_w:(H^*)^{op}\rightarrow H$ defined by $r_w(f)=( \id \otimes f )(R_w)$ and let $r_v:(H^*)^{op}\rightarrow H$ defined by $r_v(f)=(\id \otimes f)(R_v)$, then $r_w$ and $r_v$ are algebra maps by Lemma \ref{lem2.2.1}, and therefore we have
\begin{gather*}
r_w(X_s)r_w(X_t)=\tau(t,s)r_w(X_{st}),\; r_v(X_s)r_v(X_t)=\tau(t,s)r_v(X_{st}),\;s\in S,\; t\in T.
\end{gather*}
Since
\begin{align*}
r_w(X_s)r_w(X_t)&=(\sum\limits_{t' \in T}w^3(t',s)e_{t'})(\sum\limits_{t' \in T}w^4(t',t)e_{t'}x)\\
                &=\sum\limits_{t' \in T}w^3(t',s)w^4(t',t)e_{t'}x , \\
r_v(X_s)r_v(X_t)&=(\sum\limits_{t' \in T}v^3(t',s)e_{t'})(\sum\limits_{t' \in T}v^4(t',t)e_{t'}x)\\
                &=\sum\limits_{t' \in T}v^3(t',s)v^4(t',t)e_{t'}x  \\
\end{align*}
and
\begin{align*}
\tau(t,s)r_w(X_{st})&=\tau(t,s)(\sum\limits_{t' \in T}w^4(t',ts)e_{t'}x)\\
                &=\sum\limits_{t' \in T}\tau(t,s)w^4(t',st)e_{t'}x, \\
\tau(t,s)r_v(X_{st})&=\tau(t,s)(\sum\limits_{t' \in T}v^4(t',ts)e_{t'}x)\\
                &=\sum\limits_{t' \in T}\tau(t,s)v^4(t',st)e_{t'}x,
\end{align*}
we have
\begin{gather*}
w^3(t',s)=\tau(t,s)\frac{w^4(t',st)}{w^4(t',t)},\; s\in S,\; t,t'\in T,\\
v^3(t',s)=\tau(t,s)\frac{v^4(t',st)}{v^4(t',t)},\; s\in S,\; t,t'\in T.
\end{gather*}
As a result we know that if $w^4(t,t')=v^4(t,t')$ for all $t,t'\in T$ then $w^3(t,s)=w^3(t,s)$ for all $s\in S,t\in T$.
\end{proof}

The following Lemma \ref{all.1} and Lemma \ref{all.2} are used to compute the fourth matrix $w^4$ of $R$.
\begin{lemma}\label{all.1}
Assume that $R$ is a universal $\mathcal{R}$-matrix of $K(8n,\sigma,\tau)$. If we let $w^4(a,a)=\alpha$ and $w^4(a,ab)=\beta$, then $(\alpha\beta)^n \lambda_{2n,1}=1$ and $\frac{\beta^2}{\alpha^2}=\frac{\tau(b,b)}{\tau(b,a)^2}$.
\end{lemma}
\begin{proof}
Since $l(X_a)^{2n}=P_{2n}l(X_1)$ and
\begin{align*}
l(X_a)^{2n}&=[\sum_{t\in T} w^4(a,t)e_t x]^{2n}=[(\sum_{t\in T} w^4(a,t)e_t x)^2]^n\\
&=[\sum_{t\in T} w^4(a,t) w^4(a,t\triangleleft x)e_t x^2]^n=[\sum_{t\in T} w^4(a,t) w^4(a,t\triangleleft x) \sigma(t) e_t]^n\\
&=\sum_{t\in T} w^4(a,t)^n w^4(a,t\triangleleft x)^n \sigma(t)^n e_t,
\end{align*}
\begin{align*}
P_{2n}l(X_1)&=P_{2n} \sum_{t\in T} w^2(1,t)e_t=P_{2n} \sum_{t\in T} e_t\\
&=\sum_{t\in T}P_{2n}e_t,
\end{align*}
 we have $$w^4(a,t)^n w^4(a,t\triangleleft x)^n \sigma(t)^n=P_{2n}.$$ Take $t=a$, we get that $w^4(a,a)^n w^4(a,ab)^n \sigma(a)^n=P_{2n}$. Since $w^4(a,a)=\alpha,$ $w^4(a,ab)=\beta$ and $\lambda_{2n,1}=P_{2n}^{-1} \sigma(a)^n$, we have $(\alpha\beta)^n \lambda_{2n,1}=1$.

 Since $l(X_b)l(X_a)=\tau(b,a) l(X_{ab})$,
\begin{align*}
l(X_b)l(X_a)&=(\sum_{t\in T} w^2(b,t)e_t)(\sum_{t\in T} w^4(a,t)e_t x)\\
&=\sum_{t\in T}w^2(b,t) w^4(a,t)e_t x
\end{align*}
and
\begin{align*}
\tau(b,a) l(X_{ab})&=\tau(b,a) \sum_{t\in T} w^4(ab,t)e_t x\\
&=\sum_{t\in T} \tau(b,a) w^4(ab,t)e_t x,
\end{align*}
 we must have $w^2(b,t) w^4(a,t)=\tau(b,a) w^4(ab,t)$ for $t\in T$. Through letting $t=a$, then the equality $w^2(b,t) w^4(a,t)=\tau(b,a) w^4(ab,t)$ becomes
\begin{align}
\label{temp1} w^2(b,a)w^4(a,a)=\tau(b,a) w^4(ab,a).
\end{align}
We claim $w^4(ab,a)=\beta$. By the equation \eqref{e3.14}, we have $w^4(t_1,t_2)=h(t_1,t_2) w^4(t_1 \triangleleft x, t_2 \triangleleft x)$. Through letting $t_1=a,t_2=ab$, the equality $w^4(t_1,t_2)=h(t_1,t_2) w^4(t_1 \triangleleft x, t_2 \triangleleft x)$ becomes $w^4(ab,a)=h(ab,a)w^4(a,ab)$. By $h(ab,a)=\frac{\tau(ab,a)}{\tau(ab,a)}=1$, we have $w^4(ab,a)=w^4(a,ab)=\beta$. Due to the equation $w^2(b,a)w^4(a,a)=\tau(b,a) w^4(ab,a)$ and $w^4(ab,a)=\beta$, we know that
\begin{align}
\label{temp2} w^2(b,a)=\tau(b,a) \frac{\beta}{\alpha}.
\end{align}
Since $l(X_b)^2=\tau(b,b)l(X_1)$ and
\begin{align*}
&l(X_b)^2=(\sum_{t\in T} w^2(b,t)e_t)^2=\sum_{t\in T}w^2(b,t)^2 e_t,\\
&\tau(b,b)l(X_1)=\tau(b,b) \sum_{t\in T} w^2(1,t)e_t=\sum_{t\in T}\tau(b,b) e_t,
\end{align*}
 we have $w^2(b,t)^2=\tau(b,b)$. Taking $t=a$, $w^2(b,t)^2=\tau(b,b)$ becomes
$w^2(b,a)^2=\tau(b,b)$. Since  $w^2(b,a)^2=\tau(b,b)$ and $w^2(b,a)=\tau(b,a) \frac{\beta}{\alpha}$, we have $\frac{\beta^2}{\alpha^2}=\frac{\tau(b,b)}{\tau(b,a)^2}$.
\end{proof}

\begin{lemma}\label{all.2}
Assume that $R$ is a universal $\mathcal{R}$-matrix of $K(8n,\sigma,\tau)$. If we let $w^4(a,a)=\alpha$ and $w^4(a,ab)=\beta$, then $w^4$ of $R$ must have the following expression: \\[1.5mm]
$ \left\{
\begin{array}{l}
w^4(a^{2i+1},a^{2j+1})=\lambda_{2i+1,2j+1}S_{j,0}^{i+1}S_{j,1}^{i}\\
w^4(a^{2i+1},a^{2j+1}b)=\lambda_{2i+1,2j+1}S_{j,0}^{i}S_{j,1}^{i+1}\\
w^4(a^{2i+1}b,a^{2j+1})=h(a^{2i+1}b,a^{2j+1})\lambda_{2i+1,2j+1}S_{j,0}^{i}S_{j,1}^{i+1}\\
w^4(a^{2i+1}b,a^{2j+1}b)=h(a^{2i+1}b,a^{2j+1}b)\lambda_{2i+1,2j+1}S_{j,0}^{i+1}S_{j,1}^{i}
\end{array} \right.$
for $0\leq i,j \leq (n-1)$.
\end{lemma}
\begin{proof}
Since $r(X_a)^{2i+1}=P_{2i+1}r(X_{a^{2i+1}})$ and
\begin{align*}
r(X_a)^{2i+1}&=[\sum_{t\in T} w^4(t,a)e_t x]^{2i+1}\\
&=[\sum_{t\in T} w^4(t,a)e_t x]^{2i}[\sum_{t\in T} w^4(t,a)e_t x]\\
&=[(\sum_{t\in T} w^4(t,a)e_t x)^2]^{i} [\sum_{t\in T} w^4(t,a)e_t x]\\
&=[\sum_{t\in T} w^4(t,a) w^4(t\triangleleft x,a)e_t x^2]^{i} [\sum_{t\in T} w^4(t,a)e_t x]\\
&=[\sum_{t\in T} w^4(t,a) w^4(t\triangleleft x,a) \sigma(t)e_t]^{i} [\sum_{t\in T} w^4(t,a)e_t x]\\
&=[\sum_{t\in T} w^4(t,a)^i w^4(t\triangleleft x,a)^i \sigma(t)^i e_t] [\sum_{t\in T} w^4(t,a)e_t x]\\
&=\sum_{t\in T} w^4(t,a)^{i+1} w^4(t\triangleleft x,a)^i \sigma(t)^i e_t x,
\end{align*}
\begin{align*}
P_{2i+1}r(X_{a^{2i+1}})&=P_{2i+1} \sum_{t\in T} w^4(t,a^{2i+1})e_t x\\
&=\sum_{t\in T}P_{2i+1} w^4(t,a^{2i+1})e_t x,
\end{align*}
we have $w^4(t,a)^{i+1} w^4(t\triangleleft x,a)^i \sigma(t)^i=P_{2i+1} w^4(t,a^{2i+1})$. Taking $t=a$ in this equation, we get that $\alpha^{i+1} \beta^i \sigma(a)^i=P_{2i+1} w^4(a,a^{2i+1})$. Since by definition $\lambda_{2i+1,1}=P_{2i+1}^{-1} \sigma(a)^i$, we know
\begin{align}
\label{temp3} w^4(a,a^{2i+1})=S_{i,0}.
\end{align}
Similarly, if we let $t=ab$, then $w^4(t,a)^{i+1} w^4(t\triangleleft x,a)^i \sigma(t)^i=P_{2i+1} w^4(t,a^{2i+1})$ becomes $\beta^{i+1} \alpha^i \sigma(a)^i=P_{2i+1} w^4(ab,a^{2i+1})$ and so we have $w^4(ab,a^{2i+1})=\lambda_{2i+1,1} \alpha^i \beta^{i+1}$. Since $w^4(t_1,t_2)=h(t_1,t_2)w^4(t_1\triangleleft x,t_2\triangleleft x)$ by Lemma \ref{lem3.1},  let $t_1=a$ and $t_2=a^{2i+1}b$ we have $w^4(a,a^{2i+1}b)=h(a,a^{2i+1}b)w^4(ab,a^{2i+1})$. Because $w^4(ab,a^{2i+1})=\lambda_{2i+1,1} \alpha^i \beta^{i+1}$, we find that
\begin{align}
\label{temp4} w^4(a,a^{2i+1}b)=S_{i,1}.
\end{align}
Since $l(X_a)^{2i+1}=P_{2i+1}l(X_{a^{2i+1}})$,
\begin{align*}
l(X_a)^{2i+1}&=[\sum_{t\in T} w^4(a,t)e_t x]^{2i+1}\\
&=[\sum_{t\in T} w^4(a,t)e_t x]^{2i}[\sum_{t\in T} w^4(a,t)e_t x]\\
&=[(\sum_{t\in T} w^4(a,t)e_t x)^2]^{i} [\sum_{t\in T} w^4(a,t)e_t x]\\
&=[\sum_{t\in T} w^4(a,t) w^4(a,t\triangleleft x)e_t x^2]^{i} [\sum_{t\in T} w^4(a,t)e_t x]\\
&=[\sum_{t\in T} w^4(a,t) w^4(a,t\triangleleft x) \sigma(t)e_t]^{i} [\sum_{t\in T} w^4(a,t)e_t x]\\
&=[\sum_{t\in T} w^4(a,t)^i w^4(a,t\triangleleft x)^i \sigma(t)^i e_t] [\sum_{t\in T} w^4(a,t)e_t x]\\
&=\sum_{t\in T} w^4(a,t)^{i+1} w^4(a,t\triangleleft x)^i \sigma(t)^i e_t x
\end{align*}
and
\begin{align*}
P_{2i+1}l(X_{a^{2i+1}})&=P_{2i+1} \sum_{t\in T} w^4(a^{2i+1},t)e_t x\\
&=\sum_{t\in T}P_{2i+1} w^4(a^{2i+1},t)e_t x,
\end{align*}
we have $w^4(a,t)^{i+1} w^4(a,t\triangleleft x)^i \sigma(t)^i=P_{2i+1} w^4(a^{2i+1},t)$. Now let $t=a^{2j+1}$, then $w^4(a,t)^{i+1} w^4(a,t\triangleleft x)^i \sigma(t)^i=P_{2i+1} w^4(a^{2i+1},t)$ becomes
\begin{align*}
w^4(a,a^{2j+1})^{i+1} w^4(a,a^{2j+1}b)^i \sigma(a^{2j+1})^i=P_{2i+1} w^4(a^{2i+1},a^{2j+1}).
\end{align*}
Since $w^4(a,a^{2j+1})=S_{j,0}$ by \eqref{temp3}, $w^4(a,a^{2j+1}b)=S_{j,1}$ by \eqref{temp4} and $\lambda_{2i+1,2j+1}=P_{2i+1}^{-1} \sigma(a^{2j+1})^i$, we know that
\begin{align}
\label{temp5} w^4(a^{2i+1},a^{2j+1})=\lambda_{2i+1,2j+1} S_{j,0}^{i+1} S_{j,1}^i.
\end{align}
Similarly, if we let $t=a^{2j+1}b$, then $w^4(a,t)^{i+1} w^4(a,t\triangleleft x)^i \sigma(t)^i=P_{2i+1} w^4(a^{2i+1},t)$ becomes
\begin{align*}
w^4(a,a^{2j+1}b)^{i+1} w^4(a,a^{2j+1})^i \sigma(a^{2j+1}b)^i=P_{2i+1} w^4(a^{2i+1},a^{2j+1}b).
\end{align*}
Since $w^4(a,a^{2j+1})=S_{j,0}$ by equation \eqref{temp3}, $w^4(a,a^{2j+1}b)=S_{j,1}$ by equation \eqref{temp4}, $\lambda_{2i+1,2j+1}=P_{2i+1}^{-1} \sigma(a^{2j+1})^i$ and $\sigma(a^{2j+1}b)=\sigma(a^{2j+1})$, we know that
\begin{align}
\label{temp6} w^4(a^{2i+1},a^{2j+1}b)=\lambda_{2i+1,2j+1} S_{j,0}^{i} S_{j,1}^{i+1}.
\end{align}
 Due to $w^4(t_1,t_2)=h(t_1,t_2)w^4(t_1\triangleleft x,t_2\triangleleft x)$ by \eqref{e3.14}, this implies that if we let $t_1=a^{2i+1}b$ and $t_2=a^{2j+1}$ then $w^4(a^{2i+1}b,a^{2j+1})=h(a^{2i+1}b,a^{2j+1})w^4(a^{2i+1},a^{2j+1}b)$. Therefore the equation \eqref{temp6} implies that
\begin{align}
\label{temp7} w^4(a^{2i+1}b,a^{2j+1})=h(a^{2i+1}b,a^{2j+1}) \lambda_{2i+1,2j+1} S_{j,0}^{i} S_{j,1}^{i+1}.
\end{align}
 Moreover, if we let $t_1=a^{2i+1}b$ and $t_2=a^{2j+1}b$, then $w^4(t_1,t_2)=h(t_1,t_2)w^4(t_1\triangleleft x, t_2\triangleleft x)$ gives
\begin{align*}
w^4(a^{2i+1}b,a^{2j+1}b)=h(a^{2i+1}b,a^{2j+1}b)w^4(a^{2i+1},a^{2j+1}).
\end{align*}
So the equation \eqref{temp5} tells us that
\begin{align}
\label{temp8} w^4(a^{2i+1}b,a^{2j+1}b)=h(a^{2i+1}b,a^{2j+1}b) \lambda_{2i+1,2j+1} S_{j,0}^{i+1} S_{j,1}^{i}.
\end{align}
 Putting the equations \eqref{temp5},\eqref{temp6},\eqref{temp7},\eqref{temp8} together, we get what we want.
\end{proof}
Using Lemma \ref{lem3.2}-Lemma \ref{all.2}, we can prove that
\begin{proposition} \label{addpro}
If $R$ is a universal $\mathcal{R}$-matrix of $K(8n,\sigma,\tau)$, then $R=R_{\alpha,\beta}$ for some $\alpha,\beta \in \Bbbk$ such that $(\alpha\beta)^n\lambda_{2n,1}=1$ and $\frac{\beta^2}{\alpha^2}=\frac{\tau(b,b)}{\tau(b,a)^2}$.
\end{proposition}
\begin{proof}
Assume that $R=R_{w}$ is a non-trivial universal $\mathcal{R}$-matrix of $K(8n,\sigma,\tau)$.  If we take $w^4(a,a)=\alpha$ and $w^4(a,ab)=\beta$, then $(\alpha\beta)^n \lambda_{2n,1}=1$ and $\frac{\beta^2}{\alpha^2}=\frac{\tau(b,b)}{\tau(b,a)^2}$ by Lemma \ref{all.1}. Therefore we can define a $R_{\alpha,\beta}$ as before. By Lemma \ref{all.2}, we know that $w^4$ of $R_{w}$ is the
same with the fourth matrix of $R_{\alpha,\beta}$.  Owing to Lemma \ref{lem3.2}, we get that $R_{w}=R_{\alpha,\beta}$.

\end{proof}

%
\subsection{The case $\eta(a,b)=1$}\label{sec4.4}
We assume that $\eta(a,b)=1$ in this subsection, and we will give all universal $\mathcal{R}$-matrices of $K(8n,\sigma,\tau)$ in this subsection. Due to the proof is almost the same as the case $\eta(a,b)=-1$, we only give the results here without proofs.

\begin{proposition}
Assume that $R$ is a trivial quasitriangular structure on $K(8n,\sigma,\tau)$, then $$R=\sum_{1\leq i,k \leq n,\;1\leq j,l \leq 2}\alpha^{ik}\beta^{il-jk}e_{a^ib^j} \otimes e_{a^kb^l}$$ for some $\alpha, \beta\in \Bbbk$ satisfying $\alpha^{2n}=\beta^2=1.$
\end{proposition}

To give all non-trivial quasitriangular structures on $K(8n,\sigma,\tau)$, we just need to change the $R_{\alpha,\beta}$ of Proposition \ref{pro4.1} a little. That is to say if we take $\alpha,\beta \in \Bbbk$ such that $(\alpha\beta)^n\lambda_{2n,1}=1$ and $\frac{\beta^2}{\alpha^2}=\frac{\tau(b,b)}{\tau(b,a)^2}$ and let $S_{j,0}=\lambda_{2j+1,1}\alpha^{j+1}\beta^j$, $S_{j,1}=h(a,a^{2j+1}b)\lambda_{2j+1,1}\alpha^j\beta^{j+1}$ for $j\in \mathbb{N}$, then we can construct    $$R'_{\alpha,\beta}$$
 in the form of \eqref{r} through letting
\begin{itemize}
\item[(i)] $w^1$ be given by \\\\
$ \left\{
\begin{array}{l}
w^1(a^{2i},a^{2j})=w^1(a^{2i}b,a^{2j})=(\lambda_{2j,1})^{2i}(\alpha\beta)^{2ij}[\sigma(a^{2j})]^i\\
w^1(a^{2i},a^{2j}b)=w^1(a^{2i}b,a^{2j}b)=(\lambda_{2j,1})^{2i}(\alpha\beta)^{2ij}[\sigma(a^{2j})]^i
\end{array} \right.$,\\

\item[(ii)] $w^2$ be given by  \\\\
$ \left\{
\begin{array}{l}
w^2(a^{2i},a^{2j+1})=w^2(a^{2i},a^{2j+1}b)=\lambda_{2i,2j+1}[S_{j,0}S_{j,1}]^{i}\\
w^2(a^{2i}b,a^{2j+1})=w^2(a^{2i}b,a^{2j+1}b)=\frac{\tau(b,a)\beta}{\tau(b,a^{2i})\alpha}
\lambda_{2i,2j+1}[S_{j,0}S_{j,1}]^{i}
\end{array} \right.$,\\

\item[(iii)] $w^3$ be given by \\\\
$ \left\{
\begin{array}{l}
w^3(a^{2i+1},a^{2j})=w^3(a^{2i+1}b,a^{2j})=\lambda_{2j,2i+1}[S_{i,0}S_{i,1}]^{j}\\
w^3(a^{2i+1},a^{2j}b)=w^3(a^{2i+1}b,a^{2j}b)=\frac{\tau(b,a)\beta}{\tau(b,a^{2j})\alpha}
\lambda_{2j,2i+1}[S_{i,0}S_{i,1}]^{j}
\end{array} \right.$,\\

\item[(iv)] $w^4$ be given by \\\\
$ \left\{
\begin{array}{l}
w^4(a^{2i+1},a^{2j+1})=\lambda_{2i+1,2j+1}S_{j,0}^{i+1}S_{j,1}^{i}\\
w^4(a^{2i+1},a^{2j+1}b)=\lambda_{2i+1,2j+1}S_{j,0}^{i}S_{j,1}^{i+1}\\
w^4(a^{2i+1}b,a^{2j+1})=h(a^{2i+1}b,a^{2j+1})\lambda_{2i+1,2j+1}S_{j,0}^{i}S_{j,1}^{i+1}\\
w^4(a^{2i+1}b,a^{2j+1}b)=h(a^{2i+1}b,a^{2j+1}b)\lambda_{2i+1,2j+1}S_{j,0}^{i+1}S_{j,1}^{i}
\end{array} \right.$,\\
 \end{itemize}
for $0\leq i,j \leq (n-1)$.
\begin{proposition}\label{propo4.4}  The set of elements $\{R'_{\alpha,\beta}\;|\;\alpha,\beta \in \Bbbk ,(\alpha\beta)^n\lambda_{2n,1}=1\; \text{and}\; \frac{\beta^2}{\alpha^2}=\frac{\tau(b,b)}{\tau(b,a)^2}\}$ gives all non-trivial quasitriangular structures on $K(8n,\sigma,\tau)$.
\end{proposition}

\subsection{A class of minimal quasitriangular Hopf algebras}
In this subsection, we firstly identify all minimal quasitriangular Hopf algebras among $K(8n,\sigma,\tau)$ for the case $\eta(a,b)=-1$ and we get a very simple criterion for $K(8n,\sigma,\tau)$ to be minimal quasitriangular Hopf algebras. Secondly we study minimal quasitriangular structures on $K(8n,\sigma,\tau)$ for the case $\eta(a,b)=1$ and we found that $K(8n,\sigma,\tau)$ has no minimal quasitriangular structure on it. Furthermore, if $K(8n,\sigma,\tau)$ is minimal for the case $\eta(a,b)=-1$, then we will write down all its minimal quasitriangular structures on it. And using these results we can prove that $K(8n,\zeta)$ are minimal if $n\geq 4$ is even. Based on these results, we get a class of minimal quasitriangular Hopf algebras $K(8n,\zeta)$ ($n\geq 4$ is even). As a direct consequence, we can easily get that $K_8$ is minimal and give all universal $\mathcal{R}$-matrices on it.

Since we also feel interested in minimal quasitriangular Hopf algebras, we give a sufficient condition for $\Bbbk^G\#_{\sigma,\tau}\Bbbk \mathbb{Z}_{2}$ to be a minimal quasitriangular Hopf algebra.
\begin{lemma}\label{lem3.3}
Let $\Bbbk^G\#_{\sigma,\tau}\Bbbk \mathbb{Z}_{2}$ as above and let $R$ as above \ref{r}, and if $R$ is a universal $\mathcal{R}$-matrix on $\Bbbk^G\#_{\sigma,\tau}\Bbbk \mathbb{Z}_{2}$ satisfying
\begin{itemize}
  \item[(i)] the linear subspace spanned by $\{l(X_t)\;|\;t \in T\}$ equals to the linear subspace spanned by $\{r(X_t)\;|\;t \in T\}$,
  \item[(ii)] the linear subspace spanned by $\{l(E_t)\;|\;t \in T\}$ equals to the linear subspace spanned by $\{r(E_t)\;|\;t \in T\}$,
\end{itemize}
 then $(\Bbbk^G\#_{\sigma,\tau}\Bbbk \mathbb{Z}_{2},R)$ is minimal if and only if $w^i(1\leq i\leq 4)$ are four non-degenerated matrices.
\end{lemma}

\begin{proof}
Firstly, we will show that $H_l=H_r$ and $H_lH_r=H_l$. By Lemma \ref{lemm3.2},  $H_l=\langle l(X_t),l(E_t)\;|\;t\in T\rangle $  and $H_r=\langle r(X_t),r(E_t)\;|\;t\in T\rangle $. By conditions (i) and (ii), we have $H_l=H_r$ and therefore $H_lH_r=H_l$.

Secondly, we will show that $H_l=H$ if and only if $w^i(1\leq i\leq 4)$ are non-degenerated matrices. If $H_l=H$, then $l:H^*\rightarrow H$ is bijective where $l(f)=(f \otimes \id)(R)$. Since $\{X_s,E_s,X_t,E_t\;|s\in S,t\in T\}$ is a base of $H^*$ by definition, we get that $\{l(X_s),l(E_s),l(X_t),l(E_t)\;|s\in S,t\in T\}$ is a base of $H$. But by definition,
\begin{align*}
l(E_s)&=\sum_{s'\in T}w^1(s,s')e_{s'},\;\;\;\;
l(X_s)=\sum_{t\in T}w^2(s,t)e_{t},\\
l(E_t)&=\sum_{t\in T}w^3(t,s)e_{s}x,\;\;\;\;
l(X_t)=\sum_{t\in T}w^4(t,t')e_{t'}x.
\end{align*}
Therefore $w^i(1\leq i\leq 4)$ are non-degenerated matrices. Now if $w^i(1\leq i\leq 4)$ are non-degenerated matrices, then $\{l(X_s),l(E_s),l(X_t),l(E_t)\;|s\in S,t\in T\}$ is a base of $H$ and thus $l:H^*\rightarrow H_l$ is bijective. As a result we know that $\text{dim}(H_l)=\text{dim}(H)$ which implies that $H_l=H$.
\end{proof}

Using Lemma \ref{lem3.3} and Lemma \ref{lemm4.7}, we can get that
\begin{corollary} \label{coro4.1}
Let $R_{\alpha,\beta}$ as above in Proposition \ref{pro4.1}, then $R_{\alpha,\beta}$ is a minimal quasitriangular structure on $K(8n,\sigma,\tau)$ if and only if its $w^i(1\leq i\leq4)$ are non-degenerated matrices.
\end{corollary}

\begin{proof}
 By Lemma \ref{lemm4.7}, we have $l(X_t)=r(X_t)$ and $l(E_t)=r(E_{t\triangleleft x})$ for $t\in T$. Therefore $R_{\alpha,\beta}$ satisfies that the condition of Lemma \ref{lem3.3}. Now using Lemma \ref{lem3.3} again and we get the result.
\end{proof}
We are not very satisfied with this conclusion due to the absence of an easy criteria for the non-degeneracies of these matrices $w^{i}\;(1\leq i\leq 4)$. Although at the first glance these matrices seem quite complicated, we still have:

\begin{lemma}\label{min}
The matrices $w^i(1\leq i\leq4)$ of $R_{\alpha,\beta}$ are non-degenerated if and only if $(\sigma(a) \alpha \beta)^2 \sigma(a^2) \tau(a,a)^{-2}$ is a primitive $n$th root of 1.
\end{lemma}
\begin{proof}
To calculate the matrices $w^{i}\;(1\leq i\leq4)$, we number the elements in sets $S,T$ in the following way
\begin{gather*}
s_j:=a^{2j-2},\;s_{n+j}:=a^{2j-2}b,\;t_j:=a^{2j-1},\;t_{n+j}:=a^{2j-1}b,
\end{gather*}
where $1 \leq i,j \leq n$. For convenience, let us denote four matrices $A^{i}(1\leq i\leq4)$ as follows
\begin{gather*}
A^1:=(w^1(s_i,s_j))_{1\leq i,j \leq n} \quad A^2:=(w^2(s_i,t_j))_{1\leq i,j \leq n}\\ A^3:=(w^3(t_i,s_j))_{1\leq i,j \leq n}\quad A^4:=(w^4(t_i,t_j))_{1\leq i,j \leq n}.
\end{gather*}

Firstly, we determine when $(w^1(s_i,s_j))_{1\leq i,j \leq 2n}$ is non-degenerate. For $1\leq i,j \leq n$, we find that
\begin{align*}
w^1(s_i,s_{n+j})&=w^1(a^{2i-2},a^{2j-2}b)=(\lambda_{2j-2,1})^{2i-2} (\alpha \beta)^{2(i-1)(j-1)}[\sigma(a^{2j-2})]^{i-1}\\
&=w^1(a^{2i-2},a^{2j-2})=w^1(s_i,s_j),\\
w^1(s_{n+i},s_{j})&=w^1(a^{2i-2}b,a^{2j-2})=(\lambda_{2j-2,1})^{2i-2} (\alpha \beta)^{2(i-1)(j-1)}[\sigma(a^{2j-2})]^{i-1}\\
&=w^1(a^{2i-2},a^{2j-2})=w^1(s_i,s_j)
\end{align*}
and
\begin{align*}
w^1(s_{n+i},s_{n+j})&=w^1(a^{2i-2}b,a^{2j-2}b)=-(\lambda_{2j-2,1})^{2i-2} (\alpha \beta)^{2(i-1)(j-1)}[\sigma(a^{2j-2})]^{i-1}\\
&=-w^1(a^{2i-2},a^{2j-2})=w^1(s_i,s_j).
\end{align*}
Therefore we have
\begin{gather*}
(w^1(s_i,s_j))_{1\leq i,j \leq 2n} = \begin{pmatrix} A^1 & A^1 \\ A^1 & -A^1 \end{pmatrix} \sim \begin{pmatrix} A^1 & 0 \\ 0 & A^1 \end{pmatrix}.
\end{gather*} Here ``$\sim$" means that two matrices can be gotten each other through elementary operations. Thus $(w^1(s_i,s_j))_{1\leq i,j \leq 2n}$ is non-degenerate if and only if $A^1$ is non-degenerate. Since $l(E_{s_1}) l(E_{s_2})=l(E_{s_1s_2})$ and $r(E_{s_1}) r(E_{s_2})=r(E_{s_1s_2})$, we get that $w^1$ is bicharacter on $S$. Let $w^1(a^2,a^2)=\gamma$, then we have
\begin{align*}
\gamma&=(\lambda_{2,1})^2 (\alpha \beta)^2 \sigma(a^2)=(P_{2}^{-1} \sigma(a))^2 (\alpha \beta)^2 \sigma(a^2)\\
&=(\tau(a,a)^{-1} \sigma(a))^2 (\alpha \beta)^2 \sigma(a^2)=(\sigma(a) \alpha \beta)^2 \sigma(a^2) \tau(a,a)^{-2}
\end{align*}
and
\begin{align*}
w^1(s_i,s_j)&=w^1(a^{2i-2},a^{2j-2})\\
&=w^1(a^2,a^2)^{(i-1)(j-1)}\\
&=\gamma^{(i-1)(j-1)}.
\end{align*}
Therefore we have $A^1=(\gamma^{(i-1)(j-1)})_{1\leq i,j \leq n}$. So $A^1$ is non-degenerate if and only if $\gamma^m\neq 1$ for $1\leq m \leq (n-1)$. Since $w^1$ is bicharacter on $S$, we have $w^1(a^2,a^2)^n=w^1(a^{2n},a^2)=1$ and thus $\gamma^n=1$. So we get that  $(w^1(s_i,s_j))_{1\leq i,j \leq 2n}$ is non-degenerated if and only if the following condition holds
\begin{gather}
\label{non-deg-w1}(\sigma(a) \alpha \beta)^2 \sigma(a^2) \tau(a,a)^{-2} \;\text{is a primitive nth root of 1.}
\end{gather}
Secondly, we determine when $(w^2(s_i,t_j))_{1\leq i,j \leq 2n}$ is non-degenerate. Assume that $1\leq i,j \leq n$ and let $\delta:=\tau(b,a)\frac{\beta}{\alpha}$, then we have
\begin{align*}
w^2(s_{i},t_{n+j})&=w^2(a^{2i-2},a^{2j-1}b)\\
&=w^2(a^{2i-2},a^{2j-1}) \\
&=w^2(s_{i},t_{j})
\end{align*}
and
\begin{align*}
w^2(s_{i+n},t_{j})&=w^2(a^{2i-2}b,a^{2j-1})\\
&=\frac{\tau(b,a) \beta}{ \tau(b,a^{2i-2}) \alpha} \lambda_{2i-2,2j-1} [S_{j-1,0} S_{j-1,1}]^{i-1}\\
&=\tau(b,a^{2i-2})^{-1} \frac{\tau(b,a) \beta}{\alpha} \lambda_{2i-2,2j-1} [S_{j-1,0} S_{j-1,1}]^{i-1}\\
&=\tau(b,a^{2i-2})^{-1} \delta \lambda_{2i-2,2j-1} [S_{j-1,0} S_{j-1,1}]^{i-1}\\
&=\tau(b,a^{2i-2})^{-1} \delta w^2(a^{2i-2},a^{2j-1})\\
&=\tau(b,a^{2i-2})^{-1} \delta w^2(s_i,t_j).
\end{align*}
Therefore  $w^2(s_{i},t_{n+j})=w^2(s_{i},t_{j})$ and $w^2(s_{i+n},t_{j})=\tau(b,a^{2i-2})^{-1} \delta w^2(s_i,t_j)$. Since
\begin{align*}
w^2(s_{i+n},t_{j+n})&=w^2(a^{2i-2}b,a^{2j-1}b)\\
&=-w^2(a^{2i-2}b,a^{2j-1})\\
&=-w^2(s_{i+n},t_{j})\\
&=-\tau(b,a^{2i-2})^{-1} \delta w^2(s_i,t_j),
\end{align*}
we have $w^2(s_{i+n},t_{j+n})=-\tau(b,a^{2i-2})^{-1} \delta w^2(s_i,t_j)$. Let $B=(b_{ij})_{1\leq i,j \leq n}$ be a $n\times n$ matrix defined by $b_{ij}=\tau(b,a^{2i-2})^{-1} \delta w^2(s_i,t_j)$, then we have
\begin{gather*}
(w^2(s_i,t_j))_{1\leq i,j \leq 2n} = \begin{pmatrix} A^2 & A^2 \\ B & -B \end{pmatrix} \sim \begin{pmatrix} A^2 & A^2 \\ A^2 & -A^2 \end{pmatrix} \sim \begin{pmatrix} A^2 & 0 \\ 0 & A^2 \end{pmatrix}.
\end{gather*}
Therefore $(w^2(s_i,t_j))_{1\leq i,j \leq 2n}$ is non-degenerate if and only if $A^2$ is non-degenerate. Let $\alpha_j:=w^2(a^2,a^{2j-1})$, then we have $\alpha_j=P_{2}^{-1}\sigma(a^{2j-1}) S_{j-1,0} S_{j-1,1}$ by definition. By
\begin{align*}
w^2(s_{i},t_{j})&=w^2(a^{2i-2},a^{2j-1})\\
&=\lambda_{2i-2,2j-1}[S_{j-1,0} S_{j-1,1}]^{i-1}\\
&=P_{2i-2}^{-1} \sigma(a^{2j-1})^{i-1} [S_{j-1,0} S_{j-1,1}]^{i-1}\\
&=P_{2i-2}^{-1} P_{2}^{i-1} P_{2}^{-(i-1)}\sigma(a^{2j-1})^{i-1} [S_{j-1,0} S_{j-1,1}]^{i-1}\\
&=P_{2i-2}^{-1} P_{2}^{i-1} [P_{2}^{-1}\sigma(a^{2j-1}) S_{j-1,0} S_{j-1,1}]^{i-1}\\
&=P_{2i-2}^{-1} P_{2}^{i-1} \alpha_j^{(i-1)},
\end{align*}
we have
\begin{gather*}
A^2= (P_{2i-2}^{-1} P_{2}^{i-1} \alpha_j^{(i-1)})_{1 \leq i,j \leq n} \sim (\alpha_j^{(i-1)})_{1 \leq i,j \leq n}.
\end{gather*}
Therefore $A^2$ is non-degenerated if and only if $\frac{\alpha_j}{\alpha_i}\neq 1$ for $1\leq i < j\leq n$. Direct computations show that
\begin{align*}
\frac{\alpha_j}{\alpha_i}&=\frac{P_{2}^{-1}\sigma(a^{2j-1}) S_{j-1,0} S_{j-1,1}}{P_{2}^{-1}\sigma(a^{2i-1}) S_{i-1,0} S_{i-1,1}}\\
&=\frac{\sigma(a^{2j-1}) S_{j-1,0} S_{j-1,1}}{\sigma(a^{2i-1}) S_{i-1,0} S_{i-1,1}}
\end{align*}
and
\begin{align*}
\frac{S_{j-1,0} S_{j-1,1}}{S_{i-1,0} S_{i-1,1}}&=\frac{(\lambda_{2j-1,1} \alpha^j \beta^{j-1}) S_{j-1,1}}{S_{i-1,0} S_{i-1,1}}\\
&=\frac{(\lambda_{2j-1,1} \alpha^j \beta^{j-1}) (h(a,a^{2j-1}b) \lambda_{2j-1,1} \alpha^{j-1} \beta^{j}) }{S_{i-1,0} S_{i-1,1}}\\
&=\frac{h(a,a^{2j-1}b) (\lambda_{2j-1,1})^2 (\alpha \beta)^{2j-1}}{h(a,a^{2i-1}b) (\lambda_{2i-1,1})^2 (\alpha \beta)^{2i-1}}\\
&=\frac{(\lambda_{2j-1,1})^2}{(\lambda_{2i-1,1})^2} \frac{h(a,a^{2j-1}b)}{h(a,a^{2i-1}b)} (\alpha \beta)^{2(j-i)}\\
&=\frac{P_{2j-1}^{-2} \sigma(a)^{2j}}{P_{2i-1}^{-2} \sigma(a)^{2i}} \frac{h(a,a^{2j-1}b)}{h(a,a^{2i-1}b)} (\alpha \beta)^{2(j-i)}\\
&=\frac{P_{2j-1}^{-2}}{P_{2i-1}^{-2}} \frac{h(a,a^{2j-1}b)}{h(a,a^{2i-1}b)} [\sigma(a) \alpha \beta]^{2(j-i)}.
\end{align*}
This implies that
\begin{align}
\label{change1.1} \frac{\alpha_j}{\alpha_i}&=\frac{\sigma(a^{2j-1})} {\sigma(a^{2i-1})} \frac{P_{2j-1}^{-2}}{P_{2i-1}^{-2}} \frac{h(a,a^{2j-1}b)}{h(a,a^{2i-1}b)} [\sigma(a) \alpha \beta]^{2(j-i)}.
\end{align}
By Lemma \ref{lem4.1},
\begin{align}
\label{change1} P_{2k+1}^{2l}\sigma(a)^l\sigma(a^{2l})^k&=P_{2l}^{2k}\sigma(a^{2k+1})^l
[\frac{\tau(b,a)}{\tau(b,a^{2k+1})}]^l,\;k,l\geq 0,\\
h(a,a^{2k+1}b)&=\frac{\tau(b,a)}{\tau(b,a^{2k+1})}\notag
\end{align}
and thus
\begin{align}
\label{change2} P_{2k+1}^{2l}\sigma(a)^l\sigma(a^{2l})^k=P_{2l}^{2k}\sigma(a^{2k+1})^l
[h(a,a^{2k+1}b)]^l,\;k,l\geq 0.
\end{align}
Taking $k=(j-1)$ and $l=1$, then the equation \eqref{change2} becomes $P_{2j-1}^{2}\sigma(a)\sigma(a^{2})^{j-1}=P_{2}^{2j-2}\sigma(a^{2j-1})
[h(a,a^{2k+1}b)]$ and thus
\begin{align}
\label{change3} \sigma(a^{2j-1}) P_{2j-1}^{-2} h(a,a^{2j-1}b)=P_{2}^{-2j-2} \sigma(a) \sigma(a^2)^{j-1}.
\end{align}
Due to the equations \eqref{change1.1} and \eqref{change3}, we have
\begin{align*}
 \frac{\alpha_j}{\alpha_i}&=\frac{P_{2}^{-2j-2} \sigma(a) \sigma(a^2)^{j-1}} {P_{2}^{-2i-2} \sigma(a) \sigma(a^2)^{i-1}} [\sigma(a) \alpha \beta]^{2(j-i)}\\
&=[(\sigma(a)\alpha\beta)^2 \sigma(a^2) \tau(a,a)^{-2}]^{j-i}.
\end{align*}
Since we have showed that $(w^2(s_i,t_j))_{1\leq i,j \leq 2n}$ is non-degenerate if and only if $A^2$ is non-degenerate and $A^2$ is non-degenerate if and only if $\frac{\alpha_j}{\alpha_i}\neq 1$ for $1\leq i < j\leq n$, we know that $(w^2(s_i,t_j))_{1\leq i,j \leq 2n}$ is non-degenerate if and only if the following condition holds
\begin{gather}
\label{non-deg-w2}[(\sigma(a) \alpha \beta)^2 \sigma(a^2) \tau(a,a)^{-2}]^m \neq 1\;\text{for all} \;1 \leq m \leq (n-1).
\end{gather}

 Now we turn to the consideration of the non-degeneracy of $(w^3(t_i,s_j))_{1\leq i,j \leq 2n}$. Thanks to Lemma \ref{lemm4.7}, we have $l(E_t)=r(E_{t\triangleleft x})$ for all $t\in T$. Since $l(E_t)=\sum_{s\in S}w^3(t,s)e_s x$ and $r(E_{t\triangleleft x})=\sum_{s\in S}w^2(s,t\triangleleft x)e_s x$, we have $w^3(t,s)=w^2(s,t\triangleleft x)$. From this observation,  we have
\begin{gather*}
w^3(t_i,s_j)=w^2(s_j,t_{i+n}),\;w^3(t_{i+n},s_j)=w^2(s_j,t_i),\\
w^3(t_{i},s_{j+n})=w^2(s_{j+n},t_{i+n}),\;w^3(t_{i+n},s_{j+n})=w^2(s_{j+n},t_i),
\end{gather*}
for $1\leq i,j \leq n$. Therefore, $(w^3(t_i,s_j))_{1\leq i,j \leq 2n}$ is non-degenerate if and only if $(w^2(s_i,t_j))_{1\leq i,j \leq 2n}$ is non-degenerate and thus $(w^3(t_i,s_j))_{1\leq i,j \leq 2n}$ is non-degenerate if and only if the condition \eqref{non-deg-w2} holds.

Lastly we will determine when $(w^4(t_i,t_j))_{1\leq i,j \leq 2n}$ is non-degenerate. By Lemma \ref{lemm4.6}, we have $l(X_b)l(X_{a^{2i-1}})=\tau(b,a^{2i-1}) l(X_{a^{2i-1}b})$. Because
\begin{align*}
l(X_b)l(X_{a^{2i-1}})&=(\sum_{t\in T} w^2(b,t)e_t) (\sum_{t\in T} w^4(a^{2i-1},t)e_t x)\\
&=\sum_{t\in T} w^2(b,t)w^4(a^{2i-1},t)e_t x
\end{align*}
and
\begin{align*}
\tau(b,a^{2i-1}) l(X_{a^{2i-1}b})&=\tau(b,a^{2i-1})\sum_{t\in T} w^4(a^{2i-1}b,t)e_t x\\
&=\sum_{t\in T}\tau(b,a^{2i-1}) w^4(a^{2i-1}b,t)e_t x,
\end{align*}
we have $w^2(b,t)w^4(a^{2i-1},t)=\tau(b,a^{2i-1}) w^4(a^{2i-1}b,t)$. Through taking $t=a^{2j-1}$, we get that $w^2(b,a^{2j-1})w^4(a^{2i-1},a^{2j-1})=\tau(b,a^{2i-1}) w^4(a^{2i-1}b,a^{2j-1})$ and therefore
\begin{align}
\label{w4non-deg1} w^4(t_{i+n},t_j)=\tau(b,a^{2i-1})^{-1} \delta w^4(t_i,t_j),\;1 \leq i,j \leq n.
\end{align}
And if we let $t=a^{2j-1}b$, then $w^2(b,t)w^4(a^{2i-1},t)=\tau(b,a^{2i-1}) w^4(a^{2i-1}b,t)$ becomes $w^2(b,a^{2j-1}b)w^4(a^{2i-1},a^{2j-1}b)=\tau(b,a^{2i-1}) w^4(a^{2i-1}b,a^{2j-1}b)$, and hence we have
\begin{align}
\label{w4non-deg2} w^4(t_{i+n},t_{j+n})=-\tau(b,a^{2i-1})^{-1} \delta w^4(t_i,t_{j+n})\;1 \leq i,j \leq n.
\end{align}
Similarly, owing to Lemma \ref{lemm4.6}, we have $r(X_b)r(X_{a^{2i-1}})=\tau(a^{2i-1},b) r(X_{a^{2i-1}b})$. By
\begin{align*}
r(X_b)r(X_{a^{2i-1}})&=(\sum_{t\in T} w^3(t,b)e_t) (\sum_{t\in T} w^4(t,a^{2i-1})e_t x)\\
&=\sum_{t\in T} w^3(t,b)w^4(t,a^{2i-1})e_t x,\\
\tau(a^{2i-1},b) r(X_{a^{2i-1}b})&=\tau(a^{2i-1},b)\sum_{t\in T} w^4(t,a^{2i-1}b)e_t x\\
&=\sum_{t\in T}\tau(a^{2i-1},b) w^4(t,a^{2i-1}b)e_t x,
\end{align*}
we have $w^3(t,b)w^4(t,a^{2i-1})=\tau(a^{2i-1},b) w^4(t,a^{2i-1}b)$. Through taking $t=a^{2j-1}$, we get that $w^3(a^{2j-1},b)w^4(a^{2j-1},a^{2i-1})=\tau(a^{2i-1},b) w^4(a^{2j-1},a^{2i-1}b)$. Note that we already have $\tau(a^{2j-1},b)=-\tau(b,a^{2j-1})$ and $w^3(t_i,b)=-\delta$, so
\begin{align}
\label{w4non-deg3} w^4(t_{i},t_{j+n})=\tau(b,a^{2j-1})^{-1} \delta w^4(t_i,t_j),\;1 \leq i,j \leq n.
\end{align}
From the equations \eqref{w4non-deg1}, \eqref{w4non-deg2} and \eqref{w4non-deg3}, we know that
\begin{gather*}
(w^4(t_i,t_j))_{1\leq i,j \leq 2n} = \begin{pmatrix} A^4 & C \\ D & E \end{pmatrix},
\end{gather*}
where $C,D,E$ are $n\times n$ matrices defined by
\begin{align*}
C:&=(\tau(b,a^{2j-1})^{-1} \delta w^4(t_i,t_j))_{1\leq i,j \leq n}, \\
D:&=(\tau(b,a^{2i-1})^{-1} \delta w^4(t_i,t_{j}))_{1\leq i,j \leq n},\\
E:&=(-\tau(b,a^{2i-1})^{-1} \delta w^4(t_i,t_{j+n}))_{1\leq i,j \leq n}.
\end{align*}
Therefore we have
\begin{gather*}
(w^4(t_i,t_j))_{1\leq i,j \leq 2n} = \begin{pmatrix} A^4 & C \\ D & E \end{pmatrix} \sim \begin{pmatrix} A^4 & C \\ A^4 & -C \end{pmatrix} \sim \begin{pmatrix} A^4 & 0 \\ 0 & C \end{pmatrix} \sim \begin{pmatrix} A^4 & 0 \\ 0 & A^4  \end{pmatrix}.
\end{gather*}
As a result we get that $(w^4(t_i,t_j))_{1\leq i,j \leq 2n}$ is non-degenerate
if and only if $A^4$ is non-degenerate. Since
\begin{align*}
w^4(t_i,t_j)&=w^4(a^{2i-1},a^{2j-1})\\
&=\lambda_{2i-1,2j-1} S_{j-1,0}^i S_{j-1,1}^{i-1}\\
&=P_{2i-1}^{-1} \sigma(a^{2j-1})^{i-1} S_{j-1,0}^i S_{j-1,1}^{i-1},
\end{align*}
we have
\begin{gather*}
A^4=(P_{2i-1}^{-1} \sigma(a^{2j-1})^{i-1} S_{j-1,0}^i S_{j-1,1}^{i-1})_{1 \leq i,j \leq n} \sim (\sigma(a^{2j-1})^{i-1} S_{j-1,0}^{i-1} S_{j-1,1}^{i-1})_{1 \leq i,j \leq n}.
\end{gather*}
Therefore $A^4$ is non-degenerate if and only if $\frac{\beta_j}{\beta_i}\neq 1$ for all $1\leq i<j \leq n$ where $\beta_j=\sigma(a^{2j-1}) S_{j-1,0} S_{j-1,1}$. Recall that $\alpha_j=P_{2}^{-1}\sigma(a^{2j-1}) S_{j-1,0} S_{j-1,1}$. Thus $\frac{\beta_j}{\beta_i}=\frac{\alpha_j}{\alpha_i}$ and we get that $A^4$ is non-degenerate if and only if $\frac{\alpha_j}{\alpha_i} \neq 1$ for all $1\leq i<j \leq n$. But we have proved $\frac{\alpha_j}{\alpha_i}\neq 1$ for all $1\leq i<j \leq n$ if and only if the condition \eqref{non-deg-w2} holds (see the proof for the non-degeneracy of $w^2$). Since conditions \ref{non-deg-w1}, \ref{non-deg-w2}, we get what we want.
\end{proof}
Using the Lemma \ref{min}, we can give a very simple criterion for $K(8n,\sigma,\tau)$ to be a minimal quasitriangular Hopf algebra for the case $\eta(a,b)=-1$.
\begin{theorem}\label{thm4.1}
If $K(8n,\sigma,\tau)$ such that $\eta(a,b)=-1$, then it is minimal if and only if there is a $\omega \in \Bbbk$ such that $\omega^n=P_{2n}$ and $\omega^2 \sigma(a^2)\tau(a,a)^{-2}$ is a primitive $n$th root of 1. Moreover, if $K(8n,\sigma,\tau)$ is minimal, then all minimal quasitriangular structures on it can be given by
 $\{R_{\alpha,\beta}\;|\; \alpha^4=\frac{-\omega^2\tau(a,a)^2}{\sigma(a^2)}, \beta=\frac{\omega}{\sigma(a)\alpha},\;\omega \in \Bbbk \; \text{such that}\; \omega^n=P_{2n} \;\text{and} \; \omega^2 \sigma(a^2) \tau(a,a)^{-2} \text{ is a primitive} \;n\text{th root of}\;1\}.$
\end{theorem}

\begin{proof}
Firstly, we will show that if there is a $\omega \in \Bbbk$ such that
\begin{align*}
\omega^n=P_{2n}\; \text{and}\; \omega^2 \sigma(a^2)\tau(a,a)^{-2} \;\text{is a primitive} \;n\text{th root of 1},
\end{align*}
then $K(8n,\sigma,\tau)$ is minimal. Since $\Bbbk$ is algebraically closed, we can find a $\alpha\in \Bbbk$ such that $\alpha^4=\frac{w^2 \tau(b,a)^2}{\sigma(a)^2 \tau(b,b)}$. Define $\beta:=\frac{w}{\sigma(a)\alpha}$. Since
\begin{align*}
(\alpha\beta)^n \lambda_{2n,1}&=(\alpha\beta)^n (P_{2n}^{-1} \sigma(a)^n)=(\alpha \beta  \sigma(a))^n P_{2n}^{-1}\\
&=\omega^n P_{2n}^{-1}=1
\end{align*}
and
\begin{align*}
\frac{\beta^2}{\alpha^2}&=\frac{\omega^2}{\sigma(a)^2 \alpha^4} =\frac{\omega^2 \sigma(a)^2 \tau(b,b)}{\sigma(a)^2 \omega^2 \tau(b,a)^2}\\
&=\frac{\tau(b,b)}{\tau(b,a)^2},
\end{align*}
 we can define $R_{\alpha,\beta}$ by using this $\alpha$ and $\beta$. By Proposition \ref{pro4.1}, $R_{\alpha,\beta}$ is a universal $\mathcal{R}$-matrix of $K(8n,\sigma,\tau)$. By $(\sigma(a) \alpha \beta)^2\sigma(a^2) \tau(a,a)^{-2}=\omega^2 \sigma(a^2)\tau(a,a)^{-2}$ and $\omega^2 \sigma(a^2)\tau(a,a)^{-2}$ is a primitive $n$th root of 1,  $R_{\alpha,\beta}$ is a minimal quasitriangular structure on $K(8n,\sigma,\tau)$ by above Lemma \ref{min}.

Conversely, assume that $K(8n,\sigma,\tau)$ is minimal. Then we need to find a $\omega \in \Bbbk$ such that $\omega^n=P_{2n}$ and $\omega^2 \sigma(a^2)\tau(a,a)^{-2}$ is a primitive $n$th root of 1. Now assume that $R_{\alpha,\beta}$ is a minimal quasitriangular structure on $K(8n,\sigma,\tau)$. Define $\omega:=\alpha \beta \sigma(a)$. We will show that this $\omega$ satisfies our requirements.  Since $(\alpha \beta)^n \lambda_{2n,1}=1$ and $\lambda_{2n,1}=P_{2n}^{-1} \sigma(a)^n$, we have $(\alpha \beta \sigma(a))^n=P_{2n}$ and therefore $\omega^n=P_{2n}$. Because $R_{\alpha,\beta}$ is minimal quasitriangular structure, we know $(\sigma(a) \alpha \beta)^2 \sigma(a^2) \tau(a,a)^{-2}$ is a primitive $n$th root of 1 which implies that $\omega^2 \sigma(a^2) \tau(a,a)^{-2}$ is primitive $n$th root of 1. Hence we have showed $\omega$ is a needed one.

 At last, we need to show that if $K(8n,\sigma,\tau)$ is minimal, then all minimal quasitriangular structures on it can be given by
 $\{R_{\alpha,\beta}\;|\; \alpha^4=\frac{-\omega^2\tau(a,a)^2}{\sigma(a^2)}, \beta=\frac{\omega}{\sigma(a)\alpha},\;\omega \in \Bbbk \; \text{such that}\; \omega^n=P_{2n} \;\text{and} \; \omega^2 \sigma(a^2) \tau(a,a)^{-2} \text{ is a primitive} \;n\text{th root of}\;1\}$. For convenience, we denote the set $\{R_{\alpha,\beta}\;|\; \alpha^4=\frac{-\omega^2\tau(a,a)^2}{\sigma(a^2)}, \beta=\frac{\omega}{\sigma(a)\alpha},\;\omega \in \Bbbk \; \text{such that}\; \omega^n=P_{2n} \;\text{and} \; \omega^2 \sigma(a^2) \tau(a,a)^{-2} \text{ is a primitive} \;n\text{th root of}\;1\}$ by $Q$. If $R_{\alpha,\beta}$ is a minimal quasitriangular structure on it, then we can define
 $\omega:=\alpha\beta\sigma(a)$. It is not hard to see that  $\alpha^4=\frac{-\omega^2\tau(a,a)^2}{\sigma(a^2)}$ and $\beta=\frac{\omega}{\sigma(a)\alpha}$. Therefore $R_{\alpha,\beta} \in Q$. If $R_{\alpha,\beta} \in Q$, then we have $(\alpha\beta)^n \lambda_{2n,1}=1$ and $\frac{\beta^2}{\alpha^2}=\frac{\tau(b,b)}{\tau(b,a)^2}$ due to the calculations above. Therefore, $R_{\alpha,\beta}$ is a universal $\mathcal{R}$-matrix of $K(8n,\sigma,\tau)$ by Proposition \ref{pro4.1}. In addition, Lemma \ref{min} implies that $R_{\alpha,\beta}$ is a minimal quasitriangular structure.

\end{proof}
To use Theorem \ref{thm4.1} more conveniently, we give the following corollary.

\begin{corollary} \label{coro4.2}
Let $K(8n,\sigma,\tau)$ as before in Theorem \ref{thm4.1}. If $\tau(a,a^i)=1$ for $i\in \mathbb{N}$, then $K(8n,\sigma,\tau)$ is minimal if and only if there is a $\omega \in \Bbbk$ such that $\omega^n=1$ and $\omega^2 \sigma(a^2)$ is a primitive $n$th root of 1. Moreover, if $K(8n,\sigma,\tau)$ is minimal, then all minimal quasitriangular structures on it can be given by
 $\{R_{\alpha,\beta}\;|\;\alpha^4=\frac{-\omega^2}{\sigma(a^2)}, \beta=\frac{\omega}{\sigma(a)\alpha} ,\; \omega \in \Bbbk \; \text{such that }\omega^n=1 \; \text{and } \; \omega^2 \sigma(a^2)\; \text{is a primitive} \;n\text{th root of}\;1\}$.
\end{corollary}

Using Corollary \ref{coro4.2}, we can give a class of minimal quasitriangular Hopf algebras as follows

\begin{corollary} \label{coro4.3}
Let $K(8n,\zeta)$ be the Hopf algebras given in Example \ref{ex2.1.1}, then we have the following conclusions:
\begin{itemize}
  \item[(i)] if $n$ is even and $n\geq 4$, then $K(8n,\zeta)$ is minimal and all minimal quasitriangular structures on it can be given by
 $\{R_{\alpha,\beta}\;|\;\alpha^4=\frac{\omega^2}{\zeta^2}, \beta=\frac{\omega}{\alpha \zeta} ,\; \omega \in \Bbbk \; \text{such that}\; \omega^n=1 \;\text{and} \; -(\omega \zeta)^2\; \text{ is primitive} \;n\text{th root of}\;1\}$.
 \item[(ii)] if $n$ is odd or $n=2$, then $K(8n,\zeta)$ is not minimal.
\end{itemize}
\end{corollary}

\begin{proof}
Firstly, we show (i). By the definition of $K(8n,\zeta)$,  $\sigma(a^2)=-\zeta^2$. If $n$ is even and bigger than $4$, we can find a $\omega\in \Bbbk$ such that $\omega^n=1$ and $\omega^2=-1$. Then we have $\omega^2\sigma(a^2)=\zeta^2$ and thus $\omega^2\sigma(a^2)$ is a primitive $n$th root of 1. By Corollary \ref{coro4.2}, we know that $K(8n,\zeta)$ is minimal and all minimal quasitriangular structures on it can be given by
 $\{R_{\alpha,\beta}\;|\;\alpha^4=\frac{\omega^2}{\zeta^2}, \beta=\frac{\omega}{\alpha \zeta}, \; \omega \in \Bbbk \; \text{such that}\; \omega^n=1 \;\text{and} \; -(\omega \zeta)^2\; \text{ is primitive} \;n\text{th root of}\;1\}$.

Secondly, we show (ii). If $n$ is odd, then we have $(\omega^2\sigma(a^2))^n=[-(\omega\zeta)^2]^n=-1$ for arbitrary $\omega\in \Bbbk$ such that $\omega^n=1$. Hence $\omega^2\sigma(a^2)$ is not a primitive $n$th root of 1. As a result $K(8n,\zeta)$ is not minimal by Corollary \ref{coro4.2}. If $n=2$, then $\sigma(a^2)=1$. Let $\omega\in \Bbbk$ such that $\omega^2=1$. Thus $\omega^2 \sigma(a^2)=1$ which is not a primitive $2$th root of 1. Therefore $K(16,\zeta)$ is not minimal by applying Corollary \ref{coro4.2} again.
\end{proof}

As an application of Theorem \ref{thm4.1}, we use the following example to illustrate our
results.

\begin{example}
\emph{
Let $K_8$ as before in Example \ref{ex2.1.5}, and let $\tilde{a}=a,\tilde{b}=ab$, then we get that $G=\langle \tilde{a},\tilde{b}|\tilde{a}^2=\tilde{b}^2=1,\tilde{a}\tilde{b}=\tilde{b}\tilde{a}\rangle$ and $\tilde{a}\triangleleft x=\tilde{a}\tilde{b},\;\tilde{b}\triangleleft x=\tilde{b}$. This implies that $K_8$ belongs to $K(8n,\sigma,\tau)$ and such that $\eta(\tilde{a},\tilde{b})=-1$. It can be seen that $\sigma(\tilde{a}^2)=1$ and $n=1$. Since $\sigma(\tilde{a}^2)=1$,  $K_8$ is minimal by Corollary \ref{coro4.2}. Moreover, all minimal quasitriangular structures on it can be given by
 $\{R_{\alpha,\beta}\;|\;\alpha,\beta\in \Bbbk\;\text{such that}\; \alpha^4=-1, \alpha\beta=1\}$.
Combined with Proposition \ref{trivial}, we can get that all universal $\mathcal{R}$-matrices on $K_8$ by   $\{R_{\alpha,\beta}\;|\;\alpha,\beta\in \Bbbk\;\text{such that}\; \alpha^4=-1, \alpha\beta=1\}\; \cup \;\{\sum_{1\leq i,j,k,l \leq 2}\gamma^{ik+jl}\delta^{il+jk}(-1)^{jk} e_{a^i b^j} \otimes e_{a^k b^l}\;|\; \gamma,\delta \in \Bbbk \;\text{such that}\; \gamma^2=\delta^2=1\}$. This result is the same as the result in \cite[Lemma 5.4]{W}.}
\end{example}
For completeness, we give the following result for the case $\eta(a,b)=1$.
\begin{proposition}
Let $K(8n,\sigma,\tau)$ as before. If $\eta(a,b)=1$, then $K(8n,\sigma,\tau)$ is not minimal.
\end{proposition}
\begin{proof}
Let $R'_{\alpha,\beta}$ as before in Proposition \ref{propo4.4}, then we will show that $R'_{\alpha,\beta}$ is not a minimal quasitriangular structure on $K(8n,\sigma,\tau)$ and therefore we complete the proof. Similar to the proof of Lemma \ref{lemm4.7}, we know that $l(X_t)=r(X_t)$ and $l(E_t)=r(E_{t\triangleleft x})$ for $t\in T$ in this case. And hence $R'_{\alpha,\beta}$ such that the condition of Lemma \ref{lem3.3}. We claim that the $w^1$ of $R'_{\alpha,\beta}$ is not a non-degenerated matrix. To show this, let us denote a matrix $A^{1}$ as follows
\begin{gather*}
A^1:=(w^1(s_i,s_j))_{1\leq i,j \leq n},
\end{gather*}
where $s_i=a^{2i-2},$ $s_j=a^{2j-2}$. Assume $1\leq i,j \leq n$ and let $s_{n+i}=a^{2i-2}b$, then we have
\begin{align*}
w^1(s_i,s_{n+j})&=w^1(a^{2i-2},a^{2j-2}b)\\
&=(\lambda_{2j-2,1})^{2i-2} (\alpha \beta)^{2(i-1)(j-1)}[\sigma(a^{2j-2})]^{i-1}\\
&=w^1(a^{2i-2},a^{2j-2})\\
&=w^1(s_i,s_j)
\end{align*}
and
\begin{align*}
w^1(s_{n+i},s_{j})&=w^1(a^{2i-2}b,a^{2j-2})\\
&=(\lambda_{2j-2,1})^{2i-2} (\alpha \beta)^{2(i-1)(j-1)}[\sigma(a^{2j-2})]^{i-1}\\
&=w^1(a^{2i-2},a^{2j-2})\\
&=w^1(s_i,s_j)
\end{align*}
and
\begin{align*}
w^1(s_{n+i},s_{n+j})&=w^1(a^{2i-2}b,a^{2j-2}b)\\
&=(\lambda_{2j-2,1})^{2i-2} (\alpha \beta)^{2(i-1)(j-1)}[\sigma(a^{2j-2})]^{i-1}\\
&=w^1(a^{2i-2},a^{2j-2})\\
&=w^1(s_i,s_j).
\end{align*}
Therefore we have
\begin{gather*}
(w^1(s_i,s_j))_{1\leq i,j \leq 2n} = \begin{pmatrix} A^1 & A^1 \\ A^1 & A^1 \end{pmatrix} \sim \begin{pmatrix} A^1 & A^1 \\ 0 & 0 \end{pmatrix}.
\end{gather*}
  Thus $w^1$ is non-degenerated and we know that $R'_{\alpha,\beta}$ is not minimal quasitriangular structure on $K(8n,\sigma,\tau)$ by Lemma \ref{lem3.3}.
\end{proof}

\end{document}